\documentclass{preprint} 
\usepackage{amsmath}
\usepackage{times} 
\usepackage{mhequ} 
\usepackage{mhfig} 
\usepackage{ScriptFonts} 
\usepackage{Symbols}
\usepackage{mathrsfs}
\usepackage{mhenvs}

\makeatletter
\newtheorem{assumption}{Assumption} 
\def\SetAssumptionCounter#1{\def\asscount{#1}\setcounter{assumption}{0}}
\def\asscount{A}

\makeatother 

\newtheorem{metatheorem}[lemma]{Meta-Theorem}

\newcommand{\E}{\mathbf{E}}
\renewcommand{\rho}{\orho}
\def\ass#1{\ref{ass:basic}.\ref{ass:#1}}

\def\hf{{\textstyle{1\over 2}}}
\def\Osc{\mathrm{Osc}_W}

\def\Lip{\mathrm{Lip}} 

\def\Wien{\mathrm{Wien}} 
\def\WL{W_{\! L}}
\def\ad{\mathrm{ad}}
\def\Sup{\mathrm{Sup}} 
\def\TV{{\mathrm{TV}}}

\renewcommand{\div}{\mop{div}}  
\newcommand{\curl}{\mop{curl}} 
\newcommand{\Ric}{\mop{Ric}}

\renewcommand{\P}{\mathbf{P}}
\newcommand{\EE}{\mathbf{E}} 
\newcommand{\cP}{\mathcal{P}}

\newcommand{\cF}{\mathcal{F}}

\newcommand{\ccdot}{\,\cdot\,}
\renewcommand{\d}{\partial}

\newcommand{\ip}[2]{{\langle#1,#2\rangle}}

\newcommand{\Range}{\text{Range}} 

\newcommand{\SPAN}{\text{span}} 
\newcommand{\Poly}{\mop{{Poly}}}

\newcommand{\U}{U}

\newcommand{\Norm}[1]{\vert\!\vert\!\vert #1 \vert\!\vert\!\vert}
\newcommand{\norm}[1]{\| #1 \|_{\mathrm{HS}}}
\def\II{{I_\delta}} \def\/{|\!|\!|}

\newcommand{\RM}{\CR}
\newcommand{\MM}{\CM}
\newcommand{\DM}{\CD}
\newcommand{\AM}{\CA}
\newcommand{\DF}{\mathrm D}
\renewcommand{\AA}{\mathrm A}
\newcommand{\AAt}{\widetilde{\AA}}

\newcommand{\CamM}{{\CC\!\CM}}
\newcommand{\Cam}{{\CC\!\CM'}}
\begin{document}

\date{} \title{A Theory of Hypoellipticity and Unique Ergodicity
  \\[2mm] for Semilinear 
  Stochastic PDEs} \author{Martin~Hairer$^1$\thanks{Work supported by
    an EPSRC Advanced Research Fellowship, grant number EP/D071593/1},
  Jonathan~ C.~Mattingly$^2$\thanks{Work supported by an NSF PECASE
    awawrd (DMS-0449910) and a Sloan foundation fellowship.}}
\institute{$^1$Mathematics Institute, The University of Warwick, CV4
  7AL, UK
  \\ \email{M.Hairer@Warwick.ac.uk}\\
  $^2$Mathematics Department, Center of Theoretical and Mathematical
  Science, and Center of Nonlinear and Complex Systems, Duke
  University, Durham NC, USA \\
  \email{jonm@math.duke.edu}} \titleindent=0.65cm

\maketitle
\thispagestyle{empty}

\begin{abstract}\parindent1em
  We present a theory of hypoellipticity and unique ergodicity for
  semilinear parabolic stochastic PDEs with ``polynomial''
  nonlinearities and additive noise, considered as abstract evolution
  equations in some Hilbert space.  It is shown that if H\"ormander's
  bracket condition holds at every point of this Hilbert space, then a
  lower bound on the Malliavin covariance operator $\CM_t$ can be
  obtained. Informally, this bound can be read as ``Fix any
  finite-dimensional projection $\Pi$ on a subspace of sufficiently
  regular functions. Then the eigenfunctions of $\CM_t$ with small
  eigenvalues have only a very small component in the image of
  $\Pi$.''
  
  We also show how to use a priori bounds on the solutions to the
  equation to obtain good control on the dependency of the bounds on
  the Malliavin matrix on the initial condition. These bounds are
  sufficient in many cases to obtain the asymptotic strong Feller
  property introduced in \cite{HairerMattingly06AOM}.

  One of the main novel technical tools is an almost sure bound from
  below on the size of ``Wiener polynomials,'' where the coefficients
  are possibly non-adapted stochastic processes satisfying a Lipschitz
  condition. By exploiting the polynomial structure of the equations,
  this result can be used to replace Norris' lemma, which is
  unavailable in the present context.

  We conclude by showing that the two-dimensional stochastic
  Navier-Stokes equations and a large class of reaction-diffusion
  equations fit the framework of our theory.
\end{abstract}

\tableofcontents

\section{Introduction}
\label{sec:motiv-disc}

The overarching goal of this article is to prove the unique
ergodicity of a class of nonlinear stochastic partial differential
equations (SPDEs) driven by a finite number of Wiener processes. The
present greatly extends the articles
\cite{MatPar06:1742,HairerMattingly06AOM,Yuri} allowing one to
consider general polynomial nonlinearities and more general
forcing. To the best of our knowledge, this is the first infinite-dimensional
generalization of H\"ormander's ``sum of squares'' hypoellipticity
theorem for a general class of  parabolic SPDEs.   Our goal is not to present any particularly compelling
examples from the applied perspective, but rather give a sufficiently
general framework which can be applied in many settings. At the end,
we do give some examples to serve as roadmaps for the application of
the results in this article.  In this section, we give an overview of
the setting and the results to come later without descending into all
of the technical assumptions required to make everything precise. This
imprecision will be rectified starting with
Section~\ref{sec:PDEbounds} where the setting and basic assumptions
will be detailed.

In this article we will investigate  non-linear equations of the form
\begin{equ}[e:SPDEintro]
  \partial_t u(x,t) + L u(x,t) = N\big(u\big)(x,t) + \sum_{k=1}^d
  g_k(x) \dot W_k(t)\,.
\end{equ}
Here $L$ will be some positive selfadjoint operator. 
Typical examples arising in applications are $L=-\Delta$ or
$L=\Delta^2$. $N$ will be assumed to be a ``polynomial'' nonlinearity
in the sense that $N(u) = \sum_{k=1}^m N_k(u)$, where $N_k$ is
$k$-multilinear.  Examples of admissible nonlinearities are the
Navier-Stokes nonlinearity $(u\cdot \nabla)u$ or a reaction term such
as $u-u^3$. The $g_k$ are a collection of smooth, time independent
functions which dictate the ``directions'' in which the randomness is
injected. The $\{ \dot W_k : k=1,\ldots,d\}$ are a collection of
mutually independent one-dimensional white noises which are understood
as the formal derivatives of independent Wiener processes through the
It\^o calculus. We assume that the possible loss of regularity due to
the nonlinearity is controlled by the smoothing properties of the
semigroup generated by $L$. See Assumption~\ref{ass:basic} below for a
precise meaning.

On one hand, our concentration on a finite number of driving Wiener
processes avoids technical difficulties generated by spatially rough
solutions since $W(x,t)= \sum g_k(x) W_k(t)$ has the same regularity
in $x$ as the $g_k$ which we take to be relatively smooth. On the
other hand, the fact that $W$ contains only a finite number of Wiener
processes means that our dynamic is very far from being uniformly
elliptic in any sense since for fixed $t$, $u(\ccdot,t)$ is an
infinite-dimensional random variable and the noise acts only onto a
finite number of degrees of freedom.  To prove an ergodic theorem, we
must understand how the randomness injected by $W$ in the directions
$\{g_k : k=1,\ldots,d\}$ spreads through the infinite dimensional
phase space. To do this, we prove the non-degeneracy of the Malliavin
covariance matrix under an assumption that the linear span of the
successive Lie brackets\footnote{Recall that, when it is defined, the
  Lie bracket $[G,H](u)=(\DF H)(u)G(u) -(\DF G)(u)H(u)$ for two
  functions $G,H$ from the ambient Hilbert space $\CH$ to itself. Here
  $\DF$ is the Fr\'echet derivative.} of vector fields associated to
$N$ and the $g_k$ is dense in the ambient (Hilbert) space at each
point. This is very reminiscent of the condition in the ``weak''
version of H\"ormander's ``sum of squares'' theorem. It ensures that
the randomness spreads to a dense set of direction despite being
injected in only a finite number of directions. This is possible since
although the randomness is injected in a finite number of directions
it is injected over the entire interval of time from zero to the
current time. The conditions which ensure the spread of randomness is
closely related to Chow's theorem and controllability, open or
solid. As such Section~\ref{sec:mallSpec} is related to recent work on
controllability of projections of PDEs studied in
\cite{AgrachevSarychev05,agrachev_sarychev_2007} and results proving
the existence of densities for projections in \cite{AKSS07} which build on
these ideas. However, these results do not seem to be sufficient to
prove an ergodic result which is the principal aim of this work. One
seems to need quantitative control of the spectrum of the Malliavin
matrix (or the Gramian matrix in control theory terms).

In finite dimensions, bounds on the norm of the inverse of the
Malliavin matrix are the critical ingredient in proving ergodic
theorems for diffusions which are only hypoelliptic rather than
uniformly elliptic. This then shows that the system has a smooth
density with respect to Lebesgue measure. In infinite dimensions,
there is no measure which plays the ``universal'' role of Lebesgue
measure. One must therefore pass through a different set of ideas.
Furthermore, it is not so obvious how to generalise the notion of the
`inverse' of the Malliavin matrix. In finite dimension, a linear map
has dense range if and only if it admits a bounded inverse. In
infinite dimensions, these two notions are very far from equivalent
and, while it is possible in some cases to show that the Malliavin
matrix has dense range, it is hardly ever possible in a hypoelliptic
setting to show that it is invertible, or at least to characterise its
range in a satisfactory manner (See \cite{MatSuiVand07} for a linear
example in which it is possible).

The important fact which must be established is that nearby points act
similarly from a ``measure theoretic perspective.'' One classical way
to make this precise is to prove that the Markov process in question
has the strong Feller property. For a continuous time Markov process
this is equivalent to proving that the transition probabilities are
continuous in the total variation norm.  While this concept is useful
in finite dimensions, it is much less useful in infinite
dimensions. In particular, there are many natural infinite dimensional
Markov processes whose transition probabilities do \textit{not}
converge in total variation to the system's unique invariant
measure. (See examples 3.14 and 3.15 from \cite{HairerMattingly06AOM}
for more discussion of this point.) In these settings, this fact also
precludes the use of ``minorization'' conditions such as $\inf_{x \in
  \CC}\CP_t(x,\ccdot) \geq c \nu(\ccdot)$ for some fixed probability
measure $\nu$ and ``small set'' $\CC$. (see
\cite{b:MeTw93,b:GoldysMaslowski06} for more and examples were this
can be used.)

\subsection{Ergodicity in infinite dimensions and main result} 

In \cite{HairerMattingly06AOM}, the authors introduced the concept of
an \textit{Asymptotic Strong Feller} diffusion. Loosely speaking, it
ensures that transition probabilities are uniformly continuous in a
sequence of 1-Wasserstein distances which converge to the total
variation distance as time progresses. For the precise definitions, we
refer the reader to \cite{HairerMattingly06AOM}. For our present
purpose, it is sufficient to recall the following proposition:
\begin{proposition}[Proposition 3.12 from
  \cite{HairerMattingly06AOM}] \label{l:asfLipFunctions} Let $t_n$ and
  $\delta_n$ be two positive sequences with $\{t_n\}$ non-decrea\-sing
  and $\{\delta_n\}$ converging to zero.  A semigroup $\CP_t$ on a
  Hilbert space $\CH$ is asymptotically strong Feller if, for all
  $\phi:\CH\rightarrow \R$ with $\|\phi\|_\infty$ and $\|\DF
  \phi\|_\infty$ finite one has
  \begin{equ}[e:SF]
    \|\DF \CP_{t_n}\phi(u)\| \leq C(\|u\|) \bigl( \|\phi\|_\infty +
    \delta_n \|\DF \phi\|_\infty\bigr)
  \end{equ}
for all $n$ and $u \in \CH$, where $C:\R_+ \rightarrow \R$ is a fixed
non-decreasing function.
\end{proposition}

The importance of the asymptotic strong Feller property is given by
the following result which states that in this case, any two distinct
ergodic invariant measures must have disjoint topological
supports. Recalling that $u$ belongs to the support of a measure $\mu$
(denoted $\supp(\mu)$) if $\mu(B_\delta(u)) >0$ for every $\delta
>0$\footnote{Here $B_\delta(u)=\{ v : \|u-v\|< \delta \}$}, we have:
\begin{theorem}[Theorem 3.16 from \cite{HairerMattingly06AOM}]
  \label{thm:suppDisjoint} Let $\CP_t$ be a Markov semigroup on a
  Polish space $\CX$ admitting two distinct ergodic invariant measures
  $\mu$ and $\nu$. If $\CP_t$ has the asymptotic strong Feller
  property, then $\supp(\mu) \cap \supp(\nu)$ is empty.
\end{theorem}
To better understand how the asymptotic strong Feller property can be
used to connect topological properties and ergodic properties of
$\CP_t$, we introduce the following form of topological
irreducibility.
\begin{definition}
  We say that a Markov semigroup $\CP_t$ is \textit{weakly
    topologically irreducible} if for all $u_1, u_2 \in \CH$ there
  exists a $v \in \CH$ so that for any $A$ open set containing $v$
  there exists $t_1, t_2 > 0$ with $\CP_{t_i}(u_i,A) >0$.
\end{definition}
Also recall that $\CP_t$ is said to be \textit{Feller} if $\CP_t\phi$
is continuous whenever $\phi$ is bounded and continuous. We then have
the following corollary to Theorem~\ref{thm:suppDisjoint} whose proof
is given in Section~\ref{sec:ergodicResults}.
\begin{corollary}\label{cor:weakIrr}
  Any Markov semigroup $\CP_t$ on Polish space which is Feller, weakly
  topologically irreducible and asymptotically strong Feller admits at
  most one invariant probability measure.
\end{corollary}

The discussion of this section shows that unique ergodicity can be
obtained for a Markov semigroup by showing that:
\begin{claim}
\item[1.] It satisfies the asymptotic strong Feller property.
\item[2.] There exists an ``accessible point'' which must belong to
  the topological support of every invariant probability measure.
\end{claim}
It turns out that if one furthermore has some control on the speed at
which solution return to bounded regions of phase space, one can prove
the existence of spectral gaps in weighted Wasserstein-1 metrics
\cite{HaiMat08:??,HaiMatScheu,HaiMaj10}.

The present article will mainly concentrate on the first point. This
is because, by analogy with the finite-dimensional case, one can hope
to find a clean and easy way to verify condition along the lines of
H\"ormander's bracket condition that ensures a regularisation property
like the asymptotic strong Feller property.  Concerning the
accessibility of points however, although one can usually use the
Stroock-Varadhan support theorem to translate this into a
deterministic question of approximate controllability, it can be a
very hard problem even in finite dimensions.  While geometric
control theory can give qualitative information about the set of
reachable points \cite{Jurd97,AgrSac04}, the verification of the
existence of accessible points seems to rely in general on \textit{ad
  hoc} considerations, even in apparently simple finite-dimensional
problems.  We will however verify in Section~\ref{sec:accessibleGL}
below that for the stochastic Ginzburg-Landau equation there exist
accessible points under very weak conditions on the forcing.
 
With this in mind, the aim of this article is to prove the following type of
`meta-theorem':
 
\begin{metatheorem}\label{theo:main}
  Consider the setting of \eref{e:SPDEintro} on some Hilbert space
  $\CH$ and define a sequence of subsets of $\CH$ recursively by
  $\AA_0 = \{g_j\,:\, j=1,\ldots,d\}$ and
\begin{equ}
  \AA_{k+1} = \AA_k \cup \{N_m(h_1,\ldots,h_m)\,:\, h_j \in \AA_k\}\;.
\end{equ}
Under additional stability and regularity assumptions, if
 the linear span of $\AA_\infty \eqdef \bigcup_{n > 0} \AA_n$ is dense
in $\CH$, then the Markov semigroup $\CP_t$ associated to
\eref{e:SPDEintro} has the asymptotic strong Feller property.
\end{metatheorem}

The precise formulation of Meta-Theorem~\ref{theo:main} will be given in
Theorem~\ref{theo:general} below, which in turn will be a consequence
of the more general results given in Theorems~\ref{thm:asfEstimate} and \ref{theo:Malliavin}. Note
that our general results are slightly stronger than what is suggested in
Meta-Theorem~\ref{theo:main} since it also allows to consider arbitrary
``non-constant'' Lie brackets between the driving noises and the
drift, see \eref{eq:bracket} below. As further discussed in
Section~\ref{subsec:SatHor} or \ref{sec:poly}, $N_m(h_1,\ldots,h_m)$
is proportional to $\DF_{h_1}\cdots\DF_{h_m}N(u)$ where $\DF_h$ is the
Fr\'echet derivative in the direction $h$. In turn, this is equal to
the successive Lie-brackets of $N$ with the constant vector fields 
in the directions $h_1$ to $h_m$.

Under the same structural assumtpions as Meta-Theorem~\ref{theo:main}, the existence
of densities for the finite dimensional projections of
$\cP_t(x,\ccdot)$ was proven in \cite{Yuri}. The smoothness of these
densities was also discussed in \cite{Yuri}, but unfortunately there were
two errors in the proof of that article. While the arguments presented in the present article
are close in sprit to those in \cite{Yuri}, they diverge at the technical level. Our results
on the smoothness of densities
will be given in  Sections~\ref{sec:mallSpec} and \ref{sec:Wiener}.

The remainder of this section is
devoted to a short discussion of the main techniques used in the proof
of such a result and in particular on how to obtain a bound of the
type \ref{e:SF} for a parabolic stochastic PDE.

\subsection{A roadmap for the impatient}

Readers eager to get to the heart of this article but understandably
reluctant to dig into too many technicalities may want finish reading
Section 1, then jump directly to Section~\ref{sec:generalSmoothing}
and read up to the end of Section~\ref{sec:choiceh} to get a good
idea of how \eqref{e:SF} is obtained from bounds on the Malliavin
matrix.  Then they may want to go to the beginning of
Section~\ref{sec:mallSpec} and read to the end of
Section~\ref{sec:proofMall} to see how these bounds can be obtained.

\subsection{How to obtain a smoothing estimate} 

A more technical overview of the techniques will be given in Section
\ref{sec:general-framework} below. In a nutshell, our aim is to
generalise the arguments from \cite{HairerMattingly06AOM} and the type
of Malliavin calculus estimates developed in
\cite{MatPar06:1742} to a large class of semilinear parabolic SPDEs
with polynomial nonlinearities.  Both previous works relied on the
particular structure of the Navier-Stokes equations.  The technique of
proof can be interpreted as an ``infinitesimal'' version of techniques
developed in \cite{EMS,KS00} and extended in
\cite{BC,MY,MatNS,HExp02,MatMemory} combined with detailed lower
bounds on the Malliavin covariance matrix of the solution.

In \cite{EMS} the idea was the following: take two distinct initial
conditions $u_0$ and $u'_0$ for \eref{e:SPDEintro} and a realisation
$W$ for the driving noise.  Try then to find a shift $v$ belonging to
the Cameron-Martin space of the driving process and such that $\|u(t)
- u'(t)\|\to 0$ as $t \rightarrow\infty$, where $u'$ is the solution to
\eref{e:SPDEintro} driven by the shifted noise $W' = W + v$.
Girsanov's theorem then ensures that the two initial conditions induce
equivalent measures on the infinite future. This in turn implies the
unique ergodicity of the system. (See also \cite{saintFlourJCM08} for
more details.)

The idea advocated in \cite{HairerMattingly06AOM} is to consider an
infinitesimal version of this construction.  Fix again an initial
condition $u_0$ and Wiener trajectory $W$ but consider now an
\textit{infinitesimal} perturbation $\xi$ to the initial condition
instead of considering a second initial condition at distance
$\CO(1)$.  This produces an infinitesimal variation in the solution
$u_t$ given by its Fr\'echet derivative $\DF_\xi u_t$ with respect to
$u_0$. Similarly to before, one can then consider the ``control
problem'' of finding an \textit{infinitesimal} variation of the Wiener
process in a direction $h$ from the Cameron-Martin space which, for
large times $t$, compensates the effect of the variation $\xi$. Since
the effect on $u_t$ of an infinitesimal variation in the Wiener
process is given by the Malliavin derivative of $u_t$ in the direction
$h$, denoted by $\DM_h u_t$, the problem in this setting is to find an
$h(\xi,W) \in L^2([0,\infty],\R^d)$ with
\begin{equation}
  \label{eq:variatonToZero}
\EE\|\DF_\xi u_t - \DM_h u_t\| \rightarrow 0\text{ as }t \rightarrow  
\infty\;, 
\end{equation}
and such that the expected ``cost'' of $h_t$ is finite. Here, the
Malliavin derivative $\DM_h u_t$ is given by the derivative in $\eps$
at $\eps = 0$ of $u_t(W+ \eps v)$, with $v(t) = \int_0^t h(s)\,ds$.
If $h$ is adapted to the filtration generated by $W$, then the
expected cost is simply $\int_0^\infty \EE\|h_s\|^2 ds$. If it is not
adapted, one must estimate directly $\limsup \EE\| \int_0^t h_s
dW_s\|$ where the integral is a Skorokhod integral.
As will be explained in detail in Section \ref{sec:general-framework},
once one establishes \eqref{eq:variatonToZero} with a finite expected
cost $h$, the crucial estimate given in \eref{e:SF} (used to prove the
asymptotic strong Feller property) follows by a fairly general procedure.

As this discussion makes clear, one of our main tasks will be to
construct a shift $h$ having the property \eref{eq:variatonToZero}. We
will distinguish three cases of increasing generality (and technical
difficulty). In the first case, which will be referred to as
\textit{strongly contracting} (see
Section~\ref{sec:stronglyDissipative}), the linearised dynamics
contracts pathwise without modification (all Lyapunov exponents are
negative). Hence $h$ can be taken to be identically zero.  The next
level of complication comes when the system possesses a number of
directions which are unstable on average. The simplest way to deal
with this assumption is to assume that the complement of the span of
the forced directions (the $g_k$'s) is contracting on average. This
was the case in \cite{EMS,KS00,BC,MY,MatNS,HExp02,MatMemory}. We refer to this
as the ``essentially elliptic'' setting since the directions essential
to determine the system's long time behavior, the unstable directions,
are directly forced. This is a reflection of the maxim in dynamical
systems that the long time behavior is determined by the behavior in
the unstable directions. Since the noise affects all of these
directions, it is not surprising that the system is uniquely ergodic,
see Section~4.5 of \cite{HairerMattingly06AOM} for more details.

The last case (i.e.\ when the set of forced directions does not
contain all of the unstable directions) is the main concern of the
present paper.  In this setting, we study the interaction between the
drift and the forced directions to understand precisely how randomness
spreads to the system. The condition ensuring that one can gain
sufficient control over the unstable directions, requires that the
$g_k$ together with a collection of Lie brackets (or commutators) of
the form
\begin{equation}\label{eq:bracket}
  [F,g_k],\, [[F,g_k],g_j],\, [[F,g_k],F],\,  [[[F,g_k],g_j],g_l],\cdots  
\end{equation}
span all of the unstable direction. This condition will be described
more precisely in Section \ref{sec:Hormander} below. In finite
dimensions, when this collection of Lie brackets spans the entire
tangent space at every point, the system is said to satisfy the ``weak
H\"ormander'' condition. When this assumption holds for the unstable
directions (along with some additional technical assumptions), we can
ensure that the noise spreads sufficiently to the unstable directions
to find a $h$ capable of counteracting the expansion in the unstable
directions and allowing one to prove \eqref{eq:variatonToZero} with a
cost whose expectation is finite.

We will see however that the control $h$ used will not be adapted to
the filtration generated by the increments of the driving Wiener
process, thus causing a number of technical difficulties. This stems
from the seemingly fundamental fact that because we need some of the
``bracketed directions'' \eref{eq:bracket} in order to control the
dynamic, we need to work on a time scale longer than the instantaneous
one. In the ``essentially elliptic'' setting, on the other hand, we
were able to work instantaneously and hence obtain an adapted control
$h$ and avoid this technicality.

\subsection{The role of the Malliavin matrix}

Since the Malliavin calculus was developed in the 1970's and 1980's
mainly to give a probabilistic proof of H\"ormander's ``sum of
squares'' theorem under the type of bracket conditions we consider, it
is not surprising that the Malliavin matrix $\CM_t = \DM u_t \DM
u_t^*$ plays a major role in the construction of the variation $h$ in
the ``weak H\"ormander'' setting. A rapid introduction to Malliavin
calculus in our setting is given in Section \ref{sec:Malliavin}.  In
finite dimensions, the key to the proof of existence and smoothness of
densities is the finiteness of moments of the inverse of the Malliavin
matrix. This estimate encapsulates the fact the noise effects all of
the directions with a controllable cost.  In infinite dimensions while
it is possible to prove that the Malliavin matrix is almost surely
non-degenerate it seems very difficult to characterise its
range. (With the exception of the linear case \cite{ZDP}. See also
\cite{DaPElwZab95:35,FM,CerrRD,EH3} for situations where the Malliavin
matrix can be shown to be invertible on the range of the Jacobian.)
However, in light of the preceding section, it is not surprising that
we essentially only need the invertibility of the Malliavin matrix on
the space spanned by the unstable directions, which is finite
dimensional in all of our examples. More precisely, we need
information about the likelihood of eigenvectors with sizable
projections in the unstable directions to have small
eigenvalues. Given a projection $\Pi$ whose range includes the
unstable directions we will show that the Malliavin matrix $\MM_t$
satisfies an estimate of the form
\begin{equ}[e:boundMM]
  \P\Big( \inf_{\substack{\phi \in \CH\\ \|\Pi \phi\| \geq \frac12
      \|\phi\|}} \scal{\MM_t \phi,\phi} > \eps \|\phi\|^2\Big) =
  o(\eps^p)
\end{equ}
for all $p\geq 1$. Heuristically, this means we have control of the
probabilistic cost to create motion in all of the directions in the
range of $\Pi$ without causing a too large effect in the complementary
directions.  We will pair such an estimate with the assumption that the
remaining directions are stable in that the Jacobian (the linearization
of the SPDE about the trajectory $u_t$) satisfies a contractive
estimate for the directions perpendicular to the range of
$\Pi$. Together, these assumptions will let us build an infinitesimal
Wiener shift $h$ which approximately compensates for the component of
the infinitesimal shift caused by the variation in the initial
condition in the unstable directions. Once the variation in the
unstable directions have been decreased, the assumed contraction in
the stable directions will ensure that the variation in the stable
directions will also decrease until it is commiserate in size with the
remaining variation in the unstable directions. Iterating this
argument we can drive the variation to zero.

Note that one feature of the bound \eref{e:boundMM} is that all the
norms and scalar products appearing there are the same. This is a
strengthening of the result from \cite{MatPar06:1742} which fixes an
error in \cite{HairerMattingly06AOM}, see Section~\ref{sec:mallSpec}
for more details.

The basic structure of the sections on Malliavin calculus follows the
presentation in \cite{Yuri} which built on the ideas and techniques
from \cite{MatPar06:1742,Oco88:288}. As in all three of these works,
as well as the present article, the time reversed adjoint
linearization is used to develop an alternative representation of the
Malliavin Covariance matrix. In \cite{Oco88:288}, only the case of
linear drift and linear multiplicative noise was considered. In
\cite{MatPar06:1742}, a nonlinear equation with a quadratic
nonlinearity and additive noise was considered. In \cite{Yuri}, the
structure was abstracted and generalized so that it was amenable to
general polynomial nonlinearities. We follow that structure and basic
line of argument here while strengthening the estimates and correcting
some important errors.

Most existing bounds on the inverse of the Malliavin matrix in a
hypoelliptic situation make use of some version of Norris' lemma
\cite{KSAMI,KSAMII,Norr,MatPar06:1742,BauHai07:373}. In its form taken
from \cite{Norr}, it states that if a semimartingale $Z(t)$ is small
and one has some control on the roughness of both its bounded
variation part $A(t)$ and its quadratic variation process $Q(t)$, then
both $A$ and $Q$ taken separately must be small. While the versions of
Norris' lemma given in \cite{MatPar06:1742,Yuri,BauHai07:373} are not
precisely of this form (in both cases, one cannot reduce the problem
to semimartingales, either because of the infinite-dimensionality of
the problem or because one considers SDEs driven by processes that are
not Wiener processes), they have the same flavour in that they state
that if a process is composed of a ``regular'' part and an
``irregular'' part, then these two parts cannot cancel each
other. This harkens back to the more explicit estimates based on
estimates of modulus of continuity found in
\cite{KusuokaStroockII,StroockSaintFlour}. The replacement for Norris'
lemma used in the present work covers the case where one is given a
finite collection of Wiener process $W_j$ and a collection of not
necessarily adapted Lipschitz continuous processes $A_\alpha(t)$ (for
$\alpha$ a multi-index) and considers the process
\begin{equ}
  Z(t) = A_{\emptyset}(t) + \sum_{\ell=1}^M \sum_{|\alpha| = \ell}
  A_\alpha(t) W_{\alpha_1}(t)\cdots W_{\alpha_\ell}(t)\;.
\end{equ}
It then states that if $Z$ is small, this implies that all of the
$A_\alpha$'s with $|\alpha| \le M$ are small.  For a precise
formulation, see Section~\ref{sec:Wiener} below.  It is in order to be
able to use this result that we are restricted to equations with
polynomial nonlinearities. This result on Wiener polynomials is a
descendant of the result proven in \cite{MatPar06:1742} for
polynomials of degree one. In \cite{Yuri}, a result for general
Wiener polynomials was also proven. Is was show there that if $Z(t) = 0$ for
$t \in [0,T]$ then then $A_\alpha(t) =0$ for $t \in [0,T]$. This was
used to prove the existence of a density for the finite dimensional
projections of the transition semigroup.  In the same article, the
same quantitative version of this result as proven in the present article was
claimed. Unfortunately, there was a error in the proof. Nonetheless
the techniques used here are built on  and refine those developed in
\cite{Yuri}.

\subsection{Satisfying the H\"ormander-like assumption}
\label{subsec:SatHor}
At first glance the condition that the collection of functions given
in \eqref{eq:bracket} are dense in our state space may seem hopelessly
strong.  However, we will see that it is often not difficult to
ensure. Recall that the nonlinearity $N$ is a polynomial of order $m$,
and hence, it has a leading order part which is $m$-homogeneous. We
can view this leading order part as a symmetric $m$-linear map which
we will denote by $N_m$. Then, at least formally, the Lie bracket of
$N$ with $m$ constant vector fields is proportional to $N_m$,
evaluated at the constant vector fields, that is $N_m(h_1,\cdots,h_m)
\propto [\cdots[[F,h_1],\cdots],h_m]$, which is again a constant
vector field. While the collection of vector fields generated by
brackets of this form are only a subset of the possible brackets, it
is often sufficient to obtain a set of dense vector fields. For
example, if $N(u)=u-u^3$ then $N_3(v_1,v_2,v_3)=v_1v_2v_3$ and if the
forced directions $\{g_1, \cdots,g_d\}$ are $\CC^\infty$ then
$N_3(h_1,h_2,h_3)\in \CC^\infty$ for $h_i \in \{g_1, \cdots,g_d\}$.
As observed in \cite{Yuri}, to obtain a simple sufficient criteria for
the brackets to be dense, suppose that $\Lambda \subset \CC^\infty$ is
a finite set of functions that generates, as a multiplicative algebra,
a dense subset of the phase space. Then, if the forced modes
$\AA_0=\{g_1, \cdots,g_d\}$ contain the set $\{ h, h\bar h : h, \bar h
\in \Lambda\}$, the set $\AA_\infty$ constructed as in
Meta-Theorem~\ref{theo:main} will span a dense subset of phase space.

\subsection{Probabilistic and dynamical view of smoothing}
\label{sec:prob-dynam-view}
Implicit in \eqref{eq:variatonToZero} is the ``transfer of variation''
from the initial condition to the Wiener path. This is the heart of
``probabilistic smoothing'' and the source of ergodicity when it is
fundamentally probabilistic in nature. The unique ergodicity of a
dynamical system is equivalent to the fact that it ``forgets its
initial condition'' with time. The two terms appearing on the
right-hand side of \eqref{e:SF} represent two different sources of
this loss of memory. The first is due to the randomness entering the
system. This causes nearby points to end up at the same point at a
later time because they are following different noise
realisations. The fact that different stochastic trajectories can
arrive at the same point and hence lead to a loss of information is
the hallmark of diffusions and unique ergodicity due to
randomness. From the coupling point of view, since different
realizations lead to the same point yet start at different initial
conditions, one can couple in finite time.

The second term in
\eqref{e:SF} is due to ``dynamical smoothing'' and is one of the sources
of unique ergodicity in deterministic contractive dynamical systems.  If two
trajectories converge towards each other over time then the level of
precision needed to determine which initial condition corresponds to
which trajectory also increases with time. This is another type of
information loss and equally leads to unique ergodicity. However,
unlike ``probabilistic smoothing'', the information loss is never
complete at any finite time. Another manifestation of this fact is
that the systems never couples in finite time, only at infinity. In
Section~\ref{sec:stronglyDissipative} about the strongly dissipative
setting, the case of pure dynamical smoothing is considered. In
this case one has \eqref{e:SF} with only the second term present. When
both terms exist, one has a mixture of probabilistic and dynamical
smoothing leading to a loss of information about the initial
condition. In Section~2.2 of \cite{HaiMat08:??} it is shown how
\eqref{e:SF} can be used to construct a coupling in which nearby
initial conditions converge to each other at time infinity. The
current article takes a ``forward in time'' perspective, while
\cite{EMS,MatMemory} pull the initial condition back to minus
infinity. The two points of view are essentially equivalent. One
advantage to moving forward in time is that it makes proving a spectral
gap for the dynamic more natural. We provide such an estimate in
Section~\ref{sec:accessibleGL} for the stochastic Ginzburg-Landau
equation.

\subsection{Structure of the article}

The structure of this article is as follows. In
Section~\ref{sec:ergodicResults}, we give a few abstract ergodic
results both proving the results in the introduction and expanding
upon them.  In Section~\ref{sec:PDEbounds}, we introduce the
functional analytic setup in which our problem will be
formulated. This setup is based on Assumption~\ref{ass:basic} which
ensures that all the operations that will be made later
(differentiation with respect to initial condition, representation for
the Malliavin derivative, etc) are well-behaved. 
Section~\ref{sec:Malliavin} is a follow-up section where we define the Malliavin
matrix and obtain some simple upper bounds on it.
We then introduce
some additional assumptions in Section~\ref{sec:boundsdyn} which
ensure that we have suitable control on the size of the solutions and
on the growth rate of its Jacobian.

In Section~\ref{sec:generalSmoothing}, we obtain the asymptotic strong
Feller property under a partial invertibility assumption on the
Malliavin matrix and some additional partial contractivity assumptions
on the Jacobian. Section~\ref{sec:proofMalliavin} then contains the
proof that assumptions on the Malliavin matrix made in
Section~\ref{sec:generalSmoothing} are justified and can be verified
for a large class of equations under a H\"ormander-type condition. The
main ingredient of this proof, a lower bound on Wiener polynomials, is
proved in Section~\ref{sec:Wiener}. Finally, we conclude in
Section~\ref{sec:examples} with two examples for which our conditions
can be verified. We consider the Navier-Stokes equations on the
two-dimensional sphere and a general reaction-diffusion equation in
three or less dimensions.

\subsection*{Acknowledgements}

{\small
We are indebted to Hakima Bessaih who pushed us to give a clean formulation
of Theorem~\ref{theo:general}.
}

\section{Abstract ergodic results}
\label{sec:ergodicResults}

We now expand upon the abstract ergodic theorems mentioned in the
introduction which build on the asymptotic strong Feller property. We
begin by giving the proof of Corollary~\ref{cor:weakIrr} from the
introduction and then give a slightly different result (but with the
same flavour) which will be useful in the investigation of the
Ginzburg-Landau equation in Section~\ref{sec:accessibleGL}.
Throughout this section, $\cP_t$ will be a Markov semigroup on a
Hilbert space $\CH$ with norm $\|\ccdot\|$.

\begin{proof}[of Corollary~\ref{cor:weakIrr}]
Since $\CP_t$ is Feller, we know that for any $u \in \CH$
  and open set $A$ with $\CP_t(u,A)>0$ there exists an open set $B$
  containing $u$ so that
  \begin{align*}
    \inf_{u \in B} \CP_t(u,A) >0 \,.
  \end{align*}

  Combining this fact with the weak topological irreducibility, we
  deduce that for all $u_1, u_2 \in \CH$ there exists $v \in \CH$ so
  that for any $\epsilon >0$ there exists a $\delta, t_1, t_2 > 0$ with
  \begin{equation}\label{eq:openSetIrrduce}
    \inf_{ z \in B_\delta(u_i)} \CP_{t_i}(z, B_\epsilon(v)) >0 
  \end{equation}
  for $i=1,2$.

  Now assume by contradiction that we can find two distinct invariant
  probability measures $\mu_1$ and $\mu_2$ for $\CP_t$. Since any
  invariant probability measure can be written as a convex combination
  of ergodic measures, we can take them to be ergodic without loss of
  generality.  Picking $u_i \in \supp(\mu_i)$, by assumption there
  exists a $v$ so that for any $\epsilon >0$ there exists $t_1$, $t_2$
  and $\delta>0$ so that \eqref{eq:openSetIrrduce} holds. Since $u_i
  \in \supp(\mu_i)$ we know that $\mu_i(B_\delta(u_i)) >0$ and hence
  \begin{align*}
    \mu_i(B_\epsilon(v))&= \int_\CH \CP_{t_i}(z, B_\epsilon(v))
    \mu_i(dz)
    \geq  \int_{B_\delta(u_i)} \CP_{t_i}(z, B_\epsilon(v)) \mu_i(dz)\\
    &\geq \mu_i(B_\delta(u_i))\inf_{ z \in B_\delta(u_i)} \CP_{t_i}(z,
    B_\epsilon(v)) >0\,.
  \end{align*}
  Since $\epsilon$ was arbitrary, this shows that $v \in
  \supp(\mu_1)\cap \supp(\mu_2)$, which by
  Theorem~\ref{thm:suppDisjoint} gives the required contradiction.
\end{proof}

We now give a more quantitative version of
Theorem~\ref{thm:suppDisjoint}. It shows that if one has access to the
quantitative information embodied in \eqref{e:SF}, as opposed to only
the asymptotic strong Feller property, then not only are the supports
of any two ergodic invariant measures disjoint but they are actually
separated by a distance which is directly related to the function $C$
from \eqref{e:SF}.

\begin{theorem}\label{theo:uniquevar}
  Let $\{\CP_t\}$ be a Markov semigroup on a separable 
  Hilbert space $\CH$ such that \eref{e:SF} holds for some
  non-decreasing function $C$. Let $\mu_1$ and $\mu_2$ be two distinct
  ergodic invariant probability measures for $\CP_t$. Then, the bound
  $\|u_1 - u_2\| \ge 1/C(\|u_1\| \vee \|u_2\|)$ holds for any pair of
  points $(u_1, u_2)$ with $u_i \in \supp \mu_i$.
\end{theorem}

\begin{proof}
  The proof is a variation on the proof of Theorem~3.16 in
  \cite{HairerMattingly06AOM}. We begin by defining for $u,v \in \CH$
  the distance $d_n(u,v)=1 \wedge (\sqrt{\delta_n}\,\|u-v\|)$ where
  $\delta_n$ is the sequence of positive numbers from \eqref{e:SF}. As
  shown in the proof of Theorem~3.12 in \cite{HairerMattingly06AOM},
  one has
  \begin{align}\label{asfTodBound}
    d_n( \cP_t^*\delta_{u_1} ,\cP_t^*\delta_{u_2}) \leq \|u_1-u_2\|
    C(\|u_1\| \vee \|u_2\|)(1+ \sqrt{\delta_n})
  \end{align}
  where $d_n$ is the 1-Wasserstein
  distance\footnote{$d_n(\nu_1,\nu_2)= \sup \int \phi d\nu_1 - \int
    \phi d\nu_2$ where the supremum runs over functions $\phi:\CH
    \rightarrow \R$ which have Lipschitz constant one with respect to
    the metric $d_n$.}  on probability measures induced by the metric
  $d_n$.  Observe that for all $u,v \in \CH$, $d_n(u,v) \leq 1$ and
  $\lim d_n(u,v)=\one_{\{u\}}(v)$. Hence by in Lemma~3.4 of
  \cite{HairerMattingly06AOM}, for any probability measures $\mu$ and
  $\nu$, $\lim_{n \to \infty} d_n(\mu,\nu)=d_{\TV}(\mu,\nu)$ where
  $d_{\TV}(\mu,\nu)$ is the total variation
  distance\footnote{Different communities normalize the total
    variation distance differently.  Our $d_\TV$ is half of the total
    variation distance as defined typically in analysis.  The
    definition we use is common in probability as it is normalised in
    such a way that $d_\TV(\mu,\nu)=1$ for mutually singular
    probability measures.}.

  Let $\mu_1$ and $\mu_2$ be two ergodic invariant measures with
  $\mu_1 \neq \mu_2$. By Birkhoff's ergodic theorem, we know that they
are mutually  singular and thus $d_{\TV}(\mu_1,\mu_2)=1$.  We now
  proceed by contradiction. We assume that there exists a pair of
  points $(u_1,u_2)$ with $u_i \in \supp(\mu_i)$ such that $\|u_1
  -u_2\| < C(\|u_1\| \vee \|u_2\|)$. We will conclude by showing that
  this implies that $d_{\TV}(\mu_1,\mu_2)<1$ and hence $\mu_1$ and
  $\mu_2$ are not singular which will be a contradiction.  

  Our assumption on $u_1$ and $u_2$ implies that there exists a set
  $A$ containing $u_1$ and $u_2$ such that
  $\alpha\eqdef\min(\mu_1(A),\mu_2(A)) >0$ and $\beta\eqdef\sup\{ \|u -v \|
  : u,v \in A\}C(\|u_1\| \vee \|u_2\|)< 1$. As shown in the proof of
  Theorem~3.16 in \cite{HairerMattingly06AOM}, for any $n$ one has
  \begin{align*}
    d_n(\mu_1,\mu_2) &\leq 1-\alpha \big(1- \sup_{v_i \in A} d_n(
    \cP_t^*\delta_{v_1} ,\cP_t^*\delta_{v_2}) \big)\\
    &\leq 1- \alpha \big(1-\beta (1+ \sqrt{\delta_n})\big) 
  \end{align*}
  where the last inequality used the bound in
  equation~\eqref{asfTodBound}. Taking $n \rightarrow \infty$ produces
  $d_{\TV}(\mu_1,\mu_2)\leq 1- \alpha(1-\beta)$. Since $\alpha \in
  (0,1)$ and $\beta<1$ we concluded that $d_{\TV}(\mu_1,\mu_2) <
  1$. This implies a contradiction since $\mu_1$ and $\mu_2$ are
  mutually singular measures.
\end{proof}

Paired with this stronger version of Theorem~\ref{thm:suppDisjoint},
we have the following version of Corollary~\ref{cor:weakIrr} which
uses an even weaker form of irreducibility. This is a general
principle.  If one has a stronger from of the asymptotic strong Feller
property, one can prove unique ergodicity under a weaker form of
topological irreducibility. The form of irreducibility used in
Corollary~\ref{cor:uniqueIM} allows the point where two trajectories
approach to move about, depending on the degree of closeness
required. To prove unique ergodicity, the trade-off is that one needs
some control of the ``smoothing rate'' implied by asymptotic strong
Feller at different points in phase space.

\begin{corollary}\label{cor:uniqueIM}
  Let $\{\CP_t\}$ be as in Theorem~\ref{theo:uniquevar}.  Suppose
  that, for every $R_0>0$, it is possible to find $R>0$ and $T>0$ such
  that, for every $\eps > 0$, there exists a point $v$ with $\|v\| \le
  R$ such that $\CP_T(u,\CB_\eps(v)) > 0$ for every $\|u\| \le
  R_0$. Then, $\CP_t$ can have at most one invariant probability
  measure.
\end{corollary}

\begin{proof}
  Assume by contradiction that there exist two ergodic invariant
  probability measures $\mu_1$ and $\mu_2$ for $\CP_t$. Then, choosing
  $R_0$ large enough so that the open ball of radius $R_0$ intersects
  the supports of both $\mu_1$ and $\mu_2$, it follows form the
  assumption, by similar reasoning as in the proof of
  Corollary~\ref{cor:weakIrr}, that $\supp \mu_i$ intersects
  $\CB_\eps(v)$. Since $\|v\|$ is bounded uniformly in $\eps$, making
  $\eps$ sufficiently small yields a contradiction with
  Theorem~\ref{theo:uniquevar} above.
\end{proof}

\section{Functional analytic setup}
\label{sec:PDEbounds}
In this section, we introduce the basic function analytic set-up for the
rest of the paper. We will develop the needed existence and regularity
theory to place the remainder of the paper of a firm foundation.
We consider semilinear stochastic evolution
equations with additive noise in a Hilbert space $\CH$ (with norm
$\|\,\cdot\,\|$ and innerproduct $\ip{\,\cdot\,}{\,\cdot\,}$ ) of the form
\begin{equ}[e:SPDE]
  du = - Lu\,dt + N(u)\, dt + \sum_{k=1}^d g_k\, dW_k(t)\;,\quad u_0
  \in \CH\;.
\end{equ}
Here, the $W_k$ are independent real-valued standard Wiener processes
over some probability space $(\Omega, \P, \CF)$. Our main standing
assumption throughout this article is that $L$ generates an analytic
semigroup and that the nonlinearity $N$ results in a loss of
regularity of $a$ powers of $L$ for some $a < 1$. More precisely, we
have:
\begin{assumption}\label{ass:basic}\tag{A}
  There exists $a \in [0,1)$ and $\gamma_\star, \beta_\star > -a$
  (either of them possibly infinite) with $\gamma_\star + \beta_\star
  > -1$ such that:
  \begin{enumerate}
  \item The operator $L\colon \CD(L)\to \CH$ is selfadjoint and
    satisfies $\scal{u, Lu} \ge \|u\|^2$. We denote by $\CH_\alpha$,
    $\alpha \in \R$ the associated interpolation spaces (i.e.
    $\CH_\alpha$ with $\alpha > 0$ is the domain of $L^\alpha$ endowed
    with the graph norm and $\CH_{-\alpha}$ is its dual with respect
    to the pairing in $\CH$).  Furthermore, $\CH_\infty$ is the
    Fr\'echet space $\CH_\infty = \bigcap_{\alpha > 0} \CH_\alpha$ and
    $\CH_{-\infty}$ is its dual. \label{ass:L}
  \item There exists $n \ge 1$ such that the nonlinearity $N$ belongs
    to $\Poly^n(\CH_{\gamma + a}, \CH_{\gamma})$ for every $\gamma \in
    [-a,\gamma_\star)$ (see the definition of $\Poly$ in Section~\ref{sec:poly}
    below).  In particular, from the definition of
    $\Poly(\CH_{\gamma+a}, \CH_{\gamma})$, it follows that it is
    continuous from $\CH_\infty$ to $\CH_\infty$.\label{ass:N}
  \item For every $\beta \in [-a, \beta_\star)$ there exists $\gamma
    \in[0, \gamma_\star + 1)$ such that the adjoint (in $\CH$) $DN^*(u)$  of the
    derivative $DN$ of $N$ at $u$ (see again the definition in
    Section~\ref{sec:poly} below) can be extended to a continuous map
    from $\CH_\gamma$ to $\CL(\CH_{\beta + a}, \CH_{\beta})$.
    \label{ass:DN}
  \item One has $g_k \in \CH_{\gamma_\star + 1}$ for every
    $k$. \label{ass:g}
  \end{enumerate}
\end{assumption}

\begin{remark}
  If $\gamma_\star \ge 0$, then the range $\beta \in [-a, 0]$ for
  Assumption~\ass{DN} follows directly from Assumption~\ass{N}, since
  \ass{DN} simply states that for $u \in \CH_\gamma$, $DN(u)$ is a continuous
  linear map from $\CH_{-\beta}$ to $\CH_{-\beta-a}$.
\end{remark}

\begin{remark}
  The assumption $\scal{u, Lu} \geq \|u\|^2$ is made only for
  convenience so that $L^\gamma$ is well-defined as a positive
  selfadjoint operator for every $\gamma \in \R$. It can always be
  realized by subtracting a suitable constant to $L$ and adding it to
  $N$.

  Similarly, non-selfadjoint linear operators are allowed if the
  antisymmetric part is sufficiently ``dominated'' by the symmetric
  part, since one can then consider the antisymmetric part as part of
  the nonlinearity $N$.
\end{remark}

\begin{remark}\label{rem:interpolation}
  It follows directly from the Calder\'on-Lions interpolation theorem
  \cite[Appendix to~IX.4]{RS} that if $N \in \Poly(\CH_{0}, \CH_{-a})
  \cap \Poly(\CH_{\gamma_\star+a}, \CH_{\gamma_\star})$ for some
  $\gamma_\star > -a$, then $N \in \Poly(\CH_{\gamma+a},
  \CH_{\gamma})$ for every $\gamma \in [-a,\gamma_\star]$. This can be
  seen by interpreting $N$ as a sum of \textit{linear} maps from
  $\CH_{\gamma+a}^{\otimes n}$ to $\CH_{\gamma}$ for suitable values
  of $n$.
\end{remark}

It will be convenient in the sequel to define $F$ by
\begin{equ}[e:defF]
  F(u) = -L u + N(u)\;.
\end{equ}
Note that $F$ is in $\Poly^n(\CH_{\gamma+1}, \CH_{\gamma})$ for every
$\gamma \in [-1,\gamma_\star)$.  We also define a linear operator $G
\colon \R^d \to \CH_\infty$ by
\begin{equ}
  G v = \sum_{k=1}^d v_k g_k\;,
\end{equ}
for $v = (v_1,\ldots,v_d) \in \R^d$. With these notations, we will
sometimes rewrite \eref{e:SPDE} as
\begin{equ}[e:SPDE2]
  du = F(u)\, dt + G\, dW(t)\;,\quad u_0 \in \CH\;,
\end{equ}
for $W=(W_1,\ldots,W_d)$ a standard $d$-dimensional Wiener process.

\subsection{Polynomials}
\label{sec:poly}
We now describe in what sense we mean that $N$ is a ``polynomial''
vector field.
Given a Fr\'echet space $X$, we denote by $\CL_s^n(X)$ the space of
continuous symmetric $n$-linear maps from $X$ to itself. We also
denote by $\CL(X,Y)$ the space of continuous linear maps from $X$ to
$Y$.  For the sake of brevity, we will make use of the two equivalent
notations $P(u)$ and $P(u^{\otimes n})$ for $P \in \CL^n(X)$.

Given $Q \in \CL_s^k$, its \textit{derivative} is given by the
following $n-1$-linear map from $X$ to $\CL(X,X)$:
\begin{equ}
  DQ(u) v = k Q\bigl(u^{\otimes (k - 1)} \otimes v\bigr)\;.
\end{equ}
We will also use the notation $DQ^* \colon X \to \CL(X', X')$ for the
dual map given by
\begin{equ}
  \scal{w,DQ^*(u) v} = \scal{v, DQ(u)w} = k \scal[b]{v,
    Q\bigl(u^{\otimes (k - 1)} \otimes w\bigr)}\;.
\end{equ}

Given $P \in \CL_s^k$ and $Q \in \CL_s^\ell$, we define the derivative
$DQ\, P$ of $Q$ in the direction $P$ as a continuous map from $X\times
X$ to $X$ by
\begin{equ}
  DQ(u)\, P(v) = \ell Q\bigl(u^{\otimes (\ell - 1)} \otimes
  P(v)\bigr)\;.
\end{equ}
Note that by polarisation, $u \mapsto DQ(u) P(u)$ uniquely defines an
element on $\CL_s^{k+\ell-1}$.  This allows us to define a ``Lie
bracket'' $[P,Q] \in \CL_s^{k+\ell-1}$ between $P$ and $Q$ by
\begin{equ}{} [P,Q](u) = DQ(u)\, P(u) - DP(u)\, Q(u)\;.
\end{equ}
We also define $\Poly^n(X)$ as the set of continuous maps $P \colon X
\to X$ of the form
\begin{equ}
  P(u) = \sum_{k=0}^n P^{(k)}(u)\;,
\end{equ}
with $P^{(k)} \in \CL_s^k(X)$ (here $\CL_s^0(X)$ is the space of
constant maps and can be identified with $X$). We also set $\Poly(X) =
\bigcup_{n \ge 0} \Poly^n(X)$.  The Lie bracket defined above extends
to a map from $\Poly(X) \times \Poly(X) \to \Poly(X)$ by linearity.

\subsubsection{Polynomials over $\CH$}

We now specialize to polynomials over $\CH$. We begin by choosing $X$
equal to the Fr\'echet space $\CH_\infty$, the intersection of $\CH_a$
over all $a > 0$.  Next we define the space $\Poly(\CH_a,\CH_b)
\subset \Poly(\CH_\infty)$ as the set of polynomials $P\in
\Poly(\CH_\infty)$ such that there exists a continuous map $\hat P
\colon \CH_a \to \CH_b$ with $\hat P(u) = P(u)$ for all $u \in
\CH_\infty$. Note that in general (unlike $\Poly(\CH_\infty)$), $P,
Q\in \Poly(\CH_a,\CH_b)$ does not necessarily imply $[P,Q] \in
\Poly(\CH_a,\CH_b)$.  We will make an abuse of notation and use the
same symbol for both $P$ and $\hat P$ in the sequel.

\subsubsection{Taylor expansions and Lie brackets}

We now consider the Taylor expansion of a polynomial $Q$ in a
direction $g$ belonging to $\SPAN\{g_1,\cdots,g_d\} \subset \CH_{\gamma_\star+1}$.
Fix $Q \in \Poly^m(\CH_{\gamma}, \CH_\beta)$ for some $\gamma \le
\gamma_\star + 1$ and any $\beta \in \R$. For $v \in \CH_\gamma$ and
$w = (w_1,\ldots,w_d) \in \R^d$, observe that there exist
polynomials $Q_\alpha$ such that
\begin{equ}[e:TaylorQ]
  Q\Bigl(v+\sum_{k=1}^d g_k w_k\Bigr) = \sum_\alpha Q_\alpha(v)
  w_\alpha\;,\quad w_\alpha = w_{\alpha_1}\cdots w_{\alpha_\ell}\;,
\end{equ}
where the summation runs over all multi-indices $\alpha =
(\alpha_1,\ldots,\alpha_\ell)$, $\ell \ge 0$ with values in the index
set $\{1,\ldots,d\}$. It can be checked that the polynomials $Q_\alpha \in
\Poly^{m-|\alpha|}(\CH_{\gamma}, \CH_\beta)$ are given by the formula
\begin{equ}[e:repeatedBrakets]
  Q_\alpha(v) = {1\over \alpha!} \bigl[ [\ldots [Q,
  g_{\alpha_1}]\ldots ], g_{\alpha_\ell} \bigr] = {1\over
    \alpha!}D^{|\alpha|}Q(v) (g_{\alpha_1}, \ldots, g_{\alpha_{\ell}})\;.
\end{equ}
Here, $\alpha!$ is defined by $\alpha! = \alpha(1)! \cdots \alpha(d)!$, where $\alpha(j)$ counts
the number of occurences of the index $j$ in $\alpha$. (By convention, we
set $Q_\emptyset = Q$ and $Q_\alpha = 0$ if $|\alpha| > m$.)

We emphasize that multi-indices are \textit{unordered} collections of
$\{1,\ldots,d\}$ where repeated elements are allowed. As such, the
union of two multi-indices is a well-defined operation, as is the
partial ordering given by inclusion.\footnote{To be precise, one could
  identify a multi-index with its counting function $\alpha\colon
  \{1,\ldots,d\} \to \N$. With this identification, the union of
  two multi-indices corresponds to the sums of their counting
  functions, while $\alpha \subset \beta$ means that $\alpha(k) \le
  \beta(k)$ for every $k$.}

\subsection{\textit{A priori} bounds on the solution}

This section is devoted to the proof that Assumption~\ref{ass:basic}
is sufficient to obtain not only unique solutions to \eref{e:SPDE}
(possibly up to some explosion time), but to obtain further regularity
properties for both the solution and its derivative with respect to
the initial condition. We do not claim that the material presented in
this section is new, but while similar frameworks can be found in
\cite{ZDP1,Flandoli95}, the framework presented here does not seem to appear
in this form in the literature. Since the proofs are rather
straightforward, we choose to present them for the sake of
completeness, although in a rather condensed form.

We first start with a local existence and uniqueness result for the
solutions to \eref{e:SPDE}:

\begin{proposition}\label{prop:Picard}
  For every initial condition $u_0 \in \CH$, there exists a stopping
  time $\tau > 0$ such that \eref{e:SPDE} has a unique mild solution
  $u$ up to time $\tau$, that is to say $u$ almost surely satisfies
  \begin{equ}[e:mild]
    u_t = e^{-Lt} u_0 + \int_0^t e^{-L(t-s)} N(u_s)\, ds + \int_0^t
    e^{-L(t-s)} G\, dW(s)\;,
  \end{equ}
  for all stopping times $t$ with $t \le \tau$. Furthermore $u$ is
  adapted to the filtration generated by $W$ and is in $C([0,\tau),
  \CH)$ with probability one.
\end{proposition}

\begin{remark}
Since we assume that $N$ is locally Lipschitz continuous from $\CH$ 
to $\CH_{-a}$ for some $a < 1$ and since the bound $\|e^{-Lt}\|_{\CH_{-a} \to \CH} \le C t^{-a}$
holds for $t \le T$, the first integral appearing in \eref{e:mild} does converge in $\CH$. Therefore
the right hand side of \eref{e:mild} makes sense for every continuous $\CH$-valued process $u$.
\end{remark}

For notational convenience, we denote by $\WL(s,t) = \int_s^t
e^{-L(t-r)} G\, dW(r)$ the ``stochastic convolution.'' Since we
assumed that $g_k \in \CH_{\gamma_\star + 1}$, it is possible to
obtain bounds on all exponential moments of $\sup_{0 \le s < t \le T}
\|\WL(s,t)\|_\gamma$ for every $T>0$ and every $\gamma \le
\gamma_\star + 1$.

\begin{proof}
  Given a function $\xi\colon \R_+ \to \CH$ and a time $T>0$, define a
  map $\Phi_{T,\xi} \colon \CH\times \CC([0,T],\CH)\to \CC([0,T],\CH)$
  (endowed with the supremum norm) by
  \begin{equ}[e:PicardMap]
    \bigl(\Phi_{T,\xi}(u_0, u)\bigr)_t = e^{-Lt} u_0 + \xi(t) +
    \int_0^t e^{-L(t-s)} N(u_s)\, ds \;.
  \end{equ}
  Since $N \in \Poly(\CH, \CH_{-a})$ by setting $\gamma = -a$ in
  Assumption~\ass{N}, and suppressing the dependence on $u_0$, there
  exists a positive constant $C$ such that
  \begin{equs}
    \bigl\|\Phi_{T,\xi}(u) - \Phi_{T,\xi}(\tilde u)\bigl\| &\le
    \sup_{t \in [0,T]}
    C\int_0^t (t-s)^{-a} \|u_s - \tilde u_s\| \bigl(1+\|u_s\| + 
    \|\tilde u_s\|\bigr)^{n-1}\, ds \\
    &\le C \|u - \tilde u\| \bigl(1+\|u\| + \|\tilde u\|\bigr)^{n-1}
    T^{1-a}\;.
  \end{equs}
  Recall that $n$ is the degree of the polynomial nonlinearity $N$.
  It follows that, for every $\xi$, there exists $T>0$ and $R>0$ such
  that $\Phi_{T,\xi}(u_0, \cdot\,)$ is a contraction in the ball of
  radius $R$ around $e^{-Lt} u_0 + \xi(t)$. Setting $\xi(t) =
  \WL(0,t)$, this yields existence and uniqueness of the solution to
  \eref{e:mild} for almost every noise path $\WL(0,t)$ by the Banach fixed point theorem.  The largest such
  $T$ is a stopping time since it only depends on the norm of $u_0$
  and on $\xi$ up to time $T$. It is clear that $\Phi_{T,\xi}(u_0,
  u)_t$ only depends on the noise $\WL$ up to time $t$, so that the solution is adapted to the
  filtration generated by $W$, thus concluding the proof of the
  proposition.
\end{proof}

The remainder of this section is devoted to obtaining further
regularity properties of the solutions. 

\begin{proposition}\label{prop:bootstrapping}
  Fix $T > 0$. For every $\gamma
  \in [0, \gamma_\star + 1)$ there exist exponents $p_\gamma \ge 1$
  and $q_\gamma \ge 0$ and a constant $C$ such that
  \begin{equ}[e:boundugamma]
    \|u_t\|_\gamma \le C t^{-q_\gamma} \bigl(1 + \sup_{s \in [{t\over
        2},t]}\|u_s\| + \sup_{{t\over 2} \le s < r \le t}
    \|\WL(s,r)\|_\gamma \bigr)^{p_\gamma}
  \end{equ}
  for all $t \in (0,T\wedge \tau]$, where $\tau = \sup\{t>0\,:\, \|u_t\| < \infty\}$.  In particular, if $\gamma = \sum_{j=0}^k
  \delta_j$ for some $k \in \N$ and $\delta_j \in (0,1-a)$ then
  $q_\gamma \leq \sum_{j=1}^{k} \delta_j n^{j-1}$.
\end{proposition}

\begin{proof}
  The proof follows a standard ``bootstrapping argument'' on $\gamma$
  in the following way.  The statement is obviously true for $\gamma =
  0$ with $p_\gamma = 1$ and $q_\gamma = 0$.  Assume that, for some
  $\alpha = \alpha_0 \in [1/2, 1)$ and for some $\gamma = \gamma_0 \in
  [0,\gamma_\star+a)$, we have the bound
  \begin{equ}[e:boundugammaalpha]
    \|u_t\|_\gamma \le C t^{-q_\gamma} \bigl(1 + \sup_{s \in [\alpha
      t,t]}\|u_s\| + \sup_{\alpha t \le s < r \le t}
    \|\WL(s,r)\|_\gamma \bigr)^{p_\gamma}\;,
  \end{equ}
  for all $t \in (0,T]$.

  We will then argue that, for any arbitrary $\delta \in (0, 1-a)$,
  the statement \eref{e:boundugammaalpha} also holds for $\gamma =
  \gamma_0 + \delta$ (and therefore also for all intermediate values of $\gamma$)
  and $\alpha = \alpha_0^2$. Since it is possible to go from $\gamma =
  0$ to any value of $\gamma < \gamma_\star + 1$ in a finite number of
  steps (making sure that $\gamma \le 1+a$ in every intermediate step)
  and since we are allowed to choose $\alpha$ as close to $1$ as we
  wish, the claim follows at once.

  Using the mild formulation \eref{e:mild}, we have
  \begin{equ}
    u_t = e^{-(1-\alpha) L t} u_{\alpha t} + \int_{\alpha t}^t
    e^{-L(t-s)} N(u_s)\, ds + \WL(\alpha t,t)\;.
  \end{equ}
  Since $\gamma \in [0, \gamma_\star + a)$, one has $N \in
  \Poly(\CH_\gamma, \CH_{\gamma-a})$ by Assumption~\ass{N}.  Hence,
  for $t \in (0,T]$,
  \begin{equs}
    \|u_t\|_{\gamma+\delta} &\le C t^{-\delta} \|u_{\alpha t}\|_\gamma
    + \|\WL(\alpha t,t)\|_\gamma + C\int_{\alpha t}^t (t-s)^{-(\delta
      +
      a)} (1+\|u_s\|_\gamma)^n\, ds\\
    &\le C \bigl(t^{-\delta} + t^{1-\delta - a}\bigr)\sup_{{\alpha t}
      \le s \le t}\bigl(1+\|u_s\|_\gamma^n\bigr) + \|\WL(\alpha
    t,t)\|_\gamma \\
    &\le C t^{-\delta}\sup_{{\alpha t} \le s \le
      t}\bigl(1+\|u_s\|_\gamma^n\bigr) + \|\WL(\alpha t,t)\|_\gamma\;.
  \end{equs}
  Here, the constant $C$ depends on everything but $t$ and $u_0$.
  Using the induction hypothesis, this yields the bound
  \begin{equ}
    \|u_t\|_{\gamma+\delta} \le C t^{-\delta - n q_\gamma} \bigl(1 +
    \sup_{s \in [\alpha^2 t,t]}\|u_s\| + \sup_{\alpha^2 \le s < r \le
      t} \|\WL(s,r)\|_\gamma \bigr)^{n p_\gamma} + \|\WL(\alpha
    t,t)\|_\gamma \;,
  \end{equ}
  thus showing that \eref{e:boundugammaalpha} holds for $\gamma =
  \gamma_0 + \delta$ and $\alpha = \alpha_0^2$ with $p_{\gamma +
    \delta} = n p_\gamma$ and $q_{\gamma + \delta} = \delta +
  nq_{\gamma}$. This concludes the proof of
  Proposition~\ref{prop:bootstrapping}.
\end{proof}

\subsection{Linearization and its adjoint}
\label{sec:lin}

In this section, we study how the solutions to \eref{e:SPDE} depend on
their initial conditions. Since the map from \eref{e:PicardMap} used
to construct the solutions to \eref{e:SPDE} is Fr\'echet
differentiable (it is actually infinitely differentiable) and since it
is a contraction for sufficiently small values of $t$, we can apply
the implicit functions theorem (see for example \cite{RenRog} for a
Banach space version) to deduce that for every realisation of the
driving noise, the map $u_s \mapsto u_t$ is Fr\'echet differentiable,
provided that $t>s$ is sufficiently close to $s$.

Iterating this argument, one sees that, for any
 $s \le t < \tau$, the map
$u_s \mapsto u_t$ given by the solutions to \eref{e:SPDE} 
is Fr\'echet differentiable in $\CH$. Inspecting the expression for the derivative given by the
implicit functions theorem, we conclude that the derivative $J_{s,t}\phi$ in the
direction $\phi \in \CH$ satisfies the following random linear equation in its mild
formulation:
\begin{equ}[e:Jacobian]
  \d_t J_{s,t} \phi = -L J_{s,t} \phi + DN(u_t) J_{s,t} \phi\;,\quad
  J_{s,s}\phi = \phi\;.
\end{equ}
Note that, by the properties of monomials, it follows from
Assumption~\ass{N} that
\begin{equ}
  \|DN(u) v\|_{\gamma} \le C(1+ \|u\|_{\gamma+a})^{n-1}
  \|v\|_{\gamma+a}\;,
\end{equ}
for every $\gamma \in [-a, \gamma_\star)$. A fixed point argument
similar to the one in Proposition~\ref{prop:Picard} shows that the
solution to \eref{e:Jacobian} is unique, but note that it does not
allow us to obtain bounds on its moments. We only have that for any $T$ smaller
than the explosion time to the solutions of \eref{e:SPDE},
there exists a (random) constant $C$ such that
\begin{equ}[e:Jfinite]
  \sup_{0\leq s<t<T} \sup_{\|\phi\| \le 1}  \|J_{s,t} \phi\| \leq C\;.
\end{equ}
The constant $C$ depends exponentially on the size of the solution $u$
in the interval $[0,T]$. However, if we obtain better control on $J_{s,t}$ by some
means, we can then use the following bootstrapping argument:

\begin{proposition}\label{prop:bootstrapJ}
  For every $\gamma \in (0, \gamma_\star + 1)$, there exists an exponent $\tilde q_\gamma \geq 0$, 
  and constants $C>0$ and $\gamma_0
  < \gamma$ such that we have the bound
  \begin{equ}[e:boundKst]
    \|J_{t,t+s}\phi\|_\gamma \le C s^{-\gamma} \sup_{r \in
      [{s\over 2},s]} \bigl(1 + \|u_{t+r}\|_{\gamma_0}\bigr)^{\tilde
      q_\gamma}\|J_{t,t+r} \phi\|\;,
  \end{equ}
  for every $\phi \in \CH$ and every $t,s > 0$. If
  $\gamma < 1-a$, then one can choose $\gamma_0 = 0$
  and $\tilde q_\gamma=n-1$.
\end{proposition}

Since an almost identical argument will be used in the proof of
Proposition~\ref{prop:bootstrapK} below, we refer the reader there for
details. We chose to present that proof instead of this one because
the presence of an adjoint causes slight additional complications.

For $s \le t$, let us define operators $K_{s,t}$ via the solution to
the (random) PDE
\begin{equ}[e:defKst]
  \d_s K_{s,t} \phi = L K_{s,t}\phi - DN^*(u_s) K_{s,t}\phi\;,\quad
  K_{t,t}\phi = \phi\;,\quad \phi \in \CH\;.
\end{equ}
Note that this equation runs \textit{backwards} in time and is random through
the solution $u_t$ of \eref{e:SPDE}.  Here, $DN^*(u)$ denotes the
adjoint in $\CH$ of the operator $DN(u)$ defined earlier. Fixing the
terminal time $t$ and setting $\phi_s = K_{t-s,t}\phi$, we obtain a
more usual representation for $\phi_s$:
\begin{equ}[e:defPhis]
  \d_s \phi_s = - L \phi_s + DN^*(u_{t-s}) \phi_s\;.
\end{equ}
The remainder of this subsection will be devoted to obtaining
regularity bounds on the solutions to \eref{e:defKst} and to the proof
that $K_{s,t}$ is actually the adjoint of $J_{s,t}$. We start by
showing that, for $\gamma$ sufficiently close to (but less than)
$\gamma_\star+1$, \eref{e:defKst} has a unique solution for every path
$u\in \CC(\R, \CH_\gamma)$ and $\phi\in\CH$.

\begin{proposition}\label{prop:PicardKst}
  There exists $\gamma < \gamma_\star + 1$ such that, for every $\phi \in \CH$, equation
  \eref{e:defKst} has a unique continuous $\CH$-valued solution for every $s < t$ and every
  $u\in \CC(\R, \CH_\gamma)$. Furthermore, $K_{s,t}$ depends only on
  $u_r$ for $r \in [s,t]$ and the map $\phi \mapsto K_{s,t}\phi$ is linear and bounded.
\end{proposition}
\begin{proof}
  As in Proposition~\ref{prop:Picard}, we define a map $\Phi_{T,u}
  \colon \CH \times \CC([0,T],\CH) \to \CC([0,T],\CH)$ by
  \begin{equ}
    \bigl(\Phi_{T,u}(\phi_0, \phi)\bigr)_t = e^{-Lt} \phi_0 +
    \int_0^t e^{-L(t-s)} \bigl(DN^*(u_s)\bigr) \phi_s \, ds \;.
  \end{equ}
  It follows from Assumption~\ass{DN} with $\beta = -a$ that there
  exists $\gamma < \gamma_\star + 1$ such that $DN^*(u) \colon \CH \to
  \CH_{-a}$ is a bounded linear operator for every $u \in \CH_\gamma$.
  Proceeding as in the proof of Proposition~\ref{prop:Picard}, we see
  that $\Phi$ is a contraction for sufficiently small $T$.
\end{proof}

Similarly to before, we can use a bootstrapping argument to show that
$K_{s,t}\phi$ actually has more regularity than stated in Proposition~\ref{prop:PicardKst}.

\begin{proposition}\label{prop:bootstrapK}
  For every $\beta \in (0, \beta_\star+1)$, there exists $\gamma <
  \gamma_\star + 1$, an exponent $\bar q_\beta > 0$, and a
  constant $C$ such that
  \begin{equ}[e:boundKst2]
    \|K_{t-s,t}\phi\|_\beta \le C s^{-\beta} \sup_{r \in
      [{s\over 2},s]} \bigl(1 + \|u_{t-r}\|_\gamma\bigr)^{\bar
      q_\beta}\|K_{t-r,t} \phi\|\;,
  \end{equ}
  for every $\phi \in \CH$, every $t,s > 0$, and every $u \in \CC(\R,
  \CH_\gamma)$.
\end{proposition}

\begin{proof}
  Fix $\beta < \beta_\star + a$ and $\delta \in (0,1-a)$ and assume
  that the bound \eref{e:boundKst2} holds for $\|K_{s,t}\phi\|_\beta$.
  Since we run $s$ ``backwards in time'' from $s = t$, we consider
  again $t$ as fixed and use the notation $\phi_s = K_{t-s,t}\phi$.
  We then have, for arbitrary $\alpha \in (0,1)$,
  \begin{equ}
    \|\phi_s\|_{\beta+\delta} \le C s^{-\delta} \|\phi_{\alpha
      s}\|_\beta+ C\int_{\alpha s}^s (s-r)^{-(\delta + a)}
    \|DN^*(u_{t-r}) \phi_r\|_{\beta - a}\, dr\;,
  \end{equ}
  provided that $\gamma$ is sufficiently close to $\gamma_\star + 1$
  such that $DN^* \colon \CH_\gamma \to \CL(\CH_{\beta},
  \CH_{\beta-a})$ by Assumption~\ass{DN}.  Furthermore, the operator
  norm of $DN^*(v)$ is bounded by $C(1+\|v\|_{\gamma})^{n-1}$,
  yielding
  \begin{equs}
    \|\phi_s\|_{\beta+\delta} &\le C s^{-\delta} \|\phi_{\alpha
      s}\|_\beta+ C s^{-(\delta + a-1)} \sup_{r \in [\alpha s,
      s]}(1+\|u_r\|_\gamma)^{n-1} \|\phi_r\|_{\beta}\\
    &\le C s^{-\delta} \sup_{r \in [\alpha s,
      s]}(1+\|u_r\|_\gamma)^{n-1} \|\phi_r\|_{\beta}\;.
  \end{equs}
  Iterating these bounds as in Proposition~\ref{prop:bootstrapping}
  concludes the proof.
\end{proof}
The following lemma appears also in \cite{MatPar06:1742,Yuri}. It
plays a central role in establishing the representation of the
Malliavin matrix given in \eqref{e:Mscalar} on which this article as
well as \cite{MatPar06:1742,Yuri} rely  heavily.
\begin{proposition}\label{e:JKadjoint}
  For every $0 \le s < t$, $K_{s,t}$ is the adjoint of $J_{s,t}$ in $\CH$,
  that is $K_{s,t} = J_{s,t}^*$.
\end{proposition}

\begin{proof}
  Fixing $0 \le s < t$ and $\phi, \psi \in \CH_\infty$, we claim that the
  expression
  \begin{equ}[e:duality]
    \scal{J_{s,r} \phi, K_{r,t} \psi}\;,
  \end{equ}
  is independent of $r \in [s,t]$. Evaluating \eref{e:duality} at both
  $r = s$ and $r = t$ then concludes the proof.

  We now prove that \eref{e:duality} is independent of $r$ as
  claimed. It follows from \eref{e:defKst} and
  Proposition~\ref{prop:bootstrapping} that, with probability one, the
  map $r \mapsto K_{r,t} \phi$ is continuous with values in
  $\CH_{\beta+1}$ and differentiable with values in $\CH_{\beta}$,
  provided that $\beta < \beta_\star$.  Similarly, the map $r \mapsto
  J_{s,r}\psi$ is continuous with values in $\CH_{\gamma+1}$ and
  differentiable with values in $\CH_\gamma$, provided that $\gamma <
  \gamma_\star$. Since $\gamma_\star + \beta_\star > -1$ by
  assumption, it thus follows that \eref{e:duality} is differentiable
  in $r$ for $r\in (s,t)$ with
  \begin{equs}
    \d_r \scal{J_{s,r} \phi, K_{r,t} \psi} &= \scal{\bigl(L +
      DN(u_r)\bigr) J_{s,r} \phi, K_{r,t} \psi} \\
    &\quad - \scal{J_{s,r} \phi, \bigl(L + DN^*(u_r)\bigr)K_{r,t}
      \psi} = 0\;.
  \end{equs}
  Since furthermore both $r \mapsto K_{r,t} \phi$ and $r \mapsto
  J_{s,r}\psi$ are continuous in $r$ on the closed interval, the proof
  is complete. See for example \cite[p.~477]{DauLio5} for more details.
\end{proof}

\subsection{Higher order variations}

We conclude this section with a formula for the higher-order variations
of the solution. This will mostly be useful in Section~\ref{sec:examples} in order to
obtain the smoothness of the density for finite-dimensional
projections of the transition probabilities.

For integer $n \geq 2$, let $\phi=(\phi_1,\cdots,\phi_n) \in
\CH^{\otimes n}$ and
$s=(s_1,\cdots,s_n) \in [0,\infty)^n$ and define $\lor s= s_1 \lor
\cdots \lor s_n$. We will now define the $n$-th variation of the
equation $ J^{(n)}_{s,t}\phi$ which intuitively is the cumulative effect on
$u_t$ of varying the value of $u_{s_k}$ in the direction $\phi_k$.

If  $I=\{n_1<\ldots<n_{|I|}\}$ is an ordered subset of $\{1,\ldots,n\}$
(here $|I|$ means the number of elements in $I$), we introduce the notation
$s_I =(s_{n_1},\ldots,s_{n_{|I|}})$ and $\phi_I = (\phi_{n_1},\ldots,\phi_{n_{|I|}})$.
Now the  $n$-th variation of the
equation $ J^{(n)}_{s,t}\phi$ solves
\begin{align}\label{eq:higher_variations}
 \partial_t J^{(n)}_{s,t}\phi &= -L
 J^{(n)}_{s,t}\phi +    DN(u(t))  J^{(n)}_{s,t}\phi 
 +\CG^{(n)}_{s,t}(u(t),\phi),\quad t>\lor s,\\
 J^{(n)}_{s,t}\phi &= 0, \quad t\leq \lor s,\notag
\end{align}
where
\begin{equ}
 \label{eq:Gn}
 \CG^{(n)}_{s,t}(u,\phi)=\sum_{\nu=2}^{m\land n}\sum_{I_1,\ldots,I_\nu}
 D^{(\nu)}N(u)\left(J_{s_{I_1},t}^{(|I_1|)}\phi_{I_1},\ldots,J_{s_{I_\nu},t}^{(|I_\nu|)}\phi_{I_\nu}\right)\;,
\end{equ}
and the second sum runs over all partitions of $\{1,\ldots,n\}$ into
disjoint, ordered non-empty sets  $I_1,\ldots,I_\nu$. 

The variations of constants formula then implies that
\begin{equ}[e:exprJn]
  J^{(n)}_{s,t}\phi = \int_0^t   J_{r,t}
  \CG^{(n)}_{s,r}(u_r,\phi) dr\;,
\end{equ}
see also \cite{Yuri}. We obtain the following bound on the higher-order variations:

\begin{proposition}\label{prop:boundvarn}
If $\beta_\star > a-1$ then there exists $\gamma < \gamma_\star+1$ such that, for every $n>0$, there exist exponents $N_n$ and $M_n$ such that
\begin{equ}
\| J^{(n)}_{s,t}\phi\| \le C \sup_{r \in [0,t]} (1+\|u_r\|_\gamma)^{N_n} \sup_{0 \le u < v \le t} (1+\|J_{u,v}\|)^{M_n}\;,
\end{equ}
uniformly over all $n$-uples $\phi$ with $\|\phi_k\|\le 1$ for every $k$.
\end{proposition}

\begin{proof}
We proceed by induction.
As a shorthand, we set
\begin{equ}
\CE(M,N) =  \sup_{r \in [0,t]} (1+\|u_r\|_\gamma)^{N} \sup_{0 \le u < v \le t} (1+\|J_{u,v}\|)^{M}\;.
\end{equ}
The result is trivially true for $n = 1$ with $M_1 = 1$ and $N_1 = 0$. For $n > 1$, we
combine \eref{e:exprJn} and \eref{eq:Gn}, and we use Assumption~\ref{ass:basic}, part 2., to obtain
\begin{equs}
\| J^{(n)}_{s,t}\phi\| &\le  C\int_0^t  \|J_{r,t}\|_{-a \to 0}
  \Bigl(1 + \|u_r\|^n + \sum_{I} \|J^{|I|}_{s_I,r}\phi_I\|^n\Bigr) \,dr \\
  &\le C \CE\bigl(n M_{n-1}, n(N_{n-1}+1)\bigr) \int_0^t  \|K_{r,t}\|_{0\to a} \,dr\;.
\end{equs} 
To go from the first to the second line, we used the induction hypothesis, the fact that 
$K_{r,t} = J_{r,t}^*$, and the duality between $\CH_a$ and $\CH_{-a}$.

It remains to apply Proposition~\ref{prop:bootstrapK} with $\beta = a$ to obtain the required bound.
\end{proof}

\section{Malliavin calculus}
\label{sec:Malliavin}

In this section, we show that the solution to the SPDE \eref{e:SPDE}
has a Malliavin derivative and we give an expression for it. Actually, since we
are dealing with additive noise, we show the stronger result that the solution
is Fr\'echet differentiable with respect to the driving noise.
In this section, we will make the standing assumption that the explosion time $\tau$
from Proposition~\ref{prop:Picard} is infinite.  

\subsection{Malliavin derivative}

In light of Proposition~\ref{prop:Picard}, for fixed initial condition $u_0 \in \CH$
there exists an ``It\^o map'' $\Phi_t^{u_0}:
\CC([0,t],\R^d) \rightarrow \CH$ with $u_t = \Phi_t^{u_0}(W)$. We have:

\begin{proposition}\label{prop:exprMall}
For every $t > 0$ and every $u \in \CH$, the map $\Phi_t^u$ is Fr\'echet differentiable
and its Fr\'echet derivative  $\DF \Phi_t^u v$ in the direction $v \in \CC(\R_+,\R^d)$ 
satisfies the equation
\begin{equ}[e:equMall]
d\,\DF \Phi_t^u v = -L \DF \Phi_t^u v\, dt + DN(u_t) \DF \Phi_t^u v\, dt + G dv(t)
\end{equ}
in the mild sense.
\end{proposition}

\begin{remark}\label{rem:bounddv}
  Note that \eref{e:equMall} has a unique $\CH$-valued mild solution
  for every continuous function $v$ because it follows from our
  assumptions that $Gv \in \CC(\R_+,\CH_\gamma)$ for some $\gamma>0$
  and therefore $\int_0^t e^{-L(t-s)}G\,dv(s) = Gv(t) - e^{-Lt}Gv(0) -
  \int_0^t Le^{-L(t-s)} Gv(s)\,ds$ is a continuous $\CH$-valued
  process.
\end{remark}

\begin{proof}[of Proposition~\ref{prop:exprMall}]
The proof works in exactly the same way as the arguments presented in Section~\ref{sec:lin}: it follows
from Remark~\ref{rem:bounddv} that for any given $u_0 \in \CH$ and $t > 0$, the map 
\begin{equ}
  (W,u) \mapsto e^{-Lt} u_0 + \int_0^t e^{-L(t-s)} N(u(s))\, ds + \int_0^t e^{-L(t-s)}G\,dW(s)
\end{equ}
is Fr\'echet differentiable in $\CC([0,t],\R^d)\times
\CC([0,t],\CH)$. Furthermore, for $t$ sufficiently small (depending on
$u$ and $W$), it satisfies the assumptions of the implicit functions
theorem, so that the claim follows in this case.  The claim for
arbitrary values of $t$ follows by iterating the statement.
\end{proof}

As a consequence, it follows from Duhamel's formula and the fact that $J_{s,t}$ is the
unique solution to \eref{e:Jacobian} that

\begin{corollary}\label{cor:FrechetDiff}
If $v$ is absolutely continuous and of bounded variation, then 
\begin{equ}[e:reprMall]
  \DF \Phi_t^{u} v = \int_0^t J_{s,t} G dv(s)\;, 
\end{equ}
where the integral is to be understood as a Riemann-Stieltjes integral and
the Jacobian $J_{s,t}$ is as in \eref{e:Jacobian}.
\end{corollary}

In particular, \eref{e:reprMall} 
holds for every $v$ in the \textit{Cameron-Martin space}
\begin{equ}
  \CamM=\big\{v : \d_t v \in L^2([0, \infty), \R^d), \quad v(0)=0\big\}\;,
\end{equ}
which is a Hilbert space endowed with the norm $\|v\|_\CamM^2 =
\int_0^\infty |\d_t v(t)|_{\R^d}^2\, dt\eqdef\Norm{\d_t v}^2$.  Obviously,
$\CamM$ is isometric to $\Cam = L^2([0, \infty), \R^d)$, so we will in
the sequel use the notation
\begin{equ}[e:reprDPhi]
  \DM_h \Phi_t^{u} \eqdef \DF \Phi_t^{u}  v = \int_0^t J_{s,t} G dv(s) = \int_0^t J_{s,t} G h(s)\,ds \;,\quad \text{if}\;\d_t v = h\;.
\end{equ}
The representation \eref{e:reprMall} is still valid for arbitrary
stochastic processes $h$ such that $h \in \Cam$ almost surely. 

Since $G:\R^d \rightarrow \CH_{\gamma_*+1}$ is a bounded operator
whose norm we denote $\|G\|$, we obtain the bound
\begin{equ}
  \|\DM_h \Phi_t^{u}\| \leq  \|G\| \int_0^t \|J_{s,t}\|\,|h(s)|\,ds \leq
  C\|J_{\cdot,t}\|_{L^2(0,t,\CH)}  \Norm{h}\;,
\end{equ}
valid for every $h \in \Cam$. In particular, by Riesz's representation theorem, 
this shows that there exists a (random) element $\DM \Phi_t^{u}$
of $\Cam \otimes \CH$ such that 
\begin{equ}[e:DMeqJ]
\DM_h \Phi_t^{u} = \scal{\DM \Phi_t^{u}, h}_{\Cam}= \int_0^\infty \DM_s \Phi_t^{u} h(s)\,ds\;,
\end{equ}
for every $h \in \Cam$. This abuse of notation is partially justified by
the fact that, at least formally, $\DM_s \Phi_t^u= \DM_h \Phi_t^u$
with $h(r)=\delta(s-r)$.
In our particular case, it follows from \eref{e:reprMall} that one has
\begin{equ}
  \DM_s \Phi_t^{u} = J_{s,t} G \in \R^d \otimes \CH\;, \quad t>s\;,
\end{equ}
and $\DM_s \Phi_t^{u} = 0$ for $s > t$.  With this notation, the identity
  \eref{e:reprMall} can be rewritten as $\DM_h u_t=\int_0^t \DM_s
  u_t\,h(s)ds$.
It follows from the theory of
Malliavin calculus, see for example \cite{MalliavinBook,Nualart} that, for any Hilbert space $\CH$, 
there exists a closed unbounded linear operator $\DM\colon L^2(\Omega,\R)\otimes \CH \to L^2_\ad(\Omega,\CF_t, \Cam)
\otimes \CH$ such that $\CD \Phi_t$ coincides with the object described above whenever $\Phi_t$ is the solution
map to \eref{e:SPDE}. Here, $\cF_t$ is the $\sigma$-algebra generated by the
increments of $W$ up to time $t$ and $L^2_\ad$ denotes the space of $L^2$ functions
adapted to the filtration $\{\cF_t\}$. 

The operator $\DM$ simply acts as the
identity on the factor $\CH$, so that we really interpret it as an
operator from $L^2(\Omega,\R)$ to $L^2(\Omega, \Cam)$.
The operator $\DM$ is called the ``Malliavin derivative.''

We define a family of random linear operators $\AM_t\colon \Cam
\rightarrow \CH$ (depending also on the initial condition $u_0 \in
\CH$ for \eref{e:SPDE}) by $h \mapsto \scal{\DM\Phi_t^{u},h}$.  It
follows from \eref{e:reprDPhi} that their adjoints $\AM_t^*\colon\CH
\rightarrow \Cam$ are given for $\xi \in \CH$ by
\begin{equation}\label{eq:Aadj}
  (\AM_t^*\xi)(s)=
  \begin{cases}
    G^*J_{s,t}^* \xi = G^*K_{s,t}\xi& \text{ for } s  \leq t\;, \\
    0 &\text{ for } s > t\;.
  \end{cases}
\end{equation}
Similarly, we define $\AM_{s,t}\colon\Cam \rightarrow \CH$ by
$\AM_{t,s}h \eqdef \AM_t(h\one_{[t,s]})= \scal{\DM u_t, h\one_{[t,s]}}= \int_s^t J_{r,t}G h_r dr$.
Observe that $\AM_{s,t}^*\colon\CH \rightarrow \Cam$ is given for $\xi
\in \CH$ by $(\AM_{s,t}^*\xi)(r) = G^*J^*_{r,t}\xi = G^*K_{r,t}\xi$ for
$r \in [s,t]$ and zero otherwise.

Recall that the Skorokhod integral $h \mapsto \int_0^t h(s)\cdot dW(s)
\eqdef \DM^*h$ is \textit{defined} as the adjoint of the Malliavin
derivative operator (or rather of the part acting on
$L^2(\Omega,\CF_t,\R)$ and not on $\CH$).  In other words, one has the
following identity between elements of $\CH$:
\begin{equ}[e:IBP]
  \E \DM_h \Phi_t^u = \E \scal{\DM \Phi_t^u,h}= \E \Bigl(\Phi_t^u
  \int_0^t h(s)\cdot dW(s)\Bigr)\;,
\end{equ}
for every $h \in L^2(\Omega, \Cam)$ belonging to the domain of
$\DM^*$.

It is well-established \cite[Ch.~1.3]{Nualart} that the Skorokhod integral
has the following two important properties:
\begin{claim}
\item[1.] Every adapted process $h$ with $\E \Norm{h}^2<
  \infty$ belongs to the domain of $\DM^*$ and the Skorokhod integral
  then coincides with the usual It\^o integral.
\item[2.] For non-adapted processes $h$, if $h(s)$ belongs to the
  domain of $\DM$ for almost every $s$ and is such that
\begin{equ}
  \E \int_0^t \int_0^t |\DM_s h(r)|_{\R^d}^2 \, ds\, = \E \int_0^t
  \Norm{\DM_s h}^2 \,ds\,   < \infty\;,
\end{equ}
then one has the following modification of the It\^o isometry:
\begin{multline}\label{e:genIto}
  \E \Bigl(\int_0^t h(s)\cdot dW(s)\Bigr)^2 = \E \int_0^t |h(s)|_{\R^d}^2\, ds
  \\+ \E \int_0^t \int_0^t \tr{\DM_s h(r) \DM_r h(s)} \, ds\, dr\;.
\end{multline}
Note here that since $h(s) \in \R^d$, we interpret $\DM_r h(s)$ as a
$d \times d$ matrix.
\end{claim}

\subsection{Malliavin derivative of the Jacobian}

By iterating the implicit functions theorem, we can see that the map
that associates a given realisation of the Wiener process $W$ to the
Jacobian $J_{s,t} \phi$ is also Fr\'echet (and therefore Malliavin)
differentiable. Its Malliavin derivative $\DM_h J_{s,t}\phi$ in the
direction $h\in \Cam$ is given by the unique solution to
\begin{equ}
  \d_t \DM_h J_{s,t}\phi = - L \DM_h J_{s,t}\phi + DN(u_t) \DM_h
  J_{s,t}\phi + D^2N(u_t) \bigl(\DM_h u_t,J_{s,t}\phi\bigr)\;,
\end{equ}
endowed with the initial condition $\DM_h J_{s,s}\phi = 0$. Just as
the Malliavin derivative of the solution was related to its derivative
with respect to the initial condition, the Malliavin derivative of
$J_{s,t}$ can be related to the second derivative of the flow with
respect to the initial condition in the following way. Denoting by
$J^{(2)}_{s,t}(\phi, \psi)$ the second derivative of $u_t$ with
respect to $u_0$ in the directions $\phi$ and $\psi$, we see that
as in \eref{eq:higher_variations},
$J^{(2)}_{s,t}(\phi, \psi)$ is the solution to
\begin{equ}
  \d_t J^{(2)}_{s,t}(\phi, \psi) = - L J^{(2)}_{s,t}(\phi, \psi) +
  DN(u_t) J^{(2)}_{s,t}(\phi, \psi)+ D^2N(u_t) \bigl(J_{s,t}\psi,
  J_{s,t}\phi\bigr)\;,
\end{equ}
endowed with the initial condition $J^{(2)}_{s,s}(\phi, \psi) = 0$.

Assuming that $h$ vanishes outside of the interval $[s,t]$ and using
the identities $J_{r,t} J_{s,r} = J_{s,t}$ and $\DM_h u_t = \int_s^t
J_{r,t} G h(r)\, dr$, we can check by differentiating both sides and
identifying terms that one has the identity
\begin{equ}[e:exprMallJac]
\DM_h J_{s,t}\phi = \int_s^t J^{(2)}_{r,t} (G h(r), J_{s,r}\phi)\, dr\;,
\end{equ}
which we can rewrite as
\begin{equ}[e:repJ2]
\DM_r J_{s,t}\phi = J^{(2)}_{r,t} (G , J_{s,r}\phi)
\end{equ}
This identity is going to be used in Section~\ref{sec:generalSmoothing}.

\subsection{Malliavin covariance matrix}

We now define and explore the properties of the Malliavin covariance
matrix, whose non-degeneracy is central to our constructions.

\begin{definition} Assume that the explosion time $\tau = \infty$ for
  every initial condition in $\CH$.  Then, for any $t >0$, the
  Malliavin matrix $\MM_t \colon \CH \rightarrow \CH$ is the linear
  operator defined by
  \begin{equ}[e:defM]
    \MM_t \phi = \sum_{k=1}^d \int_0^t \scal{J_{s,t} g_k,\phi}
    J_{s,t}g_k\,ds\;.
  \end{equ}
\end{definition}
Observe that this is equivalent to
\begin{align*}
  \MM_t=\AM_t\AM^*_t = \int_0^t J_{s,t} G G^* J_{s,t}^*\, ds =
  \int_0^t J_{s,t} G G^* K_{s,t}\, ds\;,
\end{align*}
thus motivating the definition $\MM_{s,t}=\AM_{s,t}\AM_{s,t}^*$ for
arbitrary time intervals $0\le s < t$. From this it is clear that
$\MM_{s,t}$ is a symmetric positive operator with
\begin{equ}[e:Mscalar]
  \scal{\MM_t \phi,\phi} = \sum_{k=1}^d \int_0^t \scal{J_{s,t}
    g_k,\phi}^2\,ds=\sum_{k=1}^d \int_0^t \scal{
    g_k,K_{s,t}\phi}^2\,ds
\end{equ}
for all $\phi \in \CH$.



The meaning of the Malliavin covariance matrix defined in
\eref{e:defM} is rather intuitive, especially for the diagonal
elements $\scal{\MM_t \phi,\phi}$. If $\scal{\MM_t \phi,\phi} >0$ then
there exists some variation in the Wiener process on the time interval
$[0,t]$ which creates a variation of $u_t$ in the direction $\phi$.

It is also useful to understand on what spaces the operator norm of
$\MM_t$ is bounded. As a simple consequence of Proposition~\ref{prop:bootstrapJ}, we have:
\begin{proposition}
 For every $T >0$ and $\gamma \in [0,(1-a)\wedge {1\over 2})$, 
 $\MM_T$ can be extended to a bounded
  (random) linear operator from $\CH_{-\gamma}$ to $\CH_\gamma$ with probability
  one. In particular, $\MM_T$ is almost surely a positive,
  self-adjoint linear operator on $\CH$ such that the bound
  \begin{equ}
    \sup_{\substack{\phi,\psi \in \CH_{-\gamma} \\
        \|\phi\|_{-\gamma}=\|\psi\|_{-\gamma}=1}} \scal{\MM_T \phi,\psi}
    \leq T\,C \sup_{0\leq s< t\leq T} \sup_k\;(1+\|u_t\|)^{2n-2}
    \|J_{s,t}g_k\|^2
  \end{equ}
  holds with some deterministic constant $C$.
\end{proposition}

\begin{remark}
If the linear operator $L$ happens to have compact resolvent, which will be the case in most 
of the examples to which our theory applies, then the operator $\CM_T$ is automatically compact, since
the embedding $\CH_\gamma \hookrightarrow \CH$ is then compact for every $\gamma > 0$.
\end{remark}

\begin{proof}
 From \eref{e:defM} we have that
 \begin{equ}
   \sup_{\substack{\phi,\psi \in \CH_{-\gamma}\\
       \|\phi\|_{-\gamma}=\|\psi\|_{-\gamma}=1}} \scal{\MM_t \phi,\psi} \leq
   \sum_{k=1}^d \int_0^t \|J_{s,t} g_k\|_\gamma^2 ds \;.
 \end{equ}
Since the $g_k$ belong to $\CH$ by assumption,
the required bound now follows from Proposition~\ref{prop:bootstrapJ},
noting that the singularity at $s = t$ is integrable by the assumption $\gamma < {1\over 2}$.
\end{proof}

\section{Smoothing in infinite dimensions}
\label{sec:generalSmoothing}




We now turn our study of \eref{e:SPDE2} to one of the principal goals
of this article. As in the preceding section, we shall assume that all solutions are
global in time and that the standing assumptions from Assumption
\ref{ass:basic} continue to hold.  The aim of this section is to prove
``smoothing'' estimates for the corresponding Markov semigroup $\CP_t$
whose action on bounded test functions $\phi:\CH \rightarrow \R$ is
defined by
\begin{equ}
  \CP_t \phi(v) = \E_v \phi(u_t)\;.
\end{equ}
Here, the subscript in the expectation refers to the initial condition
for the solution $u_t$ to \eref{e:SPDE2}.  We begin with a brief
discussion of the type of estimates we will prove and the ideas used
in their proof. A long discussion on this can be found in
\cite{HairerMattingly06AOM} in which a number of the tools of this
paper were developed or \cite{saintFlourJCM08} which has a longer
motivating discussion.

Recall also that the Malliavin covariance matrix $\MM_t \colon \CH \to
\CH$ for the solution to \eref{e:SPDE2} was defined in \eref{e:defM}
as $\MM_t=\AM_t\AM_t^*$ and that it is a random, self-adjoint
operator on $\CH$. Since $\CH$ is assumed to be infinite-dimensional,
$\MM_t$ will in general not be invertible. However as discussed in the
introduction we will only need it to be ``approximately invertible''
on some subspace paired with a assumption that the dynamics is
counteractive off this subspace. The assumption of ``approximate
invertibility'' on some subspace is formulated in
Assumption~\ref{ass:Malliavin}  below and the contractivity assumption is
formulated in Assumption~\ref{ass:smoothing}. These are the two
fundamental structural assumptions needed for this theory.  In between
the statement of these two assumption two other assumptions are
given. They are more technical in nature and ensure that we can
control various quantities.

\SetAssumptionCounter{B}

\begin{assumption}[Malliavin matrix]\label{ass:Malliavin}
  There exists a function $\U \colon \CH \to [1,\infty)$ and an orthogonal
  projection operator $\Pi \colon \CH \to \CH$ such that, for every
  $\alpha > 0$, the bound
  \begin{equ}[e:Malliavin]
    \P \Bigl(\inf_{\|\Pi \phi\| \ge \alpha \|\phi\|} {\scal{\phi,
      \MM_1\phi} \over \|\phi\|^2} \le \eps \Bigr) \le
    C(\alpha,p)\, \U^p(u_0)\, \eps^p\;,
  \end{equ}
  holds for every $\eps \le 1$, $p \geq 1$ and $u_0 \in \CH$. Furthermore for some
  $\bar q\geq 2$, there exist a constant $C_U$ so that for every initial
  condition $u_0 \in \CH$, the bound
  \begin{equ}
    \E U^{\bar q}(u_n) \leq C_U^{\bar q}U^{\bar q}(u_0)\;,
  \end{equ}
holds uniformly in $n \ge 0$.
\end{assumption}

We are also going to assume in this section that the solutions to \eref{e:SPDE2}
have the following Lyapunov-type  structure, which is stronger than
Assumption~\ref{ass:global} used in the previous section:

\begin{assumption}[Lyapunov structure]\label{ass:Lyapunov}
  Equation~\eref{e:SPDE2} has global solutions for every initial condition.
  Furthermore, there exists a function $V \colon \CH \to \R_+$ such that there
  exist constants $C_L>0$ and $\eta' \in [0,1)$ such that
  \begin{equ}[e:Lyap]
    \E \exp\bigl(V(u_1)\bigr) \le \exp\bigl(\eta' V(u_0) +
    C_L\bigr)\;.
  \end{equ}
\end{assumption}

\begin{assumption}[Jacobian]\label{ass:Jacobian}
  The Jacobian $J_{s,t}$ and the second variation $J_{s,t}^{(2)}$
  satisfy the bounds
  \begin{equs}
    \E \|J_{s,t}\|^{\bar p} &\le \exp\bigl(\bar p\eta V(u_0) + \bar p C_J\bigr)\;, \\
    \E \|J^{(2)}_{s,t}\|^{\bar p} &\le \exp\bigl(\bar p\eta V(u_0) +
    \bar p
    C_J^{(2)}\bigr)\;,
  \end{equs}
  for all $0 \le s \le t \le 1$ and for some constants $\bar p \geq 10$
  and $\eta> 0$ with $\bar p\eta < 1- \eta'$ and $2/\bar q + 10/\bar p
  \leq 1$, where $\eta'$ is the constant from
  Assumption~\ref{ass:Lyapunov} and $\bar q$ the constant from
  Assumption~\ref{ass:Malliavin}.
\end{assumption}

\begin{remark}
  When we write $\|J^{(2)}\|$ we mean the operator norm from $\CH
  \otimes \CH \rightarrow \CH$, namely $\sup_{\phi,\psi \in \CH}
  \|J^{(2)}(\phi,\psi)\|/(\|\phi\|\|\psi\|)$.
\end{remark}


We
finally assume that the Jacobian of the solution has some ``smoothing
properties'' in the sense that if we apply it to a function that
belongs to the image of the orthogonal complement $\Pi^\perp = 1 -
\Pi$ of the projection operator $\Pi$ then, at least for short times,
its norm will on average be reduced:

\begin{assumption}[Smoothing]\label{ass:smoothing}
  One has the bound
  \begin{equ}[e:smoothingJ]
    \E \|J_{0,1} \Pi^\perp\|^{\bar p} \le \exp(\bar p\eta V(u_0) -
    \bar p C_\Pi)\;,
  \end{equ}
  for some constant $C_\Pi$ such that $C_\Pi - C_J > 2\kappa
  C_L$ where $\kappa=\eta/(1-\eta')$.  The constants $\eta$ and $\bar p$ appearing in this
  bound are the same as the ones appearing in
  Assumption~\ref{ass:Jacobian}, the constant $\eta'$ is the same as
  the one appearing in Assumption~\ref{ass:Lyapunov}, and the
  projection $\Pi$ is the same as the one appearing in
  Assumption~\ref{ass:Malliavin}.
\end{assumption}

\begin{remark}
  The condition $C_\Pi - C_J > 2 \kappa C_L$ may seem
  particularly unmotivated. In the next section, we try to give some
  insight into its meaning.
\end{remark}

\begin{remark}
We will see in the proof of Theorem~\ref{theo:general} below that if we assume that 
the linear operator $L$ has compact resolvent, then Assumption~\ref{ass:smoothing}
can always be satisfied by taking for $\Pi$ the projection onto a sufficiently large number
of eigenvectors of $L$.
\end{remark}

\begin{remark}
  Notice that if $\Range(\Pi) \subset \SPAN\{g_1,\dots,g_d\}$, then in
  light of the last representation in \eref{e:Mscalar} it is
  reasonable to expect \eref{e:Malliavin} to hold as long as one has
  some control over moments the modulus of continuity of $s
  \mapsto K_{s,t}$.  (This is made more precise in Lemma
  \ref{lem:primer}.) We refer to such an assumption on the range as
  the ``essentially elliptic'' setting since all of the directions
  whose (pathwise) dynamics are not controlled by
  Assumption~\ref{ass:smoothing} are directly forced.
\end{remark}

Under these assumptions we have the following result which is the
fundamental ``smoothing'' estimate of this paper. It is the linchpin
on which all of the ergodic results rest. 
\begin{theorem}\label{thm:asfEstimate} Let
  Assumptions~\ref{ass:basic} and \ref{ass:Malliavin}-\ref{ass:smoothing} hold. Then for
  any $\zeta \in [0,(C_\Pi - C_J)/2 - \kappa C_L)$ there a
  exist positive constants $C$ such that for all $n \in
  \N$ and measurable $\phi\colon\CH \rightarrow \R$ 
  \begin{align}\label{eq:asfSmooth}
    \|\DF(\CP_{2n}\phi)(u)\| \leq e^{4\kappa V(u_0)}\Big( CU^2(u_0)
    \sqrt{(\CP_{2n} \phi^2)(u)} + \gamma ^{2n}\sqrt{(\CP_{2n}
      \|\DF\phi\|^2)(u)}\Big)\;
\end{align}
where $\gamma=\exp(-\zeta)$.
\end{theorem}

\begin{remark}
By $(\CP_t \|\DF\phi\|^2)(u)$, we simply mean
$  \EE_u \bigl(\sup_{\| \xi \| =1 } \big|(\DF\phi)(u_t)\xi\big|^2\bigr)$.
\end{remark}
\begin{remark}\label{rem:getInfinityNorm} If $\|\phi\|_\infty$ or
  $\|D\phi\|_\infty$ are bounded by one then the 
  corresponding terms under the square root are bounded by one. Furthermore, in
  light of Assumption~\ref{ass:Lyapunov}, if $\phi(u)^2 \leq
  \exp(V(u))$, then
  \begin{equ}
    \sqrt{\CP_{2n} \phi^2(u)} \leq \|\phi\|_\infty \sqrt{\E
      \exp(V(u_{2n}))} \leq \|\phi\|_\infty\exp(\eta' V(u_0)/2 + C_L/(2-2\eta'))\;.
  \end{equ}
  Of course, the same bound for holds for $\sqrt{(\CP_{2n}
    \|\DF\phi\|^2)(u)}$, provided that one has an estimate of the type
  $\|D\phi\|^2(u) \leq \exp(V(u))$.
\end{remark}

\subsection{Motivating discussion}

We now discuss in what sense \eqref{eq:asfSmooth} implies
smoothing. When the term ``smoothing'' is used in the mathematics
literature to describe a linear operator $T$, it usually means that
$T\phi$ belongs to a smoother function space than $\phi$. This usually means
that $T\phi$ is ``more differentiable'' then $\phi$. A convenient way
to express this fact analytically would be an estimate of the form
\begin{equation}\label{eq:inftyOnly}
  \|\DF(T\phi)(u)\| \leq C(u)\|\phi\|_\infty\;.
\end{equation}
(Of course the ``smoothing'' property may improve the smoothness by
less than a whole derivative, or one may consider functions $\phi$
that are not bounded, but let us consider \eref{eq:inftyOnly} just for
the sake of the argument.)  This shows in a quantitative way that
$T\phi$ is differentiable while $\phi$ need not be. In light of Remark
\ref{rem:getInfinityNorm}, this is in line with the first term on the
right hand side of \eqref{eq:asfSmooth}.

The second term on the right hand side of \eqref{eq:asfSmooth} embodies
smoothing of a different type. Suppose that $T$ satisfies the estimate
\begin{equation}\label{eq:LY}
  \|\DF (T\phi)\|_\infty \leq C \|\phi\|_\infty + \gamma \|\DF\phi\|_\infty 
\end{equation}
for some positive $C$ and some $\gamma \in (0,1)$. (Note that this is
a variation of what is usually referred to as the Lasota-Yorke
inequality \cite{LY73OTE,Liv03IMA} or the Ionescu-Tulcea-Marinescu
inequality \cite{IonMar52TEP}.) Though \eqref{eq:LY} does not imply
that $T\phi$ belongs to a smoother function space then $\phi$, it does
imply that the gradients of $T\phi$ are smaller then those of $\phi$,
at least as long as the gradients of $\phi$ are sufficiently steep.
This is in line with a more colloquial idea of smoothing, though not
in line with the traditional mathematical definition used.


\subsubsection{Strongly dissipative setting}\label{sec:stronglyDissipative}

Where does the assumption $C_\Pi > C_J + 2\kappa C_L$ come
from?  This is easy to understand if we consider the ``trivial'' case
$\Pi = 0$. In this case, Assumption~\ref{ass:Malliavin} is empty
and the projection $\Pi^\perp$ is the identity. Therefore, the left
hand sides from Assumptions~\ref{ass:Jacobian} and \ref{ass:smoothing}
coincide, so that one has $C_J = - C_\Pi$ and our restriction becomes
$C_J + \kappa C_L < 0$.

This turns out to be precisely the right condition to impose if one
wishes to show that $\E \|J_{0,n}\| \to 0$ at an exponential rate:
\begin{proposition}\label{prop:JboundIn_n} 
  Let Assumptions~\ref{ass:Lyapunov} and \ref{ass:Jacobian}
  hold. Then, for any $p \in [0,\bar p/2]$, one has the bound
  \begin{equ}
    \E \|J_{0,n}\|^p \le \exp \bigl(p\kappa V(u_0) + pC_Tn\bigr)\;,
  \end{equ}
  with $\kappa = \eta / (1-\eta')$ and $C_T = C_J + \kappa C_L$.
\end{proposition}

\begin{proof}
  Using the fact that $\|J_{0,n}\| \le \|J_{n-1,n}\| \|J_{0,n-1}\|$,
 we have the following recursion relation:
  \begin{equs}
    \E \bigl( &\exp\bigl(p \kappa V(u_n)\bigr) \|J_{0,n}\|^p \bigr) \le
    \E\Bigl( \E\bigl( \exp\bigl(p \kappa V(u_n)\bigr) \|J_{n-1,n}\|^p\,|\, \CF_{n-1} \bigr)  \|J_{0,n-1}\|^p\Bigr) \\
 &\le   \E\Bigl( \Bigl(\E\bigl(\|J_{n-1,n}\|^{\bar p}\,|\, \CF_{n-1} \bigr)\Bigr)^{p \over \bar p}
\Bigl(\E\Bigl( \exp\Bigl({p \bar p \over \bar p - p} \kappa V(u_n)\Bigr)\,\Big|\, \CF_{n-1} \Bigr)\Bigr)^{\bar p - p \over \bar p}  \|J_{0,n-1}\|^p\Bigr)\\
&\le    e^{p C_T} \E \bigl( \exp\bigl(p\kappa V(u_{n-1})\bigr) \|J_{0,n-1}\|^p
    \bigr)\;,
  \end{equs}
where we made use of Assumptions~\ref{ass:Lyapunov} and \ref{ass:Jacobian} in
the second inequality.
  It now suffices to apply this $n$ times and to use the fact that
  $\|J_{0,0}\| = 1$. The assumptions $\bar p \kappa < 1$ and $p \le \bar p/2$ ensure that
  $p \bar p \le \bar p - p$ so that the bound  \eref{e:Lyap} can be used.
\end{proof}

We now use this estimate to prove a version of
Theorem~\ref{thm:asfEstimate} when the system is strongly
dissipative:

\begin{proposition}\label{prop:strongDissASF}
Let  Assumptions~\ref{ass:Lyapunov} and \ref{ass:Jacobian} hold and
set $C_T = C_J + \kappa C_L$ with $\kappa = \eta/(1-\eta')$ as before. 
Then, for any $\phi:\CH \rightarrow \R$ and $n \in \N$ one has
  \begin{align*}
    \|\DF(\CP_n\phi)(u)\| \leq \gamma^n e^{\kappa V(u)} \sqrt{\CP_n \|\DF\phi\|^2(u)} \;.
  \end{align*}
with $\gamma=e^{C_T}$. In particular, the semigroup $\CP_t$ has the asymptotic strong Feller property whenever
$C_T < 0$.
\end{proposition}
\begin{proof}
  Fixing any $\xi \in \CH$ with $\|\xi\|=1$, observe that
  \begin{align*}
    \DF(\CP_t\phi)(u)\xi = \E_u (\DF\phi)(u_t)J_{0,t}\xi \leq
    \sqrt{\E\|J_{0,t}\|^2}
    \sqrt{\E\|\DF\phi\|^2(u_t)}\;.
  \end{align*}
Applying Proposition~\ref{prop:JboundIn_n} completes the proof.
\end{proof}

Comparing this result to the bound \eref{eq:asfSmooth} stated in Theorem~\ref{thm:asfEstimate}
shows that, the combination of the smoothing Assumption~\ref{ass:smoothing}
with Assumption~\ref{ass:Malliavin} on the Malliavin matrix
allows us to consider the system as if its Jacobian was contracting at an
average rate $(C_\Pi - C_J)/2$ instead of expanding at a rate
$C_J$. This is precisely the rate that one would obtain by projecting the
Jacobian with $\Pi^\perp$ at every second step.
The additional term containing $\CP_{2n}\phi^2$ appearing in the right hand side
of \eref{eq:asfSmooth} should then be interpreted as the probabilistic ``cost'' of performing
that projection. Since this ``projection'' will be performed by using an approximate inverse
to the Malliavin matrix, it makes sense that the larger the lower bound on $\CM_t$ is, the lower
the corresponding probabilistic cost.

\begin{remark}
  It is worth mentioning, that nothing in this section required that
  the number of Wiener process be finite. Hence one is free to take
  $d=\infty$, as long as all of the solutions and linearization are
  well defined (which places conditions on the $g_k$).
\end{remark}
\subsection{Transfer of variation}
\label{sec:general-framework}

Having analyzed the strongly dissipative setting, we now turn to the
general setting. We would like to mimic the calculation used in
Proposition~\ref{prop:strongDissASF}, but we do not want to require
the system to be ``contractive'' in the sense of being strongly
dissipative. However, in settings where one can prove
\eqref{eq:inftyOnly} there is usually no requirement of strong
dissipativity but rather an assumption of hypoellipticity. This is
because the variation in the initial condition is transferred to a
variation in the Wiener space. Mirroring the discussion in
\cite{saintFlourJCM08,HairerMattingly06AOM} (where more details can be
found), we begin sketching a proof of \eqref{eq:inftyOnly} and then
show how to modify it to obtain \eqref{eq:LY}. The central idea is to
compensate as much as possible the effect of an infinitesimal
perturbation in the initial condition to an infinitesimal variation in
the driving Wiener process. In short, to transfer one type of
variation to another.

Denoting by $\CS=\{ \xi \in \CH: \|\xi\|=1\}$ the set of possible
directions in $\CH$, let there be given a map from $\CS \times
\CC([0,\infty), \R^d) \rightarrow \Cam$ denoted by $(\xi,W) \mapsto
h^{\xi}(W)$, mapping variations in the initial condition $u$ to
variations in the Wiener path $W$. We will worry about constructing a
suitable map in the next sections; for the moment we just explore
which properties of $h^\xi$ might be useful. Fixing $t$, let us begin
by assuming that the following identity holds:
\begin{equation}\label{eq:compensateVariation}
 \DF_\xi\Phi_t^u(W)=\scal{\DF\Phi_t^u(W),\xi} = \scal{\DM\Phi_t^u(W),h^\xi(W)}=\DM_{h^\xi}\Phi_t^u(W)\;.
\end{equation}
(The first and last equalities are just changes in notation.)  Here, $\DF_\xi$ denotes derivative with respect to the
initial condition in the direction $\xi \in \CH$, while $\DM$ denotes the (Malliavin) 
derivative with respect to the noise. In words,
the middle equality states that the variation in $u_t(W)$ caused by an
infinitesimal shift in the initial condition in the direction $\xi$ is
equal to the variation in $u_t$ caused by an infinitesimal shift of the
Wiener process $W$ in the direction $h^\xi(W)$. This is the basic
reasoning behind smoothness estimates proved by Malliavin calculus. We
begin as in the proof of Proposition~\ref{prop:strongDissASF}. For any
$\xi \in \CS$, one has that
\begin{equ}\label{eq:exactIBP}
 \DF_\xi\big(\phi(\Phi_t^u)\big)=(\DF\phi)(\Phi_t^u)\DF_\xi\Phi_t^u 
=(\DF\phi)(\Phi_t^u)\DM_{h^\xi}\Phi_t^u= \DM_{h^\xi}\big(\phi(\Phi_t^u)\big)\;.
\end{equ}
Taking expectations and using the Malliavin integration by parts
formula \eref{e:IBP} to obtain the last equality yields
\begin{equ}
  \DF_\xi\CP_t \phi(u)=\EE
  \DF_\xi\big(\phi(\Phi_t^u)\big)=\EE\DM_{h^\xi}\big(\phi(\Phi_t^u)\big)=\E_u
  \phi(\Phi_t^u) \int_0^t h_s^\xi\cdot dW(s)\;.
\end{equ}
Applying the Cauchy-Schwartz inequality to the last term
produces a term of the form of the first term on the right-hand side of
\eqref{eq:asfSmooth} provided $\EE |\int_0^t h_s^\xi\cdot dW(s)|^2 <
\infty$. Taken alone, provided one can find a mapping $(\xi,W) \mapsto
h^{\xi}(W)$ satisfying \eqref{eq:compensateVariation} with $\EE|
\int_0^t h_s^\xi\cdot dW(s) | < \infty$, we have proven an
inequality of the form \eqref{eq:inftyOnly}.

In the infinite-dimensional SPDE setting of this
paper, finding a map $(\xi,W) \mapsto h^{\xi}(W)$ satisfying
\eqref{eq:compensateVariation} seems hopeless, unless the noise is infinite-dimensional
itself and acts in a very non-degenerate way on the equation, see \cite{Mas89:210,DaPElwZab95:35,EH3}
or the monograph \cite{ZDP} for some results in this direction. 
Instead, we only ``approximately
compensate'' for the variation due to differentiating in the initial
direction $\xi$ with a shift in the Wiener process. As such, given an
mapping $(\xi,W) \mapsto h^{\xi}(W)$, we replace the requirement in
\eqref{eq:compensateVariation} with the definition
\begin{equation}
  \label{eq:aproxCompVar}
  \rho_t(W)= \DF_\xi\Phi_t^u(W)-\DM_{h^\xi}\Phi_t^u(W) 
\end{equation}
and hope that we can choose $h^\xi$ in such a way that $\rho_t \rightarrow
0$ as $t \rightarrow \infty$. As before, we postpone choosing a
mapping $(\xi,W) \mapsto h^{\xi}(W)$ until the next section. For the
moment we are content to explore the implications of finding such a
mapping with desirable properties.

Returning to \eqref{eq:exactIBP} but using
\eqref{eq:aproxCompVar}, we now have
\begin{equation}\label{eq:aproxIBP}
  \begin{aligned}
    \DF_\xi\big(\phi(\Phi_t^u)\big)&=(\DF\phi)(\Phi_t^u)\DF_\xi\Phi_t^u
    =(\DF\phi)(\Phi_t^u)\DM_{h^\xi}\Phi_t^u + (\DF\phi)(\Phi_t^u)\rho_t\\
    &= \DM_{h^\xi}\big(\phi(\Phi_t^u)\big)+ (\DF\phi)(\Phi_t^u)\rho_t\;.
  \end{aligned}  
\end{equation}
Taking expectations of both sides and applying the Malliavin integration by parts the first term on the right-hand side  produces
\begin{equation*}
 \DF_\xi\CP_t \phi(u)=\EE
  \DF_\xi\big(\phi(\Phi_t^u)\big)=\EE\phi(\Phi_t^u)\int_0^t h^\xi_s\cdot dW(s)+
 \EE (\DF\phi)(\Phi_t^u) \rho_t 
\end{equation*}
which in turn, after application of the Cauchy-Schwartz inequality twice,  yields
\begin{align}\label{eq:asfPrototype}
  \|\DF\CP_t \phi(u)\| \leq  C(t)  \sqrt{(\CP_t \phi^2)(u)} + \Gamma(t) \sqrt{(\CP_t \|\DF\phi\|^2)(u)}
\end{align}
with $C(t)=\sqrt{\EE\big|\int_0^t
h^\xi_s\cdot dW(s)\big|^2}$ and $\Gamma(t)=\sqrt{\EE
|\rho_t|^2}$.  Observe that provided that
\begin{equ}[e:boundCGamma]
  \limsup_{n \in\N} C(n) < \infty \quad\text{and}\quad \limsup_{n \in
    \N} {\Gamma(n)}{\gamma^{-n}} < \infty
\end{equ}
for some $\gamma \in (0,1)$ we will have proved
Theorem~\ref{thm:asfEstimate}. Choosing a mapping $(\xi,W) \mapsto
h^{\xi}(W)$ so that these two conditions hold is the topic of the next four sections.



\subsection[Choosing a variation $h_t^\xi$]{Choosing a variation $\pmb{h_t^\xi}$} \label{sec:choiceh}

As discussed in \cite{HairerMattingly06AOM} and at length in
\cite{saintFlourJCM08}, if one looks for the variation $h^\xi$ such
that \eqref{eq:compensateVariation} holds and $\int_0^t|h_s^\xi|^2 ds$
is minimized, then the answer is $h_s^\xi = (\AM_{t}^* \MM_{t}^{-1}
J_t\xi)(s)$ which by the observation in \eqref{eq:Aadj} is simply
$h_s^\xi = G^*K_{s,t}\MM_{t}^{-1} J_t\xi$. While this is not quite the
correct optimisation problem to solve since its solution $h^\xi$ is
not adapted to $W$ and hence $\EE |\int_0^t h_s^\xi\cdot dW(s)|^2 \neq
\int_0^t \EE|u_s|^2 ds$, it is in general a good enough choice.

A bigger problem is that the space on which $\MM_t$ can be inverted is
far from evident. If the range of $G$ was dense in $\CH$ (which
requires infinitely many driving Wiener processes), then there is some
chance that $\Range(J_t) \subset \Range(\MM_t)$ and the above formula
for $h_t$ could be used. This is in fact the case where the
Bismut-Elworthy-Li formula is often used and which might be refereed
to as ``truly elliptic.'' It this case the system \textit{is} in fact
strong Feller. We are precisely interested in the case when only a
finite number of directions are forced (or the variance decays so
fast that this is effectively true). One of the fundamental ideas used
in this article is that we need only effective control of the system on a finite
dimensional subspace since the dynamic pathwise control embodied in
Assumption~\ref{ass:smoothing} can control the remaining degrees of freedom.

 While Theorem~\ref{theo:Malliavin} of the next section gives conditions
that ensure that $\MM_t$ is almost surely non-degenerate, it does not
give much insight into the structure of the range since it only deals
with finite dimensional projections. However,
Assumption~\ref{ass:Malliavin} ensures that it is unlikely the
eigenvectors with sizable projection in $\Pi \CH$ have small
eigenvalues. As long as this is true, the ``regularised inverse''
$(\MM_t +\beta)^{-1}$, which always exists since $\MM_t$ is positive
definite, will be a ``good inverse'' for $\MM_t$, at least on $\Pi
\CH$. This suggests that we make the choice $h_s^\xi =
G^*K_{s,t}(\MM_{t}+ \beta)^{-1} J_t\xi$ for some very small
$\beta>0$. Observe that
\begin{equ}[e:exprRho]
\DF_\xi u_t- \DM_{h^\xi} u_t = J_{t}\xi -
\MM_t(\MM_{t}+ \beta)^{-1} J_t\xi= \beta(\MM_{t}+ \beta)^{-1} J_t\xi\;,
\end{equ}
which will be expected to be small as long as $J_t\xi$ has small projection (relative
to the size of $\beta$) in $\Pi^\perp \CH$. But in any case, the norm of the right hand side in \eref{e:exprRho}
will never exceed the norm of $J_t\xi$, so that for small values of $\beta$,
$\|\DF_\xi u_t- \DM_{h^\xi} u_t\|$ is expected to behave like $\|\Pi^\perp J_t\xi\|$.

Assumption~\ref{ass:smoothing} precisely states that if one projects the Jacobian onto
$\Pi^\perp\CH$, then the system behaves as if it was ``strongly dissipative'' as in
Section~\ref{sec:stronglyDissipative}. All together, this motivates
alternating between choosing $h^\xi =\AM_{n,n+1}^*
(\MM_{n,n+1}+\beta_n)^{-1}J_{n,n+1}\rho_n$ for even $n$ and $h^\xi\equiv 0$ on
$[n,n+1]$ for odd $n$.

Since we will split time into intervals of length one, we introduce the
following notations:
\begin{equ}
  J_n = J_{n, n+1}\;,\quad \AM_n = \AM_{n,n+1}\;,\quad \MM_n = \MM_{n,n+1}\;.
\end{equ}
We then define the map $(\xi,W) \mapsto h^{\xi}(W)$ recursively by
\begin{equ}[e:defh]
  h_s^\xi =
  \begin{cases}
    (\AM_{2n}^* (\beta_{2n}+ \MM_{2n})^{-1}
J_t\rho_{2n})(s) & \text{for } s \in [ 2n, 2n + 1) \text{ and } n \in \N\;,\\
0 &   \text{for } s \in [ 2n-1, 2n) \text{ and } n \in \N\;.\\
  \end{cases}
\end{equ}
Here, as before, $\rho_0 = \xi$, $\rho_t= J_{0,t}\xi - \AM_{0,t}h_s^\xi = \DF_\xi u_t -
\DM_{h^\xi} u_t$, and $\beta_n$ is a sequence of positive random
numbers measurable with respect to $\CF_n$ which will be
chosen later. 

Observe that these definitions are not circular since the construction
of $h^\xi_s$ for $s \in [n,n+1)$ only requires the knowledge of $\rho_n$, which
in turn depends only on $h^\xi_s$ for $s \in [0,n)$. 
The remainder of this section is devoted to showing that this particular choice of
$h^\xi$ is ``good'' in the sense that it allows to satisfy \eref{e:boundCGamma}.
We are going to assume throughout this section that Assumptions~\ref{ass:basic} and
\ref{ass:Malliavin}-\ref{ass:smoothing} hold, so that we are in the setting of Theorem~\ref{thm:asfEstimate},
and that $h^\xi$ is defined
as in \eref{e:defh}.

\subsection{Preliminary bounds and  definitions}
\label{sec:boundDef}

We start by a stating a few straightforward consequences of Assumption~\ref{ass:Lyapunov}:
\begin{proposition}\label{prop:Lyop}
For any $\alpha \leq 1$, one has the bound
  \begin{equ}
    \E \exp\bigl(\alpha V(u_1)\bigr) \le \exp\bigl(\alpha \eta' V(u_0) +
    \alpha C_L\bigr)\;.
\end{equ}
Furthermore, for $\eta>0$ and $p>0$ such that $\eta p \le 1$, one has 
\begin{equ}
    \E \exp( \eta p V(u_n)) \leq \exp( p \eta (\eta')^n V(u_0) +
    p\kappa C_L)\;.
\end{equ}
Finally, setting $\kappa = \eta/(1-\eta')$ as before, one has the bound
\begin{equ}
    \E \exp\Big( \eta p \sum_{k=0}^n V(u_k)\Big) 
    \leq \exp( p\kappa V(u_0) + p \kappa C_L n)\;,
\end{equ}
provided that $\kappa p \le 1$.
\end{proposition}
\begin{proof}
The first bound follows immediately from Jensen's inequality.
  The second and third inequalities are shown by rewriting the estimate
  from Assumption~\ref{ass:Lyapunov} as
  \begin{equ}
    \E\big( \exp(\eta p V(u_n) )| \CF_{n-1}\big) \leq \exp\big( \eta
    p\eta' V(u_{n-1}) +\eta p C_L\big)\;,
  \end{equ}
and iterating it.
\end{proof}

Similarly, we obtain a bound on the Jacobian and on the Malliavin derivative $\AM_n$ 
of the solution flow between times $n$ and $n+1$: 
\begin{proposition}\label{prop:Jcontrol}
For any $p \in [0,\bar p]$, one has
  \begin{equs}
    \sup_{n \le s < t \le n+1} \E \|J_{s,t}\|^p &\leq \exp
    \bigl(p(\eta')^{n} \eta V(u_0) + pC_J + p \kappa C_L \bigr) \label{e:boundJn}\\
    \E\|\AM_{n}\|^p &\leq \|{G}\|^p\exp(p\eta
    (\eta')^{n} V(u_{0}) + p\kappa C_L+ p C_J)\;. \label{e:boundAn}
  \end{equs}
Furthermore, \eref{e:boundJn} also holds for $J_{s,t}^{(2)}$ with $C_J$ replaced by $C_J^{(2)}$.  
\end{proposition}
\begin{proof}
  We only need to show the bound for $p= \bar p$, since lower values follow again from
  Jensen's inequality. The bound \eref{e:boundJn} is an immediate consequence of 
  Assumption~\ref{ass:Lyapunov} and Proposition~\ref{prop:Lyop}. The second bound
  follows by writing
  \begin{equs}
    \|\AM_{n}h\|^p &= \Bigl\|\int_n^{n+1} J_{r,n+1}G h_r dr\Bigr\|^p \\
    & \leq
    \|{G}\|^p \Big(\int_n^{n+1}\|J_{r,n+1}\|^2dr\Big)^{\frac{p}2}
    \Big(\int_n^{n+1}|h_r|^2 dr\Big)^{\frac{p}2}\\
    &\leq\|{G}\|^p\Big(\int_n^{n+1}\|J_{r,n+1}\|^{p}dr\Big)
    \Norm{h}_n^p\;,
  \end{equs}
and then applying the first bound.
\end{proof}
In addition to these first Malliavin derivatives, we will need the
control of the derivative of various objects involving the Malliavin
derivative. The following lemma gives control over two objects
related to the second Malliavin derivative:
\begin{lemma}\label{lem:DJ_DA_Control} For all $p \in [0,\bar p/2]$, one has the bounds
  \begin{align*}
    \sup_{s,r \in [n,n+1]}\E\|\DM_s^i J_{r,n+1}\|^p& \leq \exp(2p\eta (\eta')^{n} V(u_{0}) +
    2p\kappa C_L + p C_J + pC_J^{(2)})\;, \\
      \sup_{s \in [n,n+1]}\E\|\DM_s^i \AM_{n}\|^p &\leq \Norm{G}^p\exp(2p\eta (\eta')^{n} V(u_{0}) +
    2p\kappa C_L + p C_J + pC_J^{(2)})\;.
  \end{align*}
\end{lemma}
\begin{proof}
 For this, we note that by \eref{e:repJ2} one has the identities
  \begin{equs}
    \DM_s^i J_{r,n+1}\xi &= \left\{\begin{array}{rl} J_{s,n+1}^{(2)}(J_{r,s}\xi, g_i) & \text{for $r \le s$,} \\[0.5em] J_{r,n+1}^{(2)}(J_{s,r}g_i, \xi)
    & \text{for $s \le r$,} \end{array}\right. \\
    \qquad \DM_s^i \AM_{n} v &= \int_{n}^{n+1} \DM_s^i
    J_{r,n+1}Gv_r\, dr\;.
  \end{equs}
  Hence if $p \in [0,\bar p/2]$ (which by the way also ensures that $2p \kappa <
  1$) it follows from Proposition~\ref{prop:Jcontrol} that
\begin{align*}
  \E\|\DM_s^i J_{r,n+1}\|^p& \leq \bigl(\E\|J_{s,n+1}^{(2)}\|^{2p}\,
\E\|J_{r,s}\|^{2p}\bigr)^{\frac{1}{2}}
  \leq \E \exp(2p \eta V(u_{n}) + p C_J + pC_J^{(2)})\\
  & \leq \exp(2p\eta (\eta')^{n} V(u_{0}) + 2p\kappa C_L + p C_J +
  pC_J^{(2)})
\end{align*}
for $r \le s$ and similarly for $s \le r$. Since, for $p \ge 1$, we can write
\begin{align*}
  \E\|\DM_s^i \AM_n\|^p &\leq\|{G}\|^p \int_{n}^{n+1} \E\|\DM_s^i J_{r,n+1}\|^{p}
  dr\;,
\end{align*}
the second estimate then follows from the first one.
\end{proof}

\subsection[Controlling the error term $\rho_t$]{Controlling the error
  term $\pmb{\rho_t}$}

The purpose of this section is to show that the ``error term''
$\rho_t= \DF_\xi u_t - \DM_{h^\xi} u_t$ goes the zero as $t
\rightarrow \infty$, provided that the ``control'' $h^\xi$ is chosen
as explained in Section~\ref{sec:choiceh}.  We begin by observing that
for even integer times, $\rho_n$ is given recursively by
\begin{equ}[e:defrho]
  \rho_{2n+2} = J_{2n+1} \rho_{2n+1} = J_{2n+1} \RM^{\beta_{2n}}_{2n}
  J_{2n} \rho_{2n}\;,
\end{equ}
where $\RM_k^\beta$ is the operator
\begin{equ}
  \RM_k^\beta \eqdef 1 - \MM_{k} ( \beta + \MM_{k})^{-1}= \beta
  (\beta + \MM_{k})^{-1}\;.
\end{equ}
Observe that $\RM_k^\beta$ measures the error between $\MM_{k} (
\beta + \MM_{k})^{-1}$ and the identity, which we will see is small
for $\beta$ very small.  This recursion is of the form $\rho_{2n+2} =
\Xi_{2n+2} \rho_{2n}$, with the (random) operator $\Xi_{2n+2} \colon \CH
\to \CH$ defined by $\Xi_{2n+2}= J_{2n+1} \RM^{\beta_{2n}}_{2n} J_{2n}
$. Notice that $\Xi_{2n}$ is $\cF_{2n}$-measurable and that $\Xi_k$ is defined only for even integers $k$. 
Define the $n$-fold
product of the $\Xi_{2k}$ by
\begin{align*}
  \Xi^{(2n)}= \prod_{k=1}^n \Xi_{2k}\;,
\end{align*}
so that $\rho_{2n} = \Xi^{(2n)} \xi$.

It is our aim to show that it is possible under the assumptions of
Section~\ref{sec:generalSmoothing} to choose the sequence $\beta_n$ in
an adapted way such that for a sufficiently small constant $\bar \eta$
and $p \in [0,\bar p/2]$ one has
\begin{equ}[e:boundXi]
  \EE\|\rho_{2n}\|^p\leq\E\bigl( \|\Xi^{(2n)}\|^{p}\bigr)\|\rho_0\|^p
  \le \exp\bigl(p\bar \eta V(u_{0}) - p n
  \tilde\kappa\bigr)\|\rho_0\|^p \;.
\end{equ}
for some $\tilde \kappa >0$. This will give the needed control over
the last term in \eqref{eq:asfPrototype}.

By Assumption~\ref{ass:Malliavin}, we have a bound on the Malliavin
covariance matrix of the form
\begin{equ}[e:Malliavin2]
  \P \Bigl(\inf_{\|\Pi \phi\| \ge \alpha \|\phi\|} \scal{\phi,
    \MM_k\phi} \le \eps \|\phi\|^2 \,\Big|\, \CF_k\Bigr) \le C(\alpha,
  p)\, \U^p(u_k)\, \eps^p\;. 
\end{equ}
Here, by the Markov property, the quantities $\eps$ and $\alpha$ do not 
necessarily need to be constant, but are allowed to be
$\cF_k$-measurable random variables.

In order to obtain \eref{e:boundXi}, the idea is to decompose
$\Xi_{2n+2}$ as
  \begin{equs}[e:decompos]
    \Xi_{2n+2} &= J_{2n+1} \RM^{\beta_{2n}}_{2n} J_{2n} = \bigl(J_{2n+1}
    \Pi^\perp\bigr)\RM^{\beta_{2n}}_{2n} J_{2n}+J_{2n+1} \bigl(\Pi
    \RM^{\beta_{2n}}_{2n}\bigr) J_{2n}\qquad \\&\eqdef I_{2n+2,1} + I_{2n+2,2}\;.
  \end{equs}
  The crux of the matter is controlling the term $\Pi
  \RM^{\beta_{2n}}_{2n}$ since $J_{2n+1} \Pi^\perp$ is controlled by
  Assumption \ref{ass:smoothing} and we know that
  $\|\RM^{\beta_{2n}}_{2n}\| \le 1$.  To understand and control the
  $I_{2n+2,2}$ term, we explore the properties of a general operator of
  the form of $\RM_{2n}^\beta$.

\begin{lemma}\label{lem:MBound}
  Let $\Pi$ be an orthogonal projection on $\CH$ and $M$ be a
  self-adjoint, positive linear operator on $\CH$  satisfying for some
  $\gamma >0$ and $\delta \in (0,1]$
  \begin{equ}[eq:smallEV]
    \inf_{\xi \in \Lambda_\delta} {\scal{M \xi,\xi} \over \|\xi\|^2}
    \geq \gamma \;,
    \end{equ}
    where $\Lambda_\delta =\{ \xi : \|\Pi \xi \| \geq \delta
    \|\xi\|\}$.  Then, defining $R=1 - M(\beta +M)^{-1}= \beta (\beta
    + M)^{-1}$ for some $\beta >0$, one has $\|\Pi R\| \leq \delta
    \vee \sqrt{\beta/\gamma}$.
\end{lemma}
\begin{proof} Since $\|R\| \leq 1$, for $R\xi \in \Lambda_\delta^c$
  one has
  \begin{align*}
    \frac{\|\Pi R \xi\|}{\|\xi\|} \leq \frac{\|\Pi R \xi\|}{\| R
      \xi\|} \leq \delta\;.
  \end{align*}
Now for $R\xi \in \Lambda_\delta$, we have by assumption \eref{eq:smallEV}
\begin{align*}
  \gamma\frac{\|\Pi R\xi\|^2}{ \|\xi\|^2} \le \gamma\frac{\|R\xi\|^2}{
    \|\xi\|^2} \leq \frac{\scal{MR\xi,R\xi}}{\|\xi\|^2} \leq
  \frac{\scal{(M+\beta)R\xi,R\xi}}{\|\xi\|^2} = \beta
  \frac{\scal{\xi,R\xi}}{\|\xi\|^2} \leq \beta\;.
\end{align*}
Combining both estimates gives the required bound.
\end{proof}

This result can be applied almost directly to our setting in the
following way:

\begin{corollary}\label{lem:Rsmall}
  Let $M(\omega)$ be a random operator satisfying the conditions of
  Lemma \ref{lem:MBound} almost surely for some random variable
  $\gamma$. If we choose $\beta$ such that, for some (deterministic)
  $\delta \in (0,1)$ , $p\geq 1$ and $C >0$, one has the bound $\P (
  \beta \geq \delta^2 \gamma ) \leq C \delta^p$, then $\E \|\Pi R\|^p
  \leq(1+C) \delta^p$.
  
  In particular, for any $\delta \in (0,1)$, setting
  \begin{align}\label{e:betaChoice}
    \beta_k= \frac{\delta^3}{ \U(u_{k}) C(\delta, \bar p)^\frac1{\bar
        p}}\;,
  \end{align}
  where $C$ is the constant from \eref{e:Malliavin2}, produces the
  bound $\E \big( \|\Pi \RM_{2n}^{\beta_{2n}}\|^p | \CF_{2n}\big) \leq 2
  \delta^p$, valid for every $p \le \bar p$.
\end{corollary}
\begin{proof}
  To see the first part define $\Omega_0=\{ \omega : \beta(\omega)
  \leq \delta^2 \gamma(\omega) \}$. It the follows from Lemma
  \ref{lem:MBound}, the fact that $\|\Pi R\| \leq 1$ and the
  assumption $\P(\Omega_0^c) \leq C \delta^p$, that
  \begin{equ}[e:boundPiRgeneral]
    \EE \|\Pi R\|^p \leq \E \bigg(\Big(\delta \vee
    \sqrt\frac\beta\gamma\Big)^p\one_{\Omega_0} +\one_{\Omega_0^c} \bigg)
    \leq \delta^p + \P(\Omega_0^c) \leq(1+C)\delta^p\;,
  \end{equ}
  as required.

  To obtain the second statement, it is sufficient to consider
  \eref{e:Malliavin2} with $\eps = \beta_{2n}/\delta^2$, so that one can
  take for $\gamma$ the random variable equal to $\eps$ on the set for
  which the bound \eref{e:Malliavin2} holds and $0$ on its complement.
  It then follows from the choice \eref{e:betaChoice} for $\beta_{2n}$
  that the assumption for the first part are satisfied with $C = 1$
  and $p = \bar p$, so that the statement follows.
\end{proof}

We now introduce a ``compensator'' 
\begin{align*}
  \chi_{2n+2} = \exp\big(\eta V(u_{2n+1})+ \eta V(u_{2n})\big)\;,
\end{align*}
and, in analogy to before, we set $\chi^{(2n)}=\prod_{k=1}^n
\chi_{2k}$. Proposition \ref{prop:Lyop} implies that for any $p \in
[0,\bar p]$
\begin{align}
    \label{eq:XiBound}
    \E ( \chi^{(2n)})^p \leq \exp( p \kappa V(u_0) + p \kappa C_L 2n) \;,
  \end{align}
  where $\kappa = \eta/(1-\eta')$. Note that
  Assumption~\ref{ass:Jacobian} made sure that $\eta$ is sufficiently
  small so that $\kappa \bar p \le 1$.
  With these preliminaries complete, we now return to the analysis of
  \eref{e:decompos}.
  \begin{lemma}\label{lem:contract} For any $\eps >0$ and $p \in
    [0,\bar p/2]$, there exists a $\delta >0$ sufficiently small so
    that if one
  chooses $\beta_n$ as in Corollary \ref{lem:Rsmall} and $\eta$ such
  that $\kappa \bar p \le 1$, one has
  \begin{align*}
    \EE ( \|\Xi_{2n+2}\|^p \chi_{2n+2}^{-p} | \CF_{2n}) \leq \exp( p C_J
    - p C_\Pi + \eps p)\;.
  \end{align*}
\end{lemma}
\begin{proof}
  Since for every $\eps > 0$ there exists a constant $C_\eps$ such
  that $|x+y|^p \le e^{p\eps/2} |x|^p + C_\eps^p |y|^p$, recalling the
  definition of $I_{2n+2,1}$ and $I_{2n+2,2}$ from \eqref{e:decompos} we
  have that
\begin{equs}
  \EE ( \|\Xi_{2n+2}\|^p \chi_{2n+2}^{-p} | \CF_{2n}) &\leq e^{\eps p/2}
  \EE ( \|I_{2n+2,1}\|^p \chi_{2n+2}^{-p} | \CF_{2n}) \\
  &\qquad + C_\eps^p \EE (
  \|I_{2n+2,2}\|^p \chi_{2n+2}^{-p} | \CF_{2n})\;.
\end{equs}
We begin with the first term since it is the most
straightforward one. Using the fact that $\|\RM^{\beta_{2n}}_{2n}\| \leq 1$
and that $\bar p \eta < 1-\eta'$ by the assumption on $\eta$, we
obtain from Assumptions~\ref{ass:Lyapunov} and \ref{ass:Jacobian} that
\begin{align*}
  \EE ( \|I_{2n+2,1}\|^p \chi_{2n+2}^{-p} | \CF_{2n})
  \leq&\exp(-p\eta V(u_{2n})) \EE\Big(\EE\big(\|J_{2n+1} \Pi^\perp\|^{p}| \CF_{2n+1}\big)
  \\&\qquad  \times \exp(-p\eta
  V(u_{2n+1}))\|J_{2n}\|^{p} \Big|\CF_{2n}\Big)\\
  \leq & \exp( p C_J - p C_\Pi)\;.
\end{align*}
Turning to the second term, we obtain for any $\delta \in (0,1)$ the bound
\begin{equs}
  \EE ( \|I_{2n+2,2}\|^p &\chi_{2n+2}^{-p} | \CF_{2n}) \leq \exp(-p\eta V(u_{2n}))
  \sqrt{\EE \bigl(\|\Pi \RM^{\beta_{2n}}_{2n}\|^{2p} | \CF_{2n}\bigr)}\\&\quad
  \times \sqrt{\EE\Big(
    \EE\big( \|J_{2n+1}\|^{2p}| \CF_{2n+1}\big)\exp(-2p\eta
    V(u_{2n+1}))\|J_{2n}\|^{2p} \Big| \CF_{2n}\Big)}\\
  &\leq \exp( p 2C_J) \delta^p \sqrt{2}\;,
\end{equs}
provided that we choose $\beta_n$ as in Corollary \ref{lem:Rsmall}.
Choosing now $\delta$ sufficiently small (it suffices to choose it
such that $\delta^p \le {\eps p \over 2\sqrt 2} C_\eps^{-p} e^{-pC_J
  -p C_\Pi}$ for every $p \le \bar p/2$) we obtain the desired bound.
\end{proof}

Combining Lemma \ref{lem:contract} with \eqref{eq:XiBound}, we obtain
the needed result which ensures that the ``error term'' $\rho_t$ from
\eqref{eq:asfPrototype} goes to zero.
\begin{lemma}\label{lem:Rconstrol}
  For any $p \in [0,\bar p/4]$ and $\tilde \kappa \in [0,C_\Pi - C_J - 2
  \kappa C_L)$ there exists a choice of the $\beta_n$ of the form
  \eref{e:betaChoice} so that
  \begin{align*}
    \EE \|\Xi^{(2n+2)}\|^p \leq \exp\big( p \kappa V(u_0) - p\tilde
    \kappa n \big)\;,
  \end{align*}
for all $u_0 \in \CH$.
\end{lemma}
\begin{proof} Since
  \begin{align*}
    \EE \|\Xi^{(2n+2)}\|^p \leq \Big( \EE \|\Xi^{(2n+2)}\|^{2p}
    (\chi^{(2n+2)})^{-2p}\Big)^\frac12 \Big(
    \EE(\chi^{(2n+2)})^{2p}\Big)^\frac12\;,
  \end{align*}
  the result follows by combining Lemma \ref{lem:contract} with
  \eqref{eq:XiBound}.
\end{proof}


\subsection[Controlling the size of the variation $h_t^\xi$]{Controlling the size of the variation $\pmb{h_t^\xi}$}

\label{sec:sizeVariation}
We now turn to controlling the size of
\begin{equation}
  \label{eq:costofH}
  \EE \Big| \int_0^{n} \ip{h_s^\xi}{dW_s} \Big|^2\;,
\end{equation}
uniformly as $n\rightarrow \infty$. We assume throughout this section that
$h_t^\xi$ was constructed as in Section~\ref{sec:choiceh} with $\beta_n$
as in \eref{e:betaChoice}.

Since our choice of $h_s^\xi$ is not adapted to the $W_s$, this does
\textit{not} follow from a simple application of It\^o's isometry.  However, the
situation is not as bad as it could be, since the control is ``block
adapted.'' By this we mean that $h_{n}$ is adapted to $\cF_{n}$ for
every integer value of $n$.  For non-integer values $t \in (n,n+1]$,
$h_t$ has no reason to be $\cF_t$-measurable in general, but it is
nevertheless $\cF_{n+1}$-measurable.  The stochastic integral in
\eref{eq:costofH} is accordingly not an It\^o integral, but a
Skorokhod integral. Hence to estimate \eqref{eq:costofH} we must use
its generalization given in \eref{e:genIto} which produces
\begin{equation}
  \label{eq:skorhodItoFormula}
  \EE\Big|\int_0^{{2n}} \ip{h_s^\xi}{dW_s}\Big|^2 \leq \EE
  \Norm{h^\xi}^2_{[0,2n]}  + \sum_{k=0}^{n-1} \int_{2k}^{2k+1}
  \int_{2k}^{2k+1}   \EE  \norm{\DM_s h_t}^2\,ds\,dt
\end{equation}
where $\Norm{f}^2_{I}=\int_I |f(s)|^2 \, ds$ and $\norm{M}$ denotes the
Hilbert-Schmidt norm on linear operators from $\R^d$ to $\R^d$. We see the
importance of the ``block adapted'' structure of $h_s$. If not for this
structure, the integrand appearing in the second term above would need
to decay both in $s$ and $t$ to be finite.

The main result of this section is

\begin{proposition}\label{prop:boundvar}
  Let Assumptions~\ref{ass:Malliavin}--\ref{ass:smoothing} hold.
  Then, if one chooses $\beta_n$ as in \eref{e:betaChoice}, there
  exists a constant $C>0$ such that
  \begin{equ}
    \lim_{n \to \infty} \EE \Big| \int_0^{n} \scal{{h_s^\xi},{dW_s}}
    \Big|^2 \le C \exp \bigl((8\eta + 2\kappa)
    V(u_0)\bigr)\U^2(u_0) \|\xi\|^2\;.
  \end{equ}
\end{proposition}

\begin{proof}[of Proposition \ref{prop:boundvar}] In the interest of
  brevity we will set $\widetilde \MM_n =\MM_{n} + \beta_n$ and $I_n =
  [n,n+1]$.  We will also write $\Norm{h}_{I}$ for the norm on $L^2(I,
  \R^d)$ viewed as a subset of $\Cam$ and we will use $\|\cdot\|$
  and $\Norm{\cdot}_I$ to denote respectively the induced operator
  norm on linear maps from $\CH$ to $\CH$ and $\Cam$ to
  $\CH$. Hopefully without too much confusion, we will also use
  $\Norm{\cdot}_I$ to denote the induced operator norm on linear maps
  from $\CH$ to $\Cam$. In all cases, we will further abbreviate $\Norm{h}_{I_n}$ to
  $\Norm{h}_{n}$.

Observe now that the definitions of $\widetilde\MM_n$ and $\AM_n$ imply
the following almost sure bounds:
  \begin{equ}\label{eq:AMBounds}
    \Norm{\widetilde\MM_n^{-1/2}\AM_n}_n  \leq 1\;, \qquad \Norm{\AM_n^*
    \widetilde\MM_n^{-1/2}}_n  \leq 1\;, \qquad \|\widetilde \MM_n^{-1/2}\| \leq
    \beta_n^{-1/2}\;.
  \end{equ}

We start by bounding the first term on
  the right hand-side of \eqref{eq:skorhodItoFormula}. Observe that
  \begin{equation}\label{eq:itoPartBound}
    \Norm{h}_{[0,2n]}^2  = \sum_{k=0}^n   \Norm{h}_{2k}^2\,.
  \end{equation}
  Using the bound on $\AM_{k}^*\widetilde\MM_{k}^{-1}$ from
  \eqref{eq:AMBounds}, we obtain
  \begin{align*}
    \Norm{h}_{2k}=\Norm{\AM_{{2k}}^*\widetilde\MM_{2k}^{-1} J_{2k}
    \rho_{2k}}_{2k} \leq \beta_{2k}^{-1/2} \|J_{2k}\| \|\rho_{2k}\|\;.
  \end{align*}
  By our assumption that $10/\bar p + 2/\bar q \leq 1$ we can find
  $1/q + 1/r + 1/p = 1$ with $q \leq \bar q$, $2r \leq \bar p$ and $2p\leq
  \bar p$. By the H\"older inequality we thus have
  \begin{align*}
    \EE \Norm{h}_{2k}^2\leq \big(\E \beta_{2k}^{-q}\big)^{1/q}
    \big(\E\|J_{2k}\|^{2r}\big)^{1/r} \big(\E\|\rho_{2k}\|^{2p}\big)^{1/p}\;.
  \end{align*}
  From Proposition~\ref{prop:Jcontrol},
  Assumption~\ref{ass:Malliavin} and Lemma~\ref{lem:Rconstrol}, we obtain the existence
  of a positive constant $C$ (depending only on the choice made for $\tilde \kappa$ and on the bounds
  given by our standing assumptions) such that one has the bounds 
\begin{equs}[e:boundsRhoJBeta]
\big(\E  \|J_{2k}\|^{2r}\big)^{1/r} &\leq \exp \bigl(2\eta (\eta')^{2k} V(u_0)
  + 2C_J + 2\kappa C_L \bigr)\;, \\
\big(\EE \beta_{2k}^{-q}\big)^{1/q} &\leq C \U(u_0)\;, \\
\big(\E\|\rho_{2k}\|^{2p}\big)^{1/p} &\le \exp(2 \kappa V(u_0) - 2\tilde \kappa k) \|\xi\|^2\;.
\end{equs}
combining these bounds and summing over $k$ yields
\begin{align}\label{eq:controlItoPart}
 \EE   \Norm{h}_{[0,2n]} ^2 \leq C \U(u_0) \exp(2 (\eta + \kappa) V(u_0)) \|\xi\|^2\;,
\end{align}
uniformly in $n \ge 0$.

We now turn to bound the second term on the right hand side of \eref{eq:skorhodItoFormula}.
  Since the columns of the matrix representation of the integrand are
  just $\DM_s^i$, the $i$th component of the Malliavin derivative, we
  have
  \begin{equation}\label{eq:nonAddPart}
    \int_{2k}^{2k+1}
    \int_{2k}^{2k+1}  \norm{\DM_s h_t}^2\,ds\,dt=
    \sum_{i=1}^m\int_{2k}^{2k+1} \Norm{\DM_s^i h}_{2k}^2\,ds\,.
  \end{equation}
  From the  definition of $h_t$, Lemma~\ref{lem:DJ_DA_Control}, the relation $\widetilde\MM_{2k} = \AM_{2k}\AM_{2k}^* + \beta_{2k}$, and the fact that both $\rho_{2k}$ and $\beta_{2k}$ are $\CF_{2k}$-measurable,
   we have that
    for fixed $s \in I_{2k}$, $ \DM_s^i h$ is an element of $L^2(I_{2k},\R) \subset \Cam$ with:
  \begin{align}\label{eq:DH}
    \DM_s^i h &= (\DM_s^i \AM^*_{2k})\widetilde\MM_{2k}^{-1} J_{2k}
    \rho_{2k} + \AM^*_{2k}\widetilde\MM_{2k}^{-1}
    (\DM_s^i J_{2k}) \rho_{2k}\\
    &- \AM^*_{2k}\widetilde \MM_{2k}^{-1}\big( (\DM_s^i \AM_{2k})
    \AM_{2k}^* +\AM_{2k}(\DM_s^i \AM_{2k}^*) \big)\widetilde \MM_{2k}^{-1}
    J_{2k} \rho_{2k}\,.\notag
  \end{align}
  For brevity we suppress the subscripts $k$ on the operators and
  norms for a moment. 
It then follows from \eref{eq:AMBounds} that one has the almost sure bounds
  \begin{equs}
    \Norm{\widetilde\MM^{-1}\AM} &\leq \| \widetilde\MM^{-1/2}\|
    \Norm{\widetilde\MM^{-1/2}\AM} \leq \beta^{-1/2}\;, \\
    \Norm{\AM^* \widetilde\MM^{-1}} & \leq \Norm{\AM^*
    \widetilde\MM^{-1/2}}\|
    \widetilde\MM^{-1/2}\| \leq \beta^{-1/2}\;, \\
    \Norm{(\DM_s^i \AM^*)\widetilde\MM^{-1} J}&\leq \Norm{\DM_s^i
    \AM^*}\|\widetilde\MM^{-1}\|
    \|J\|\leq \beta^{-1} \Norm{\DM_s^i \AM}\|J\|\;, \\
    \Norm{\AM^*\widetilde\MM^{-1} (\DM_s^i J
    )}&\leq\Norm{\AM^*\widetilde\MM^{-1}}\|\DM_s^i
    J\|\leq\beta^{-1/2}\|\DM_s^i J\|\;.
  \end{equs}
  In particular, this yields the bounds
  \begin{equs}
    \Norm{\AM^*\widetilde \MM^{-1} (\DM_s^i \AM) \AM^*\widetilde
      \MM^{-1} J} &\leq \Norm{\AM^*\widetilde \MM^{-1}}^2
    \Norm{\DM_s^i \AM}\|J\|
    \leq \beta^{-1} \Norm{\DM_s^i \AM} \|\| J\|\\
    \Norm{\AM^*\widetilde \MM^{-1}\AM(\DM_s^i \AM^*) \widetilde
      \MM^{-1} J} &\leq \Norm{\AM^*\widetilde
      \MM^{-1/2}}^2\Norm{\DM_s^i \AM^*}\| \widetilde
    \MM^{-1}\|\|J\| \\
    & \leq \beta^{-1}\Norm{\DM_s^i \AM} \|J \|\,.
  \end{equs}
  Applying all of these estimates to \eqref{eq:DH} we obtain the bound
  \begin{align*}
    \Norm{\DM_s^i h}_{2k} \leq 3 \beta_{2k}^{-1} \Norm{\DM_s^i \AM_{2k}}_{2k}
    \|J_{2k}\| \| \rho_{2k}\|+ \beta_{2k}^{-1/2}\|\DM_s^i J_{2k}\| \| \rho_{2k}\|\,.
  \end{align*}
  The assumption that $10/\bar p + 2/\bar q \leq 1$ ensures that we
  can find $q \leq \bar q/2$, $r \leq \bar p/2$ and $p \leq \bar p/4$ with
  $1/r + 2/p+1/q = 1$.  Applying H\"older's inequality to the
  preceding products yields: 
  \begin{align*}
    \EE \Norm{\DM_s^i h}_{2k}^2 \leq& 18 \bigl(\E
    \beta_{2k}^{-2q}\big)^{1/q}\bigl( \EE \Norm{\DM_s^i
      \AM_{2k}}_{2k}^{2p}\EE \|\rho_{2k}\|^{2p}\big)^{1/p} \bigl(\EE
    \|J_{2k} \|^{2r}\bigr)^{1/r} \\ &+ 2\bigl(\E
    \beta_{2k}^{-q}\bigl)^{1/q} \bigl(\EE\|\DM_s^i J_{2k}\|^{2p}
    \EE\|\rho_{2k}\|^{2p}\big)^{1/p}\;.
  \end{align*}
  We now use previous estimates to control each term. 
  From Lemma~\ref{lem:DJ_DA_Control} and Proposition~\ref{prop:Jcontrol}, we have the bounds 
  \begin{align*}
    \bigl( \EE\|\DM_s^i
    J_{2k}\|^{2p}\big)^{1/p}
    &\leq \exp(4\eta (\eta')^{2k} V(u_{0}) + 4\kappa C_L + 2C_J +  2C_J^{(2)})\;,\\
    \bigl( \EE \Norm{\DM_s^i \AM_{2k}}_{2k}^{2p}\big)^{1/p}
    &\leq \|G\|^2\exp(4\eta (\eta')^{2k} V(u_{0}) + 4\kappa C_L + 2C_J +  2C_J^{(2)})\;.
  \end{align*}
Recall furthermore the bounds on $\rho_{2k}$ and $J_{2k}$ already mentioned in \eref{e:boundsRhoJBeta}.
Lastly, from Assumption~\ref{ass:Malliavin} we have that, similarly as before, there exists a positive constant $C$
such that
\begin{align*}
  \big(\EE \beta_{2k}^{-q}\big)^{1/q} \leq  \big(\EE \beta_{2k}^{-2q}\big)^{1/q} \leq C \U^2(u_0)\;.
\end{align*}
Combining all of these estimates produces
\begin{align*}
   \sum_{i=1}^m\int_{2k}^{2k+1} \Norm{\DM_s^i h}_{2k}^2\,ds \leq C \exp((8 \eta + 2\kappa) V(u_0)) \U^2(u_0)\;,
\end{align*}
for some different constant $C$ depending only on $C_J,C_J^{(2)}, C_L,
\eta, \kappa,\tilde \kappa$ and the choice of $\delta$ in \eref{e:betaChoice}. 
Combining this estimate with \eqref{eq:controlItoPart} and \eqref{eq:skorhodItoFormula} concludes the proof.
 \end{proof}

\section{Spectral properties  of the Malliavin matrix}
\label{sec:mallSpec}
The results in this section build on the ideas and techniques from
\cite{MatPar06:1742} and \cite{Yuri}. In the first, the specific case
of the 2D-Navier Stokes equation was studied using similar ideas. The
time reversed representation of the Malliavin matrix used there is
also the basis of our analysis here (see also \cite{Oco88:288}). In
the context for the 2D-Navier Stokes equations, a result
analogous to Theorem~\ref{theo:Malliavin} was proven. As here, one of
the key results needed is a connection between the typical size of a
non-adapted Wiener polynomial and the typical size of its
coefficients. In \cite{MatPar06:1742}, since the non-linearity was
quadratic, only Wiener polynomials of degree one were considered
and the calculations and formulation were made a
coordinate dependent fashion. In \cite{Yuri}, the calculations were
reformulated in a basis free fashion which both made possible the
extension to more complicated non-linearities and the inclusion of
forcing which was not diagonal in the chosen basis. Furthermore in
\cite{Yuri}, a result close to Theorem~\ref{theo:Malliavin} was
claimed. Unfortunately, the auxiliary Lemma~9.12 in
that article contains a mistake, which left the proof of this result incomplete.  

That being
said, the techniques and presentation used in this and the next
section build on and refine those from \cite{Yuri}. One technical, but
important, distinction between Theorem~\ref{theo:Malliavin} and the
preceding versions is that Theorem~\ref{theo:Malliavin} allows for
rougher test functions. This is accomplished by allowing $K_{t,T}$ to
have a singularity in a certain interpolation norm as $t \rightarrow T$. See
equation \eqref{e:apadjnorm} for the precise form. This extension is
important in correcting an error in \cite{HairerMattingly06AOM} which
requires control of the Malliavin matrix of a type given by
Theorem~\ref{theo:Malliavin}, that is with test functions rougher than
those allowed in \cite{MatPar06:1742}. Indeed, the second inequality in equation
(4.25) of \cite{HairerMattingly06AOM} is not justified, since the operator $M_0$ is only selfadjoint
in $L^2$ and not in $H^1$. Theorem~\ref{theo:Malliavin}
rectifies the situation by dropping the requirement to work with $H^1$ completely.

\subsection{Bounds on the dynamic}
\label{sec:boundsdyn}

\SetAssumptionCounter{C}

As the previous sections have shown, it is sufficient to have control
on the moments of $u$ and $J$ in $\CH$ to control their moments in many
stronger norms. This motivates the next assumption. For the entirety of this section we fix a $T_0>0$.
\begin{assumption}\label{ass:global} 
  There exists a continuous function $\Psi_0\colon \CH \to [1,\infty)$
  such that, for every $T \in(0,T_0]$ and every $p \ge 1$ there exists a
  constant $C$ such that
   \begin{equs}
     \E \sup_{T \le t \le 2T} \|u_t\|^p &\le C\Psi_0^p(u_0)\;,\\
     \E \sup_{T \le s < t \le 2T} \|J_{s,t}\|^p &\le C\Psi_0^p(u_0)\;,
  \end{equs}
  for every $u_0 \in \CH$. Here, $\|J\|$ denotes the operator norm of
  $J$ from $\CH$ to $\CH$.
\end{assumption}
Under this assumption, we immediately obtain control over the adjoint $K_{s,t}$.
\begin{proposition}\label{lem:kbound}
  Under Assumption~\ref{ass:global} for every $T  \in(0,T_0]$
  and every $p \ge 1$ there exists a constant $C$ such that
  \begin{equ}
    \E \sup_{T \le s < t \le 2T} \|K_{s,t}\|^p \le C\Psi_0^p(u_0)\;,
  \end{equ}
  for every $u_0 \in \CH$.
\end{proposition}
\begin{proof}
  By Proposition~\ref{e:JKadjoint} we know that $K_{s,t}$ is the
  adjoint of $J_{s,t}$ in $\CH$. Combined with
  Assumption~~\ref{ass:global} this implies the result.
\end{proof}
In the remainder of this section, we will study the solution
to \eref{e:SPDE} away from $t = 0$ and up to some terminal time
$T$ which we fix from now on. We also introduce the interval $\II = [{T\over 2},
T-\delta]$ for some $\delta \in (0, {T\over 4}]$ to be determined later.
Given $u_t$ a solution to \eref{e:SPDE}, we also define a process
$v_t$ by $v_t = u_t - GW(t)$, which is more regular in time.
Using Assumption~\ref{ass:global} and the \text{a priori} estimates
from the previous sections, we obtain:
\begin{proposition}\label{prop:apriori} Let
  Assumption~\ref{ass:global} hold and $\Psi_0$ be the function
  introduced there.  For any fixed $\gamma < \gamma_\star$ and $\beta
  < \beta_\star$ there exists a positive $q$ so that if
  $\Psi=\Psi_0^q$ then the solutions to \eref{e:SPDE} satisfy the
  following bounds for every initial condition $u_0 \in \CH$:
  \minilab{e:aprioriuv}
  \begin{equs}
    \EE \sup_{t\in\II} {\|u_t\|_{\gamma+1}^p} &\le {C_p \Psi^p(u_0)}\;,
    \label{e:apsol}\\
    \EE \sup_{t\in\II} \|\d_t v_t\|_{\gamma}^p &\le {C_p \Psi^p(u_0)}\;.
    \label{e:aplip}
  \end{equs}
  Furthermore, its linearization $J_{0,t}$ is bounded by  \minilab{e:aprioriJ}
  \begin{equs}
    \EE \sup_{t\in\II}  \sup_{\|\phi\| \le 1}{\|J_{0,t} \phi\|_{\gamma+1}^p} &\le {C_p
      \Psi^p(u_0)}\;,
    \label{e:Jsol}\\
    \EE \sup_{t\in\II}  \sup_{\|\phi\| \le 1}\|\d_t J_{0,t} \phi\|_{\gamma}^p &\le {C_p
      \Psi^p(u_0)}\;.
    \label{e:Jlip}
  \end{equs}
  Finally, the adjoint $K_{t,T}$ to the linearization satisfies the
  bounds \minilab{e:aprioriK}
  \begin{equs}
    \EE \sup_{t \in \II} \sup_{\|\phi\| \le 1}
    \|K_{t,T}\phi\|_{\beta+1}^p &\le {C_p\Psi^p(u_0) \over
      \delta^{\bar
        p_\beta p}}\;, \label{e:apadjnorm} \\
    \EE \sup_{t \in \II} \sup_{\|\phi\| \le 1} \|\d_t
    K_{t,T}\phi\|_\beta^p &\le {C_p\Psi^p(u_0) \over \delta^{\bar
        p_\beta p}}\;, \label{e:apadj}
  \end{equs}
 where $\bar p_\beta$ is as in Proposition~\ref{prop:bootstrapK}. In all these bounds,
 $C_p$ is a constant depending only on $p$ and on the details of the equation \eref{e:SPDE}.
\end{proposition}

\begin{remark}
  One can assume without loss of generality, and we will do so from
  now on, that the exponent $q$ defining $\Psi$ is greater or equal to
  $n$, the degree of the nonlinearity. This will be useful in the
  proof of Lemma~\ref{lem:lipDeltaAdj} below.
\end{remark}

\begin{proof}
  It follows immediately from Assumption~\ref{ass:global} that
  \begin{equ}
    \E \sup_{t \in [T/4, T]} \|u_t\|^p \le C\Psi_0^{p}(u_0)\;.
  \end{equ}
  Combining this with Proposition~\ref{prop:bootstrapping} yields
  the first of the desired bounds with $q=p_\gamma$. Here, $\Psi_0$ is as
  in Assumption~\ref{ass:global} and $p_\gamma$ is as in
  Proposition~\ref{prop:bootstrapping}.

  Turning to the bound on $\d_t v_t$, observe that $v$ satisfies the
  random PDE
  \begin{equ}
    \d_t v_t = F\big(v_t + GW(t)\big) = F(u_t)\;,\quad v_0 = u_0\;.
  \end{equ}
  It follows at once from Proposition~\ref{prop:bootstrapping} and
  Assumption~\ass{N} that the quoted estimate holds with
  $q=p_{\gamma+1}$. More precisely, it follows from
  Proposition~\ref{prop:bootstrapping} that $u_t \in \CH_\alpha$ for
  every $\alpha < \gamma_\star + 1$.  Therefore, $Lu_t \in \CH_\gamma$
  for $\gamma < \gamma_\star$. Furthermore, $N \in
  \Poly(\CH_{\gamma+1}, \CH_\gamma)$ by Assumption~\ass{N}, so that
  $N(u_t) \in \CH_{\gamma}$ as well.  The claim then follows from the
  \textit{a priori} bounds obtained in
  Proposition~\ref{prop:bootstrapping}.

  Concerning the bound \eref{e:Jsol} on the linearization $J_{0,t}$, Proposition~\ref{prop:bootstrapJ}
  combined with Assumption~\ref{ass:global} proves the result with
  $q=\bar q_\gamma+1$.  The line of reasoning used to bound
  $\|\d_t v_t\|_\gamma$ also controls    $\|\d_t J_{s,t}\|_\gamma$ for $s<t$ and
  $s,t \in \II$, since $\d_t J_{s,t}= - L J_{s,t} + DN(u_t)
  J_{s,t}$.  

  Since Proposition~\ref{lem:kbound} give an completely analogous
  bound for $K_{s,t}$ in $\CH$ as for $J_{s,t}$ the results on $K$
  follow from the \textit{a priori} bounds in  Proposition~\ref{prop:bootstrapK}.
 \end{proof}

\subsection{A H\"ormander-like theorem in infinite dimensions}
\label{sec:Hormander}

In this section, we are going to formulate a lower bound on the
Malliavin covariance matrix $\MM_t$ under a condition that is very
strongly reminiscent of the bracket condition in H\"ormanders
celebrated ``sums of squares'' theorem \cite{Ho,H1}. The proof of the result presented in this
section will be postponed until Section~\ref{sec:proofMalliavin} and constitutes
the main technical result of this work.

Throughout all of this section and Section~\ref{sec:proofMalliavin}, we are going to make use of the bounds
outlined in Proposition~\ref{prop:apriori}.  We therefore now fix once
and for all some choice of constants
\begin{equ}[e:gammabeta]
  \gamma \in [-a,\gamma_\star) \quad\text{and}\quad \beta \in
  [-a,\beta_\star) \quad\text{satisfying}\quad \gamma + \beta \ge -1
  \;.
\end{equ}
From now on, we will only ever use Proposition~\ref{prop:apriori} with
this fixed choice for $\gamma$ and $\beta$. This is purely a
convenience for expositional clarity since we will need these bounds
only finitely many times.  As a side remark, note that one should think
of these constants as being arbitrarily close to $\gamma_\star$ and
$\beta_\star$ respectively.

With $\gamma$ and $\beta$ fixed as in \eref{e:gammabeta}, we introduce
the set
\begin{equ}[e:defPolygb]
  \Poly(\gamma,\beta) \eqdef \Poly(\CH_{\gamma},\CH_{-\beta-1}) \cap
  \Poly(\CH_{\gamma+1},\CH_{-\beta})
\end{equ}
for notational convenience. (For integer $m$, $\Poly^m(\gamma,\beta)$
is defined analogously.) A polynomial $Q \in \Poly(\gamma,\beta)$ is
said to be \textit{admissible} if
\begin{equ}{} [Q_\alpha, F_\sigma] \in \Poly(\gamma,\beta)\;,
\end{equ}
for \textit{every} pair of multi-indices $\alpha, \sigma$. Here, $Q_\alpha$ and
$F_\sigma$ are defined as in \eref{e:TaylorQ} and $F$ is the drift
term of the SPDE \eref{e:SPDE} defined in \eref{e:defF}.

This definition allows us to define a family of increasing subsets
$\AA_i \subset \Poly(\gamma,\beta)$ by the following recursion:
\begin{equs}
  \AA_1 &= \{g_k\;,\quad k=1,\ldots, d\} \subset \CH_{\gamma_\star + 1}  
 \approx  \Poly^0(\CH_{\gamma_\star+1}) \subset \Poly(\gamma,\beta)\;, \\
  \AA_{i+1} &= \AA_i \cup \{Q_\alpha,\; [F_\sigma, Q_\alpha]\;:\; Q\in \AA_i,\;
  \text{$Q$ admissible}, \text{ and } \alpha,\sigma \text{ multi-indices} \} \;.
\end{equs}
\begin{remark}
  Recall from \eref{e:repeatedBrakets} that $Q_\alpha$ is proportional
  to the iterated ``Lie bracket'' of $Q$ with $g_{\alpha_1}$,
  $g_{\alpha_2}$ and so forth. Similarly, $[F_\sigma, Q_\alpha]$ is
  the Lie bracket between two different iterated Lie brackets. As such,
  except for the issue of admissibility, the set of brackets considered here
  is exactly the same as in the traditional statement of H\"ormander's
  theorem, only the order in which they appear is slightly different.
\end{remark}
To each $\AA_N$ we associate a positive symmetric quadratic form-valued function
$\CQ_N$ by
\begin{equ}
  \scal{\phi, \CQ_N(u)\phi} = \sum_{Q \in \AA_N} \scal{\phi,
    Q(u)}^2\;.
\end{equ}
Lastly for $\alpha \in (0,1)$, and for a given orthogonal projection $\Pi\colon \CH \to \CH$,
we define $\CS_\alpha \subset \CH$ by
\begin{equ}[e:Sdef]
  \CS_\alpha = \{\phi \in \CH \setminus \{0\} \,:\, \|\Pi\phi\| \ge
  \alpha \|\phi\| \}\;.
\end{equ}
With this notation, we make the following non-degeneracy assumption:

\begin{assumption}\label{ass:Hormander}
  For every $\alpha>0$, there exists $N>0$ and a function
  $\Lambda_\alpha \colon \CH \to [0,\infty)$ such that
  \begin{equ}
    \inf_{\phi \in \CS_\alpha} {\scal{\phi, \CQ_N(u)\phi} \over
      \|\Pi\phi\|^2} \ge \Lambda_\alpha^2(u) \;,
  \end{equ}
  for every $u \in \CH_a$.  Furthermore, for every $p \ge 1$, $t >0$ and every
  $\alpha \in (0,1)$, there exists $C$ such that $\E\,
  \Lambda_\alpha^{-p}(u_t) \le C \Psi^p(u_0)$ for every initial
  condition $u_0 \in \CH$.
\end{assumption}

\begin{remark}
Assumption~\ref{ass:Hormander} is in some sense weaker than the usual non-degeneracy
condition of H\"ormander's theorem, since it only requires $\CQ_N$ to be sufficiently non-degenerate
on the range of $\Pi$. In particular, if $\Pi = 0$, then  Assumption~\ref{ass:Hormander} is void and always
holds with $\Lambda_\alpha = 1$, say. This is the reason why, by choosing for $\Pi$ a projector onto
some finite-dimensional subspace of $\CH$, one can expect Assumption~\ref{ass:Hormander} to hold
for a finite value of $N$, even in our situation where $\AA_N$ only contains finitely many elements.
\end{remark}

\begin{remark}
As will be seen in Section~\ref{sec:examples}, it is often possible to choose $\Lambda_\alpha$ to be
a constant, so that the second part of Assumption~\ref{ass:Hormander} is automatically satisfied.
\end{remark}

When Assumption~\ref{ass:Hormander}  holds, we have the following result whose proof
is given in Section~\ref{sec:proofMalliavin}.
\begin{theorem}\label{theo:Malliavin}
Consider an SPDE of the type \eref{e:SPDE} such that Assumptions~\ref{ass:basic} and \ref{ass:global} hold.
  Let furthermore the Malliavin matrix $\MM_t$ be defined as in \eref{e:defM} and
  $\CS_\alpha$ as in \eref{e:Sdef}. Let $\Pi$ be a finite rank orthogonal
  projection satisfying Assumption~\ref{ass:Hormander}. Then, there
  exists $\theta>0$ such that, for every $\alpha \in (0,1)$, every $p \ge 1$ and every $t>0$
  there exists a  constant $C$ such that the bound
  \begin{equ}
    \P \Bigl( \inf_{\phi \in \CS_\alpha} {\scal{\phi, \MM_t \phi}\over \|\phi\|^2} \le
    \eps \Bigr) \le C \Psi^{\theta p}(u_0) \eps^{p}\;,
  \end{equ}
holds for every $u_0 \in \CH$ and every $\eps \le 1$.
\end{theorem}

\begin{remark} If $\Pi$ is a finite rank orthogonal projection
  satisfying Assumption~\ref{ass:Hormander} then
  Theorem~\ref{theo:Malliavin} provides the critically ingredient to
  prove the smoothness of the density of $(\CP_t^*\delta_x)\Pi^{-1}$
  with respect to Lebesgue measure. Though \cite{Yuri} contains a few
  unfortunate errors, it still provides the framework needed to deduce
  smoothness of these densities from  Theorem~\ref{theo:Malliavin}. In
  particular, one needs to prove that $\Pi u_t$ is infinitely
  Malliavin differentiable. Section~5.1 of \cite{Yuri} shows how to
  accomplish this in a setting close to ours, see also \cite{MatPar06:1742}.     
\end{remark}
\subsection{Proof of Theorem~\ref{theo:Malliavin}}
\label{sec:proofMalliavin}

While the aim of this section is to prove
Theorem~\ref{theo:Malliavin}, we begin with some preliminary
definitions which will simplify its presentation.
Many of the arguments used will rely on the construction of
``exceptional sets'' of small probability outside of which certain
intuitive implications hold. This justifies the introduction of the
following notational shortcut:

\begin{definition}\label{def:negligible}
  Given a collection $H = \{H^{\eps}\}_{\eps \le 1}$ of subsets of the
  ambient probability space $\Omega$, we will say that ``$H$ is a
  family of negligible events'' if, for every $p \ge 1$ there exists a
  constant $C_p$ such that $\P(H^{\eps}) \le C_p\eps^p$ for every
  $\eps \le 1$.
 
  Given such a family $H$ and a statement $\Phi_\eps$ depending on a
  parameter $\eps>0$, we will say that ``$\Phi_\eps$ holds modulo $H$''
  if, for every $\eps \le 1$, the statement $\Phi_\eps$ holds on the
  complement of $H^\eps$.
 
  We will say that the family $H$ is ``universal'' if it does not depend
  on the problem at hand.  Otherwise, we will indicate which
  parameters it depends on.
\end{definition}

Given two families $H_1$ and $H_2$ of negligible sets, we write $H =
H_1 \cup H_2$ as a shortcut for the sentence ``$H^\eps = H_1^\eps \cup
H_2^\eps$ for every $\eps \le 1$.'' Let us state the following useful
fact, the proof of which is immediate:
 
\begin{lemma}\label{lem:sumup}
  Let $H^\eps_n$ be a collection of events with $n \in \{1,\ldots,
  C\eps^{-\kappa}\}$ for some arbitrary but fixed constants $C$ and
  $\kappa$ and assume that $\P(H^\eps_k) = \P(H^\eps_\ell)$ for any
  pair $(k,\ell)$.  Then, if the family $\{H^\eps_1\}$ is negligible,
  the family $\{H^\eps\}$ defined by $H^\eps = \bigcup_{n} H^\eps_n$
  is also negligible.
\end{lemma}

\begin{remark}
  The same statement also holds of course if the equality between
  probabilities of events is replaced by two-sided bounds with
  multiplicative constants that do not depend on $k$, $\ell$, and $\eps$.
\end{remark}

An important particular case is when the family $H$ depends on the initial condition $u_0$ to
\eref{e:SPDE}. We will then say that $H$ is ``$\Psi$-controlled'' if the
constant $C_p$ can be bounded by $\tilde C_p \Psi^p(u_0)$, where
$\tilde C_p$ is independent of $u_0$. 

In this language, the conclusion of Theorem~\ref{theo:Malliavin} can
be restated as saying that there exists $\theta>0$ such that, for every $\alpha > 0$,
the event
\begin{equ}
  \inf_{\phi \in \CS_\alpha} \scal{\phi, \MM_T \phi} \le
  \eps\|\phi\|^2 
\end{equ}
is a $\Psi^{\theta}$-controlled family of negligible events.  Recall that the terminal time
$T$ was fixed once and for all and that the function $\Psi$ was defined in Proposition~\ref{prop:apriori}.
We
further restate this as an implication in the following theorem which
is easily seen to be equivalent to Theorem~\ref{theo:Malliavin}:

\begin{theorem}\label{thm:Malliavin2}
  Let $\Pi$ be a finite rank orthogonal projection satisfying
  Assumption~\ref{ass:Hormander}. Then, there exists $\theta>0$ such
  that  for every $\alpha \in (0,1)$, the implication
\begin{equ}
  \phi \in \CS_\alpha\qquad \Longrightarrow
  \qquad \scal{\phi, \MM_T \phi} > \eps \|\phi\|^2 
\end{equ}
holds modulo a $\Psi^{\theta}$-controlled family of negligible events.
\end{theorem}

\subsection{Basic structure and idea of proof of
  Theorem~\ref{thm:Malliavin2}}
 \label{sec:proofMall}

We begin with an overly simplified version of the argument which
neglects some technical difficulties. The basic idea of the proof is to
argue that if $\scal{\MM_T \phi,\phi}$ is small then
$\scal{\CQ_k(u_T) \phi,\phi}$ must also be small (with high probability) for every $k > 0$. This
is proved inductively, beginning with the directions which are
directly forced, namely those belonging to $\AA_1$.
Assumption~\ref{ass:Hormander} then guarantees
in turn that $\|\Pi\phi\|$ must be
small with high probability.  On the other hand, since $\phi \in
\CS_\alpha$, we know for a fact that $\|\Pi \phi\| \geq \alpha \|\phi\|$ which is
not small. Hence one of the highly improbable events must have
occurred.

This sketch of proof is essentially the same as that of H\"ormander's theorem
in finite dimensions, see \cite{Malliavin,KSAMI,KSAMII,Norr,Nualart}. Trying to adapt this argument to the infinite-dimensional
case, one is rapidly faced with two major hurdles. First, processes of the form
$t \mapsto \scal{J_{t,T}g, \phi}$ appearing in the definition of $\MM_T$ are not adapted
to the filtration generated by the driving noise. In finite dimensions, this difficulty is overcome by
noting that
\begin{equ}
\MM_t = J_{0,t} \hat \MM_t J^*_{0,t}\;,\qquad \hat \MM_t = \int_0^t J_{0,s}^{-1} GG^* \bigl(J^*_{0,s}\bigr)^{-1}\,ds\;,
\end{equ}
and then working with $\hat \MM_t$ instead of $\MM_t$. ( $\hat \MM_t$
is often called the reduced Malliavin covariance matrix.) The
processes $t \mapsto \scal{J_{0,t}^{-1} g, \phi}$ appearing there are
now perfectly nice semimartingales and one can use Norris' lemma
\cite{Norr}, which is a quantitative version of the Doob-Meyer
decomposition theorem, to show inductively that if $\scal{\phi,
  \MM_T\phi}$ is small, then $t \mapsto \scal{J_{0,t}^{-1} Q(u_t),
  \phi}$ must be small for every vector field $Q \in \AA_k$.  In our
setting, unlike in some previous results for infinite-dimensional
systems \cite{BauTei05:1765}, the Jacobian $J_{0,t}$ is not
invertible. This is a basic feature of dissipative PDEs with a smoothing
linear term which is the dominating term on  the right hand side. Such
dynamical systems only generate semi-flows as opposed to invertible flows.  

   Even worse, there appears to be no good theory
characterising a large enough subset belonging to its range. The only
other situations to our knowledge where this has been overcome
previously are the linear case \cite{Oco88:288}, as well as the
particular case of the two-dimensional Navier-Stokes equations on the
torus \cite{MatPar06:1742} and in \cite{Yuri} for a setting close to
ours. As in those settings, we do not attempt to define something like
the operator $\hat \MM_t$ mentioned above but instead we work directly
with $\MM_t$, so that we do not have Norris' lemma at our disposal. It
will be replaced by the result from Section~\ref{sec:Wiener} on
``Wiener polynomials.''  This result states that if one considers a
polynomial where the variables are independent Wiener processes and
the coefficients are arbitrary (possibly non-adapted) Lipschitz
continuous stochastic processes, then the polynomial being small
implies that with high probability each individual monomial is small.
It will be shown in this section how it is possible to exploit the
polynomial structure of our nonlinearity in order to replace Norris'
lemma by such a statement.

Another slightly less serious drawback of working in an
infinite-dimensional setting is that we encounter singularities at
$t=0$ and at $t=T$ (for the operator $J_{t,T}$).  Recall the
definition of the time interval $\II = [{T\over 2}, T-\delta]$ from
Section~\ref{sec:PDEbounds}. We will work on this interval which is
strictly included in $[0,T]$ to avoid these singularities.  There will
be a trade-off between larger values of $\delta$ that make it easy to
avoid the singularity and smaller values of $\delta$ that make it
easier to infer bounds for $\scal{\CQ_k(u_T) \phi,\phi}$.

When dealing with non-adapted processes, it is typical to replace certain
standard arguments which hinge on adaptivity by arguments which use
local time-regularity properties instead. This was also the approach
used in \cite{MatPar06:1742,Yuri}. To this end we introduce the
following H\"older norms.
For $\theta \in (0,1]$, we define the H\"older norm for functions $f
\colon \II \to \CH$ by
\begin{equ}[e:defHolderNorm]
  \/f\/_{\theta} = \sup_{s,t \in \II} {\|f(s) - f(t)\|\over
    |t-s|^\theta}\;,
\end{equ}
and similarly if $f$ is real-valued. (Note that even though we use the
same notation as for the norm in the Cameron-Martin space in the
previous section, these have nothing to do with each other. Since on
the other hand the Cameron-Martin norm is never used in the present
section, we hope that this does not cause too much confusion.) We
furthermore set
\begin{equ}
  \/f\/_{\theta,\gamma} = \sup_{s,t \in \II} {\|f(s) - f(t)\|_\gamma
    \over |t-s|^\theta}\;,
\end{equ}
where $\|\cdot\|_\gamma$ denotes the $\gamma$th interpolation norm defined in Assumption~\ref{ass:basic}.
Finally, we are from now on going to assume that $\delta$ is a
function of $\eps$ through a scaling relation of the type
\begin{equ}[e:defdelta]
  \delta = {T\over 4}\eps^{r}
\end{equ}
for some (very small) value of $r$ to be determined later.

\subsection{Some preliminary calculations}

We begin with two preliminary calculations. The first translates a
given growth of the moments of a family of random variables into a
statement saying that the variables are ``small,'' modulo a negligible
family of events. As such, it is essentially a translation of Chebyshev's
inequality into our language. The second is an interpolation result
which controls the supremum of a function's derivative by the supremum
of the function and the size of some H\"older coefficient.

\begin{lemma}\label{lem:stupid}
  Let $\delta$ be as in \eref{e:defdelta} with $r>0$, let $\Psi\colon \CH \to [1,\infty)$ be an arbitrary function,
   and let
  $X_\delta$ be a $\delta$-dependent family of random variables such
  that there exists $b \in \R$ ($b$ is allowed to be negative) such that, for every $p \ge 1$, $\E
  |X_\delta|^p \le C_p \Psi^p(u_0)\delta^{-b p}$.  Then, for any $q >
 br$ and any $c>0$, the family of events
  \begin{equ}
    \Bigl\{ |X_\delta| > {\eps^{-q} \over c}\Bigr\}
  \end{equ}
  is $\Psi^{1\over q-b r}$-dominated negligible.
\end{lemma}
\begin{proof}
  It follows from Chebychev's inequality that
  \begin{equ}
    \P \Bigl( |X_\delta| > {\eps^{-q} \over c}\Bigr) \le C_p c^p
    \Psi^p \delta^{-bp} \eps^{qp} = \bar C_\ell \bigl(\Psi^{ 1\over
      q-br}\bigr)^\ell \eps^\ell\;,
  \end{equ}
  where $\bar C_\ell$ is equal to $C_p c^p$ with $\ell = p(q-br)$.
  Provided that $q-br > 0$, this holds for every $\ell > 0$ and the
  claim follows.
\end{proof}

\begin{lemma}\label{lem:dtf}
  Let $f \colon [0,T] \to \R$ be continuously differentiable and let
  $\alpha \in (0,1]$. Then, the bound
  \begin{equ}
   \|\d_t f\|_{L^\infty} = \/ f\/_{1} \le 4\|f\|_{L^\infty} \max \Bigl\{ {1 \over
      T} , \|f\|_{L^\infty}^{-{1 \over 1+\alpha}} \/
   \d_t  f\/_{\alpha}^{1\over 1+\alpha} \Bigr\}
  \end{equ}
  holds, where $\/f\/_\alpha$ denotes the best $\alpha$-H\"older constant for $f$.
\end{lemma}

\begin{proof}
  Denote by $x_0$ a point such that $|\d_t f(x_0)| = \|\d_t
  f\|_{L^\infty}$. It follows from the definition of the
  $\alpha$-H\"older constant $\|\d_t f\|_{\CC^\alpha}$ that $|\d_t
  f(x)| \ge {1\over 2} \|\d_t f\|_{L^\infty}$ for every $x$ such that
  $|x - x_0| \le \bigl(\|\d_t f\|_{L^\infty} / 2 \|\d_t
  f\|_{\CC^\alpha}\bigr)^{1/\alpha}$.  The claim then follows from the
  fact that if $f$ is continuously differentiable and $|\d_t f(x)| \ge A$ over an interval $I$, then there
  exists a point $x_1$ in the interval such that $|f(x_1)| \ge
  A|I|/2$.
\end{proof}

\subsection{Transferring properties of $\phi$ back from the terminal
  time}
  
We now prove a result which shows that if $\phi \in \CS_\alpha$ then
with high probability both $\|\Pi K_{T-\delta,T}\phi\|$ and the ratio $\|\Pi
K_{T-\delta,T}\phi\|/\| K_{T-\delta,T}\phi\|$ can not change
dramatically for small enough $\delta$. This allows us to step back from the
terminal time $T$ to the right end point of the time interval
$\II$. As mentioned at the start of this section, this is needed to
allow the rougher test functions used in Theorem~\ref{theo:Malliavin}.

\begin{lemma}\label{lem:closing}
  Let \eqref{e:lipDeltaAdj} hold and fix any orthogonal projection
  $\Pi$ of $\CH$ onto a finite dimensional subspace of $\CH$ spanned
  by elements of $\CH_{1}$.  Recall furthermore the
  relation \eref{e:defdelta} between $\delta$ and $\eps$. There exists
  a constant $c\in(0,1)$ such that, for every $r >0$  and every $\alpha >0$, the implication
  \begin{equ}
    \phi \in \CS_\alpha \quad\Longrightarrow\quad K_{T-\delta,T} \phi
    \in \CS_{c\alpha} \quad\text{and}\quad \|\Pi K_{T-\delta,T} \phi
    \| \geq \frac{\alpha}2 \|\phi\|\;,
  \end{equ}
  holds modulo a $\Psi^{1/r}$-controlled family of negligible events.
\end{lemma}

To prove this Lemma, we will need the following axillary lemma whose
proof is given at the end of the section.
\begin{lemma}\label{lem:lipDeltaAdj} For any $\delta \in (0,T/2]$,
  one has the bound \minilab{e:aprioriK2}
  \begin{equs}
    \EE \sup_{\|\phi\| \le 1} \|K_{T-\delta,T}\phi- e^{-\delta L}
    \phi\|^p  &\le C_p \Psi^{np}(u_0) \delta^{(1-a)p}\;, \label{e:boundKphi} \\
    \EE \sup_{\|\phi\| \le 1} \|K_{T-\delta,T}\phi - \phi\|_{-1}^p
    &\le {C_p\Psi^{np}(u_0) \delta^{(1-a)p}}\;, \label{e:lipDeltaAdj}
  \end{equs}
  for every $p \ge 1$ and every $u_0 \in \CH$. Here, $n$ is the degree of the nonlinearity $N$.
\end{lemma}

\begin{proof}[of Lemma~\ref{lem:closing}]
  We begin by showing that, modulo some $\Psi^{1/r}$-dominated family
  of negligible events,
  \begin{equ}
    \|\Pi \phi\| \geq \alpha \|\phi\| \quad\Longrightarrow\quad \|\Pi
    K_{T-\delta,T} \phi \| \geq \frac{\alpha}2 \|\phi\|\;.
  \end{equ}
  By the assumption on $\Pi$, we can find a collection
  $\{v_k\}_{k=1}^N$ in $\CH_{1}$ with $\|v_k\| = 1$ such that $\Pi
  \phi = \sum_{k} v_k \scal{v_k, \phi}$. Therefore, there exists a
  constant $C_1 = \sup_k \|v_k\|_{1}$ so that $\|\Pi \phi\| \leq C_1
  \|\phi\|_{-1}$. Combining Lemma~\ref{lem:stupid} with
  Lemma~\ref{lem:lipDeltaAdj}, we see that
  \begin{equ}[e:negligible]
    \sup_{\phi \in \CH\,:\, \|\phi\| = 1} \|K_{T-\delta,T}\phi -
    \phi\|_{-1} \le \frac\alpha{2 C_1} \;,
  \end{equ}
  modulo a $\Psi^{n \over (1-a)r}$-dominated family of negligible
  events. Hence, modulo the same family
  of events,
  \begin{equs}
    \|\Pi K_{T-\delta,T} \phi \| &\geq \|\Pi \phi\| - C_1
     \|K_{T-\delta,T}\phi -  \phi\|_{-1}\\
    &\geq \alpha \|\phi\| - \frac\alpha2 \|\phi\| = \frac{\alpha}2
    \|\phi\|\;.
  \end{equs}
  Combining now Lemma~\ref{lem:stupid} with \eref{e:boundKphi}, we
  see that
  \begin{equ}
    \|K_{T-\delta,T} \phi \| \le \|\phi\| + \|e^{-\delta L} \phi\|
    \le C \|\phi\|\;,
  \end{equ}
  modulo a $\Psi^{n\over (1-a) r}$-dominated family of negligible
  events, thus showing that $K_{T-\delta,T} \phi$ belongs to $\CS_{c\alpha}$  with $c = 1/(2C)$
  and concluding the proof.
\end{proof}

We now give the proof of the auxiliary lemma used in the proof of Lemma~\ref{lem:closing}.
\begin{proof}[of Lemma~\ref{lem:lipDeltaAdj}]
It follows from \eref{e:defKst} and the variation of constants formula that 
\begin{equ}
K_{T-\delta,T}\phi - e^{-\delta L}\phi = \int_{T-\delta}^T e^{-(T-s)L} DN^*(u_s)K_{s,T}\phi\,ds\;.
\end{equ}
It now follows from Assumption~\ref{ass:basic}, point 3 that there exists $\gamma_0 \in [0,\gamma_\star + 1)$
such that $DN^*(u)$ is a bounded linear map from $\CH$ to $\CH_{-a}$ for every $u \in \gamma_0$
and that its norm is bounded by $C \|u\|_{\gamma_0}^{n-1}$ for some constant $C$.
The first bound then follows by combining Proposition~\ref{prop:apriori} with the fact that
$e^{-Lt}$ is bounded by $Ct^{-a}$ as an operator from $\CH_{-a}$ to $\CH$ as a consequence
of standard analytic semigroup theory  \cite{Ka}.

In order to obtain the second bound, we write
  \begin{equs}
    \|K_{T-\delta,T}\phi - \phi\|_{-1} &\le \|K_{T-\delta,T}\phi -
    e^{-L \delta}\phi\|_{-1} +
    \|e^{-L \delta} \phi - \phi\|_{-1} \\
    &\le \|K_{T-\delta,T}\phi - e^{-L \delta}\phi\| + C \delta\;,
  \end{equs}
  where the last inequality is again a consequence of standard analytic
  semigroup theory.  The claim then follows from
  \eref{e:boundKphi}.
\end{proof}

\subsection{The smallness of $\MM_T$ implies the smallness of $\CQ_N(u_{T-\delta})$}

In this section, we show that if $\scal{\MM_T \phi,\phi}$ is small
then $\scal{\CQ_N(u_{t})K_{t,T}\phi,K_{t,T}\phi}$
must also be small with high probability for every $t \in I_\delta$. The precise statement is given by the following result:
\begin{lemma}\label{lem:contradiction}
  Let the Malliavin matrix $\MM_T$ be defined as in \eref{e:defM} and
  assume that Assumptions~\ref{ass:basic} and \ref{ass:global} are satisfied.  Then, for
  every $N>0$, there exist $r_N>0$, $p_N > 0$ and $q_N > 0$ such
  that, provided that $r \le r_N$, the implication
  \begin{equ}
    \scal{\phi, \MM_T\phi} \le \eps \|\phi\|^2 \quad\Longrightarrow\quad
    \sup_{Q \in \AA_N}\sup_{t\in \II} |\scal{K_{t,T}\phi, Q(u_t)}| \le
    \eps^{p_N} \|\phi\|\;,
  \end{equ}
  holds modulo some $\Psi^{q_N}$-dominated negligible family of
  events.
\end{lemma}

\begin{proof}
  The proof proceeds by induction on $N$ and the steps of this induction are 
  the content of the next two subsections. 
  Since $\AA_1 =
  \{g_1,\ldots,g_d\}$, the case $N = 1$ is implied by
  Lemma~\ref{lem:primer}  below, with $p_1 = 1/4$, $q_1 = 8$, and $r_1 =
  1/(8\bar p_\beta)$.  
  
  The inductive step is then given by combining
  Lemmas~\ref{lem:recursion1} and \ref{lem:recursion2} below.  At each step,
  the values of $p_n$ and $r_n$ decrease while $q_n$ increases, but
  all remain strictly positive and finite after finitely many steps.
\end{proof}

\subsection{The first step in the iteration}

The ``priming step'' in the inductive proof of Lemma~\ref{lem:contradiction}
 follows from the fact that the directions which are directly forced by the Wiener processes are 
 not too small with high probability.
\begin{lemma}\label{lem:primer}
  Let the Malliavin matrix $\MM$ be defined as in \eref{e:defM} and
  assume that Assumptions~\ref{ass:basic} and \ref{ass:global} are satisfied. Then,
  provided that $r \le 1/(8\bar p_\beta)$, the implication
  \begin{equ}
    \scal{\phi, \MM_T\phi} \le \eps \|\phi\|^2
    \quad\Longrightarrow\quad \sup_{k = 1\ldots d}\;\sup_{t\in \II}
    |\scal{K_{t,T}\phi, g_k}| \le \eps^{1/4} \|\phi\|\;,
  \end{equ}
  holds modulo some $\Psi^{8}$-dominated negligible family of events.
  Here, $\bar p_\beta$ is as in \eref{e:apadj} and $\beta$ was fixed
  in \eref{e:gammabeta}.
\end{lemma}

\begin{proof} For notational compactness, we scale $\phi$ to have norm
  one by replacing $\phi$ with $\phi/\|\phi\|$. We will still refer to
  this new unit vector as $\phi$. Now assume that $\scal{\phi, \MM_T\phi}
  \le \eps$. It then follows from \eref{e:defM} that
  \begin{equ}
    \sup_{k =1\ldots d} \int_\II \scal{g_k, K_{t,T} \phi}^2\, dt \le
    \eps\;.
  \end{equ}
  Applying Lemma~\ref{lem:dtf} with $f(t) = \int_{T/2}^t |\scal{g_k, K_{s,T} \phi}|\, ds$
  and $\alpha = 1$, it follows that there exists a constant $C>0$ such that, for every $k = 1\ldots d$,
  either
  \begin{equ}
    \sup_{t \in \II} |\scal{g_k, K_{t,T} \phi}| \le \eps^{1/4}\;,
  \end{equ}
  or
  \begin{equ}[e:smalleventgK]
    \/ \scal{g_k, K_{\cdot,T} \phi} \/_1 \ge C\eps^{-1/4}\;.
  \end{equ}
  Therefore, to complete the proof, we need only to show that the
  latter events form a $\Psi^{4}$-dominated negligible family for
  every $k$.  Since $\/\scal{g_k, K_{\cdot,T} \phi}\/_1 \le
  \|g_k\|_{-\beta} \/K_{\cdot ,T}\phi\/_{1,\beta}$, the bound
  \eref{e:smalleventgK} implies that
  \begin{equ}[e:smalleventgK2]
    \sup_{\phi \in \CH\,:\, \|\phi\| = 1} \/K_{t,T}\phi\/_{1,\beta}
    \ge {C\eps^{-1/4} \over g^*}\;,
  \end{equ}
  where $g^* = \max_{k} \|g_k\|_{-\beta}$ (which is finite since we
  have by assumption that $-\beta \le \gamma + 1 < \gamma_\star + 1$
  and since $g_k \in \CH_{\gamma_\star+1}$ for every $k$) . This event
  depends only on the initial condition $u_0$ and on the model under
  consideration. In particular, it is independent of $\phi$.

The claim now follows from the \textit{a priori} bound \eref{e:apadj}
and Lemma~\ref{lem:stupid} with $q = {1\over 4}$ and $b = \bar p_\beta$.
\end{proof}

\subsection{The iteration step}
\label{sec:iterate}

Recall that we consider evolution equations of the type
\begin{equ}[e:main2]
  du_t = F(u_t)\, dt + \sum_{k=1}^d g_k dW_k(t)\;,
\end{equ}
where $F$ is a ``polynomial'' of degree $n$.
The aim of this section is to implement the following recursion: if,
for any given polynomial $Q$, the expression
$\scal{Q(u_t), K_{t,T} \phi}$ is ``small'' in the supremum norm, then
both the expression $\scal{[Q,F](u_t), K_{t,T} \phi}$ and
$\scal{[Q,g_k](u_t), K_{t,T} \phi}$ must be small in the supremum norm
as well.

The main technical tool used in this section will be the estimates on
``Wiener polynomials'' from Section~\ref{sec:Wiener}. Using the notation
\begin{equ}
W_\alpha(t) \eqdef W_{\alpha_1}(t)\, W_{\alpha_2}(t)\cdots W_{\alpha_\ell}(t)\;,
\end{equ}
for a multi-index $\alpha = (\alpha_1,\ldots,\alpha_\ell)$, this estimate states that
if an expression of the type $\sum_{|\alpha|\le m} A_\alpha(t)W_\alpha(t)$ is small,
then, provided that the processes $A_\alpha$ are sufficiently regular in time, each
of the $A_\alpha$ must be small. In other words, two distinct monomials in a Wiener polynomial
cannot cancel each other out. Here, the processes $A_\alpha$ do not have to be
adapted to the filtration generated by the $W_k$, so this gives us some kind of anticipative
replacement of Norris' lemma. The main trick that we use in order to take advantage of such a result
is to switch back and forth between considering the process $u_t$ solution to \eref{e:main2}
and the process $v_t$ defined by
\begin{equ}
v_t \eqdef u_t - \sum_{k=1}^d g_k W_k(t)\;,
\end{equ}
which has more time-regularity than $u_t$. Recall furthermore that given a polynomial $Q$ and
a multi-index $\alpha$, we denote by $Q_\alpha$ the corresponding term \eref{e:repeatedBrakets} appearing in the (finite)
Taylor expansion of $Q$.

Recall the definition $\Poly^m(\gamma,\beta) = \Poly^m(\CH_\gamma,\CH_{-\beta-1}) \cap \Poly^m(\CH_{\gamma+1},\CH_{-\beta})$. We first show that if $Q \in \Poly^m(\gamma,\beta)$ and 
$\scal{Q(u_t), K_{t,T} \phi}$ is small, then the expression $\scal{Q_\alpha(v_t), K_{t,T} \phi}$ 
(note the appearance of $v_t$ rather than $u_t$) must be small as well
for every multi-index $\alpha$:

\begin{lemma}\label{lem:recursion0}
  Let $Q \in \Poly^m(\gamma,\beta)$ for some $m \ge 0$ and for
  $\gamma$ and $\beta$ as chosen in \eref{e:gammabeta}. Let
  furthermore $q > 0$ an set $\bar q = q 3^{-m}$. Then, the implication
  \begin{equ}
    \sup_{t \in \II} |\scal{Q(u_t), K_{t,T} \phi}| \le \eps^q\|\phi\|
    \quad\Longrightarrow\quad \sup_{\alpha} \sup_{t \in \II}
    |\scal{Q_\alpha(v_t), K_{t,T} \phi}| \le \eps^{\bar q}
    \|\phi\|\;,
  \end{equ}
  holds modulo some $\Psi^{6 (m+1) / \bar q}$-dominated negligible
  family of events, provided that $r < \bar q / (6\bar
  p_\beta)$.
\end{lemma}

\begin{proof}
  Note first that both inner products appearing in the implication are
  well-defined by Proposition~\ref{prop:apriori} and the assumptions
  on $Q$.  By homogeneity, we can assume that $\|\phi\| = 1$.  Since
  $Q$ is a polynomial, \eref{e:TaylorQ} implies that
  \begin{equ}
    \scal{Q(u_t), K_{t,T} \phi} = \sum_\alpha \scal{Q_\alpha(v_t),
      K_{t,T} \phi} W_\alpha(t)\;.
  \end{equ}
  Applying Theorem~\ref{theo:Norris}, we see that, modulo some
  negligible family of events $\Osc^m$, $\sup_{t \in \II}
  |\scal{Q(u_t), K_{t,T} \phi}| \le \eps^q$ implies that either
  \begin{equ}[e:goodterm]
    \sup_\alpha \sup_{t \in \II} |\scal{Q_\alpha(v_t), K_{t,T} \phi}|
    \le \eps^{\bar q}\;,
  \end{equ}
  or there exists some $\alpha$ such that
  \begin{equ}[e:largeLip]
    \/ \scal{Q_\alpha(v_t), K_{\cdot ,T} \phi}\/_1 \ge \eps^{- \bar q / 3}\;.
  \end{equ}
  We begin by arguing that the second event is negligible.  Since $Q$
  is of degree $m$, there exists a constant $C$ such that
  \begin{equs}
    \/\scal{Q_\alpha(v_t), &K_{\cdot,T} \phi}\/_1 \le \sup_{t \in \II}
    \|K_{t,T}\phi\|_{\beta+1} \/Q_\alpha(v_\cdot)\/_{1,-\beta-1} + \sup_{t
      \in \II} \|Q_\alpha(v_t)\|_{-\beta} \/K_{\cdot,T}\phi\/_{1,\beta} \\
    &\le C\sup_{t \in \II} \|K_{t,T}\phi\|_{\beta+1}
   \sup_{t \in \II} \|v_t\|_{\gamma}^{m-1} \/v\/_{1,\gamma} + C\sup_{t \in \II}
    \|v_t\|_{\gamma+1}^{m} \/K_{\cdot,T}\phi\/_{1,\beta}\;.
  \end{equs}
  Here, we used the fact that $Q_\alpha \in \Poly^m(\CH_\gamma,
  \CH_{-1-\beta})$ to bound the first term and the fact that $Q_\alpha
  \in \Poly^m(\CH_{\gamma+1}, \CH_{-\beta})$ to bound the second term.
  The fact that $Q_\alpha$ belongs to these spaces is a consequence of
  $g_k \in \CH_{\gamma_\star + 1}$ and of the definition
  \eref{e:TaylorQ} of $Q_\alpha$.
  
  Therefore, \eref{e:largeLip} implies that either
  \begin{equ}[e:eventX]
    X_\delta \eqdef \sup_{\phi \in \CH\,:\, \|\phi\| = 1}\sup_{t \in
      \II} \|K_{t,T}\phi\|_{\beta+1} \sup_{t \in
      \II}\|v_t\|_{\gamma}^{m-1}
    \/v\/_{1,\gamma} \ge {1\over 2C} \eps^{- \bar q /3}
  \end{equ}
  or
  \begin{equ}[e:eventY]
    Y_\delta \eqdef \sup_{\phi \in \CH\,:\, \|\phi\| = 1} \sup_{t \in
      \II} \|v_t\|_{\gamma+1}^{m} \/K_{\cdot,T}\phi\/_{1,\beta} \ge
    {1\over 2C} \eps^{- \bar q / 3}\;.
  \end{equ}
  Combining the Cauchy-Schwarz inequality with \eref{e:apadj} of
  Proposition~\ref{prop:apriori}, we see that $X_\delta$ and
  $Y_\delta$ satisfy the assumptions of Lemma~\ref{lem:stupid} with
  $\Phi = \Psi^{m+1}$ and $b = \bar p_\beta$, thus showing that the families of
  events \eref{e:eventX} and \eref{e:eventY} are both $\Psi^{6 (m+1)
     / \bar q}$-dominated negligible, provided that $r < \bar q
 / (6\bar p_\beta)$.
\end{proof}

In the sequel, we will need the follow simple result which is, in
some way, a converse to Theorem~\ref{theo:Norris}.
\begin{lemma}\label{lem:weinerBuildUp}
  Given any integer $N>0$ and any two exponents $0 < \bar q < q$, there exists a universal family of negligible events
  $\textrm{Sup}_W^N$ such that the
  implication 
  \begin{equ}{}
    \sup_\alpha \sup_{t \in \II} |A_\alpha(t)| < \eps^q \qquad
    \Longrightarrow \qquad \sup_{t \in \II} \Big|\sum_{\alpha : |\alpha| \leq N} A_\alpha(t) W_\alpha(t)\Big|
    <\eps^{\bar q}
  \end{equ}
  holds modulo $\textrm{Sup}_W^N$ for any collection of processes $\{A_\alpha(t) : |\alpha| \leq N\}$.
\end{lemma}

\begin{proof}
  Observe that
  \begin{equ}{}
    \sup_{t \in \II}\Big|\sum_{\alpha : |\alpha| \leq N} A_\alpha(t) W_\alpha(t)\Big|
    \leq \left(\sup_\alpha \sup_{t \in \II} |A_\alpha(t)| \right)
    \Bigg(\sum_{\alpha : |\alpha| \leq N} \sup_{t \in \II} |W_\alpha|\Bigg)
  \end{equ}
  Since for any $p>0$,
  \begin{equ}{}
    \sum_{\alpha : |\alpha| \leq N} \sup_{t \in \II} |W_\alpha| > \eps^{-p}
  \end{equ}
  is a negligible family of events, the claim follows at once.
\end{proof}

As a corollary to Lemmas \ref{lem:recursion0} and
\ref{lem:weinerBuildUp}, we now obtain the key estimate for
Lemma~\ref{lem:contradiction} in the particular case where the
commutator is taken with one of the constant vector fields:

\begin{lemma}\label{lem:recursion1}
  Let $Q\in  \Poly^m(\gamma,\beta)$ be a polynomial of degree $m$
  and let $q > 0$. Then, for $\bar q = q 3^{-(m+1)}$, the implication
  \begin{equ}
    \sup_{t \in \II} |\scal{Q(u_t), K_{t,T} \phi}| \le \eps^q\|\phi\|
    \; \Longrightarrow\; \sup_{\alpha} \sup_{t \in \II}
    |\scal{Q_\alpha (u_t), K_{t,T} \phi}| \le \eps^{\bar q}\|\phi\|\;,
  \end{equ}
  holds for all $\phi \in \CH$ modulo some $\Psi^{2 (m+1) / \bar
    q}$-dominated negligible family of events, provided that $r < \bar
  q / (2\bar p_\beta)$.
\end{lemma}

\begin{proof}
  Since it follows from \eref{e:TaylorQ} that $(Q_\alpha)_\beta =
  Q_{\alpha \cup \beta}$, we have the identity
  \begin{equ}{} Q_\alpha(u_t) = \sum_\beta
    (Q_\alpha)_\beta(v_t)W_\beta = \sum_\beta Q_{\alpha
      \cup \beta} (v_t)W_\beta \;.
  \end{equ}
  Combining Lemma~\ref{lem:recursion0} and
  Lemma~\ref{lem:weinerBuildUp} with $N=m$ proves the claim.
\end{proof}

In the next step, we show a similar result for the commutators between $Q$ and $F$.
We are going to use the fact that if a function $f$
is differentiable with H\"older continuous derivative, then $f$ being
small implies that $\d_t f$ is small as well, as made precise by
 Lemma~\ref{lem:dtf}. As previously, we start by showing a result that involves the process
 $v_t$ instead of $u_t$:

\begin{lemma}\label{lem:BasicFinteration}
   Let $Q$ be as in Lemma~\ref{lem:recursion0} and such that
  $[Q_\alpha,F_\sigma] \in \Poly(\gamma,\beta)$ for any two multi-indices
  $\alpha, \sigma$. Let furthermore $q > 0$ and set $\bar q=q 3^{-2m}/8$. Then the implication 
  \begin{equ}
        \sup_{t \in \II} |\scal{Q(u_t), K_{t,T} \phi}| \le \eps^q\|\phi\|
    \quad \Longrightarrow\quad \sup_{\alpha,\sigma} \sup_{t \in \II}
    |\scal{[Q_\alpha,F_\sigma](v_t), K_{t,T} \phi}| \le \eps^{\bar q}
    \|\phi\|\;,
  \end{equ}
 holds modulo some $\Psi^{6 (m+1) / \bar q}$-dominated negligible
  family of events, provided that $r < \bar q / (6\bar
  p_\beta)$. (As before the empty multi-indices are included in the supremum.)  
\end{lemma}

\begin{proof}
   By homogeneity, we can assume that $\|\phi\| = 1$.  Combining
  Lemma~\ref{lem:recursion0} with Lemma~\ref{lem:dtf} and defining
  $\hat q = q 3^{-m}$, we obtain that $\sup_{t \in \II} |\scal{Q(u_t),
    K_{t,T} \phi}| \le \eps^q$ implies for $f_{\alpha,\phi}(t) \eqdef
  \d_t \scal{Q_\alpha(v_t), K_{t,T} \phi}$ the bound
  \begin{equ}[e:boundfaphi]
    \sup_{t \in \II} |f_{\alpha,\phi}(t)| \le C
    \max\bigl\{\eps^{\hat q}, \eps^{\hat q\over 4}
    \/f_{\alpha,\phi}\/_{1/3}^{3/4}\bigr\}\;,
  \end{equ}
  modulo some $\Psi^{6 (m+1) / \hat q}$-dominated negligible family of
  events, provided that $r \le \hat q / (6\bar p_\beta)$. Note that
  this family is in particular independent of both $\alpha$ and
  $\phi$.  Here and in the sequel, we use the letter $C$ to denote a generic constant
  depending on the details of the problem that may change from one expression to the next.
  
One can see that $\scal{Q_\alpha(v_t), K_{t,T}
    \phi}$ is differentiable in $t$ by combining
  Proposition~\ref{prop:apriori} with the fact that $Q_\alpha \in
  \Poly(\CH_\gamma, \CH_{-1-\beta})\cap \Poly(\CH_{\gamma+1},
  \CH_{-\beta})$ as in the proof of Lemma~\ref{lem:recursion0}. See
  \cite{DauLio5} for a more detailed proof of a similar statement.

  Computing the derivative explicitly, we obtain
  \begin{equ}
    f_{\alpha,\phi}(t) = \scal{DQ_\alpha(v_t) F(u_t) - DF(u_t)
      Q_\alpha(v_t), K_{t,T} \phi} \eqdef \scal{B_\alpha(t), K_{t,T} \phi}
    \;.
  \end{equ}
  The function $B_\alpha$ can be further expanded to
  \begin{equ}
    B_\alpha(t) = \sum_\sigma \bigl(DQ_\alpha(v_t) F_\sigma (v_t) - DF_\sigma
    (v_t) Q_\alpha(v_t)\bigr) \, W_\sigma(t) = \sum_\sigma [Q_\alpha,
    F_\sigma](v_t) W_\sigma(t)\;.
  \end{equ}
  Notice that, by the assumption that $[Q_\alpha, F_\sigma] \in \Poly(\gamma,\beta)$, one has
  \begin{equs}
    \/[Q_\alpha, F_\sigma](v_\cdot) W_\sigma(\cdot)\/_{{1\over
        3},-1-\beta} &\le C(1+\sup_{t \in \II} \|v_t\|_{\gamma})^{n+m-2 -
      |\alpha| - |\sigma|}
    \|\d_t v_t\|_{\gamma}\sup_{t \in \II}|W_\sigma(t)| \\
    &\quad + C\/W_\sigma\/_{1\over 3} (1+\sup_{t \in \II}
    \|v_t\|_{\gamma})^{n+m-1-|\alpha| - |\sigma|}\;, \\
    \|[Q_\alpha, F_\sigma](v_t) W_\sigma(t)\|_{-\beta} &\le C
    (1+\|v_t\|_{\gamma+1})^{n+m-1-|\alpha|-|\sigma|} |W_\sigma(t)|\;.
  \end{equs}
  (Here it is understood that if one of the exponents of the norm of $v_t$
  is negative, the
  term in question actually vanishes.)  It therefore follows from
  Proposition~\ref{prop:apriori} that
  \begin{equ}
    \E \/B_\alpha\/_{{1\over 3},-1-\beta}^p \le C_p \Psi^{(n+m-1)p}(u_0)
    \;,\qquad \E \sup_{t \in \II}\|B_\alpha(t)\|_{-\beta}^p \le C_p
    \Psi^{(n+m-1)p}(u_0) \;,
  \end{equ}
  for every $p\ge 1$ and some constants $C_p$.

  Since the H\"older norm of $f_{\alpha, \phi}$ is bounded
  by
  \begin{equ}
    \/\scal{B_\alpha(\cdot), K_{\cdot,T} \phi}\/_{1\over 3} \le
    \/B_\alpha\/_{{1\over 3},-1-\beta} \sup_{t \in \II} \|K_{t,T}\|_{\beta+1}
    + \/K_{\cdot,T}\/_{{1\over 3},\beta} \sup_{t \in
      \II}\|B_\alpha(t)\|_{-\beta} \;,
  \end{equ}
we can use the bounds on $B_\alpha$ just obtained, the
  Cauchy-Schwarz inequality, Proposition~\ref{prop:apriori}, and
  Lemma~\ref{lem:stupid}, to obtain
  \begin{equ}[e:bounddoublebracket]
    \sup_{\alpha} \sup_{\|\phi\| \le 1}
    \/f_{\alpha,\phi}\/_{1/3}^{3/4} \le \eps^{-{\hat q\over 8}}\;,
  \end{equ}
  modulo some $\Psi^{12 (n+m) / \hat q}$-dominated negligible family
  of events, provided that $r \le \min\{\hat q / 12, \hat q / (6\bar
  p_\beta)\}$.  As a consequence, modulo this family, we obtain from
  \eref{e:boundfaphi} the bound $\sup_\alpha \sup_{t \in \II}
  |f_{\alpha,\phi}(t)| \le C \eps^{\hat q\over 8}$ which can be rewritten
  as
  \begin{equ}[e:QalphaFsigmaBound0]
    \sup_\alpha \sup_{t \in \II} \left|\sum_\sigma \scal{[Q_\alpha,
    F_\sigma](v_t), K_{t,T} \phi} W_\sigma(t)\right| \le C \eps^{\hat q\over
      8}\;.
  \end{equ}
  Since $[Q_\alpha, F_\sigma] \in \Poly(\gamma,\beta)$ the same
  reasoning as in Lemma~\ref{lem:recursion0} combined with
  Theorem~\ref{theo:Norris} on Wiener polynomials implies that modulo
  some negligible family of events $\Osc^m$, the estimate
  \eref{e:QalphaFsigmaBound0} implies that either
  \begin{equ}[e:goodterm2]
    \sup_{\alpha,\sigma} \sup_{t \in \II} |\scal{[Q_\alpha,
    F_\sigma](v_t), K_{t,T} \phi}|
    \le \eps^{\bar q}\;,
  \end{equ}
  or there exists some $\alpha$ and $\sigma$ such that
  \begin{equ}[e:largeLip2]
    \/ \scal{[Q_\alpha,
    F_\sigma](v_t), K_{\cdot ,T} \phi}\/_1 \ge \eps^{-\bar q/3}\;.
  \end{equ}
  Again following the same logic as Lemma~\ref{lem:recursion0}, we see
  that the family of events in \eref{e:largeLip2} is $\Phi^{6 (m+1) /\bar
    q}$-dominated negligible provided that $r < \bar q /(6
  \bar p_\beta)$.
\end{proof}

In order to turn this result into a result involving the process $u_t$, we need the following expansion:

\begin{lemma}\label{lem:BracketRep}
  Given any two multi-indices $\alpha$ and $\sigma$ (including the
  empty indices), there exist an integer $N$ and a collection of
  multi-indices $\{\alpha_i, \sigma_i, \zeta_i : i=1\ldots N\}$ and
  constants $\{c_i : i=1\ldots N\}$ so that
  \begin{equ}{}
    [Q_\alpha,F_\sigma](u_t) = \sum_{i=1}^N c_i [Q_{\alpha_i},F_{\sigma_i}](v_t)
    W_{\zeta_i}(t) 
  \end{equ}
\end{lemma}
\begin{proof}
  First observe that
  \begin{equ}{}
    [Q_\alpha,F_\sigma](u_t)=\sum_\zeta  [Q_\alpha,F_\sigma]_\zeta(v_t) W_\zeta(t)\;.
  \end{equ}
  The Jacobi identity for Lie bracket states that 
  \begin{equs}{}
   \DF_{g_k}[Q_\alpha,F_\sigma]&=[g_k, [Q_\alpha,F_\sigma]] =[ [g_k,Q_\alpha], F_\sigma] + [Q_\alpha,[g_k,F_\sigma]]\\
    &=(|\alpha|+1)[Q_{\alpha \cup(k)},F_\sigma] +
    (|\sigma|+1)[Q_\alpha, F_{\sigma\cup(k)}]\;.
\end{equs}
By iterating this calculation, we see that for any multi-index $\zeta$,
$[Q_\alpha,F_\sigma]_\zeta$ is equal to some linear combination of a
finite number of terms of the form $[Q_{\alpha_i},F_{\sigma_i}]$ for
some multi-indices $\alpha_i$ and $\sigma_i$.
\end{proof}

In very much the same way as before, it then follows that:

\begin{corollary}\label{lem:recursion2}
  Let $Q$ be as in Lemma~\ref{lem:recursion0} and such that
  $[Q_\alpha,F_\sigma] \in \Poly(\gamma,\beta)$ for any two
  multi-indices $\alpha, \sigma$. Let furthermore $q > 0$ and set $\bar q =q
      3^{-2(m+1)}/8$. Then the implication
  \begin{equ}
    \sup_{t \in \II} |\scal{Q(u_t), K_{t,T} \phi}| \le \eps^q\|\phi\|
    \quad \Longrightarrow\quad \sup_{\alpha,\sigma} \sup_{t \in \II}
    |\scal{[Q_\alpha,F_\sigma](u_t), K_{t,T} \phi}| \le \eps^{\bar q}
    \|\phi\|\;,
  \end{equ}
  holds modulo some $\Psi^{2 (m+1) / (3\bar q)}$-dominated
  negligible family of events, provided that $r < 3\bar q /
  (2\bar p_\beta)$.
\end{corollary}

\begin{proof}
  It follows from Lemma~\ref{lem:BracketRep} that 
  \begin{align*}
    \scal{[Q_\alpha,F_\sigma](u_t),K_{T,t}\phi} = \sum_{i=1}^N c_i
    \scal{[Q_{\alpha_i},F_{\sigma_i}](v_t),K_{T,t}\phi}
    W_{\gamma_i}(t)\;.
  \end{align*}
  Combining the control of the
  $\scal{[Q_{\alpha_i},F_{\sigma_i}](v_t),K_{T,t}\phi}$ obtained in
  Lemma~\ref{lem:BasicFinteration} with Lemma~\ref{lem:weinerBuildUp}
  gives the quoted result.
\end{proof}

\subsection{Putting it all together: proof of Theorem~\ref{thm:Malliavin2}}

We now finally combine all of the results we have just accumulated
 to give the proof of the main theorem of these
sections.
\begin{proof}[of Theorem~\ref{thm:Malliavin2}]
  We are going to prove the statement by showing that there exists $\theta>0$ and, for every $\alpha>0$,
   a $\Psi^{\theta}$-dominated family of negligible events such that,
  modulo this family, the assumption $\inf_{\phi \in \CS_\alpha} \scal{\phi,
    \MM_T \phi} \le \eps\|\phi\|^2$ leads to a contradiction for all $\eps$
  sufficiently small.

From now on, fix $N$ as in Assumption~\ref{ass:Hormander}.
  By Lemmas~\ref{lem:closing} and \ref{lem:contradiction}, we see that
  there exist constants $\theta, q, r_0 > 0$ such that, modulo some
  $\Psi^{\theta}$-dominated family of negligible events, one has the
  implication
  \begin{equ}
    \left.\begin{array}{r}  \phi \in \CS_\alpha   \\[0.3em]
        \scal{\phi, \MM_T\phi} \le \eps
        \|\phi\|^2 \end{array}\right\} \;\Longrightarrow\;
    \left\{\begin{array}{l} K_{T-\delta, T}\phi \in \CS_{c\alpha}
        \;\text{and}\; \|\Pi
        K_{T-\delta,T} \phi \|  \geq \frac{\alpha}2 \|\phi\| \\[0.3em]
        \scal{K_{T-\delta, T}\phi, \CQ_N (u_{T-\delta}) K_{T-\delta,
            T}\phi} \le \eps^q\|\phi\|^2\;,
      \end{array}\right. 
  \end{equ}
  provided that we choose $r \le r_0$ in the definition
  \eref{e:defdelta} of $\delta$.  By Assumption~\ref{ass:Hormander},
  this in turn implies (modulo the same family of negligible events)
  \begin{equ}
    \cdots \quad\Longrightarrow\quad \frac{\alpha}2 \|\phi\| \le \|\Pi
    K_{T-\delta, T}\phi\| \le \Lambda_{c\alpha}^{-1}(u_0) \eps^{q\over
      2}\|\phi\|\;.
  \end{equ}
  On the other hand, it follows from Lemma~\ref{lem:stupid} and the
  assumption on the inverse moments of $\Lambda_{c\alpha}$ that,
  modulo some $\Psi^{4\over q}$-dominated family of negligible events,
  one has the bound
  \begin{equ}
    \Lambda_{c\alpha}^{-1}(u_0) \le \eps^{- {q\over 4}}\;.
  \end{equ}
  Possibly making $\theta$ smaller, it follows that, modulo some
  $\Psi^{\theta}$-dominated family of negligible events, one has the
  implication
  \begin{equ}
    \left.\begin{array}{r}  \phi \in \CS_\alpha   \\[0.3em]
        \scal{\phi, \MM_T\phi} \le \eps
        \|\phi\|^2 \end{array}\right\} \;\Longrightarrow\; {\alpha
      \over 2} \le \eps^{q\over 4}\;,
  \end{equ}
  which cannot hold for $\eps$ small enough, thus concluding the proof
  of Theorem~\ref{thm:Malliavin2}
\end{proof}

\section{Bounds on Wiener polynomials}
\label{sec:Wiener}

We will use the terminology of ``negligible sets'' introduced in
Definition~\ref{def:negligible}. We will always work on the time
interval $[0,1]$, but all the results are independent (modulo change
of constants) of the time interval, provided that its length is
bounded from above and from below by two positive constants
independent of $\eps$.  This is seen easily from the scaling
properties of the Wiener process. 

The results of this section are descendents of similar results obtained in
\cite{MatPar06:1742,Yuri} by related techniques. In \cite{Yuri} it was
proven that if a Wiener polynomial, with continuous, bounded variation
coefficients, is identically zero on an interval then so are its
coefficients. This is enough to prove the almost sure invertibility of 
projections of the Malliavin matrix, which in turn implies the
existence of a density for the projections of the transition
probabilities. To prove smoothness of the densities or the ergodic results
of this paper, more quantitative control is needed. In \cite{Yuri}, a
result close to \eqref{theo:Norris} is claimed. However an error in
Lemma~9.12 of that article leaves the proof incomplete. Arguing along
similar, though slightly different lines, we prove the needed result
below. We build upon the presentation in \cite{Yuri} but simplify it
significantly. (The presentation in \cite{Yuri} was already a
significant simplification over that in \cite{MatPar06:1742}.)

\begin{theorem}\label{theo:Norris}
  Let $\{W_k\}_{k=1}^d$ be a family of i.i.d.\ standard Wiener
  processes and, for every multi-index $\alpha= (\alpha_1, \ldots,
  \alpha_\ell)$, define $W_\alpha = W_{\alpha_1}\ldots
  W_{\alpha_\ell}$ with the convention that $W_\alpha = 1$ if $\alpha
  = \emptyset$.  Let furthermore $A_\alpha$ be a family of (not
  necessarily adapted) stochastic processes with the property that
  there exists $m \ge 0$ such that $A_\alpha = 0$ whenever $|\alpha| >
  m$ and set $Z_A(t) = \sum_{\alpha} A_\alpha(t) W_\alpha(t)$.

  Then, there exists a universal family of negligible events $\Osc^m$
  depending only on $m$ such that the implication
  \begin{equ}[e:dichotomy]
    \|Z_A\|_{L^\infty} \le \eps \quad\Longrightarrow\quad
    \left\{\begin{array}{rl} \text{either}& \sup_\alpha
        \|A_\alpha\|_{L^\infty} \le \eps^{3^{-m}} \\ \text{or} &
        \sup_\alpha \|A_\alpha\|_{\Lip} \ge
        \eps^{-3^{-(m+1)}} \end{array}\right.
  \end{equ}
  holds modulo $\Osc^m$. (The supremum norms are taken on the interval
  $[0,1]$.)
\end{theorem}

\begin{remark}
  Informally, we can read the statement of Theorem~\ref{theo:Norris}
  as ``if $Z_A$ is small, then either all of the coefficients
  $A_\alpha$ are small, or at least one of them oscillates very
  fast.'' The exponents appearing in the statement of
  Theorem~\ref{theo:Norris} are somewhat arbitrary. By going through
  the proof more carefully, we can see that for any $\kappa > 2$, it
  is possible to find a constant $C_\kappa > 0$ such that the
  exponents in \eref{e:dichotomy} can be replaced by $\kappa^{-m}$ and
  $-C_\kappa \kappa^{-m}$ respectively. Here, the coefficient
  $C_\kappa$ tends to $0$ as $\kappa \to 2$. While the precise values
  of the exponents in \eref{e:dichotomy} arising from our proof are unlikely to
  be sharp, they are not far from it, as can be seen by
  looking at processes of the form $Z(t) = \eps^{1-{\theta\over 2}} (W_{
    \theta}(t) - W(t))$, where $W_{\theta} $ is the linear
  interpolation of the Wiener process $W$ over intervals of size
  $\eps^{\theta}$.
\end{remark}

\begin{remark}
  The reason why the family of negligible sets appearing in this
  statement is called $\Osc^m$ is that it relies on the fact that the
  Wiener processes typically fluctuate sufficiently fast on every
  small time interval so that their effects can be distinguished from
  those of the multiplicators $A_\alpha$ which fluctuate over much
  longer timescales. It is important to note that $\Osc^m$ depends on
  the processes $A_\alpha$ only through the value of~$m$.
\end{remark}

Before we start with the proof, we show the following result, which is
essentially the particular case of Theorem~\ref{theo:Norris} where $m
= 1$ and where the coefficients $A_\alpha$ do not depend on time.
Here, $\scal{\cdot,\cdot}$ denotes the scalar product in $\R^d$.

\begin{lemma}\label{lem:constant}
  Let $\{W_k\}_{k=1}^d$ be a collection of i.i.d.\ standard Wiener
  processes. Then, for any exponent $\kappa > 0$, there exists a
  universal family $\Osc$ of negligible events such that the
  bound
  \begin{equ}[e:boundwiggle]
    \sup_{t \in [0,1]} \bigl|\scal{A, W(t)}\bigr| \ge \eps^\kappa
    |A|\;,
  \end{equ}
  holds modulo $\Osc$ for any choice of coefficients $A \in
  \R^d$.
\end{lemma}

\begin{remark}
  We would like to stress again the fact that the family of events
  $\Osc$ is \textit{independent} of the choice of
  coefficients $A$ and depends only on the realisation of the $W_k$'s.
\end{remark}

\begin{proof}
  Fix $\kappa > 0$ and define a family of events $B$ by $B^\eps =
  \{\sup_{t \in [0,1]} |W(t)| \ge \eps^{-\kappa}\}$. It follows
  immediately from the fact that the supremum of a Wiener process has
  Gaussian tails that the family $B$ is negligible. Consider now the
  unit sphere $S^d$ in $\R^d$.  For every $A \in S^d$, the process
  $W_A(t) = \scal{A, W(t)}$ is a standard Wiener process and so $\P
  \bigl(\sup_{t \in [0,1]} |W_A(t)| \le 2\eps^{\kappa} \bigr) \le C_1
  \exp(-C_2 \eps^{-2\kappa})$ for some constants $C_1$ and $C_2$ that
  are independent of $A$. Denote this event by $H_A^\eps$.

  Choose now a collection $\{A_k\}$ of points in $S^d$ such that
  $\sup_{A \in S^d} \inf_k |A - A_k| \le \eps^{2\kappa}$ and define
  $H^\eps = \bigcup_k H_{A_k}^\eps$. Since this can be achieved with
  $\CO (\eps^{-2\kappa (d-1)})$ points, the family $H$ is negligible
  by Lemma~\ref{lem:sumup}.  We now define $\Osc = H \cup B$
  and we note that, modulo $\Osc$, one has for every $\bar A
  \in \R^d$ the bound
  \begin{equs}
    \sup_{t \in [0,1]} \bigl| \scal{\bar A, W(t)}\bigr| &\ge |\bar A| \inf_{A \in S^d} 
    \sup_{t \in [0,1]} \bigl|\scal{A, W(t)}\bigr| \\
    &\ge |\bar A| \Bigl(\inf_{k}\sup_{t \in [0,1]} \bigl|\scal{A_k,
      W(t)}\bigr| - \eps^\kappa\Bigr) \ge |\bar A| \eps^\kappa\;,
  \end{equs}
  as required.
\end{proof}

We now turn to the
\begin{proof}[of Theorem~\ref{theo:Norris}] The proof proceeds by
  induction on the parameter $m$. For $m = 0$, the statement is
  trivial since in this case one has $Z_A(t) = A_\emptyset(t)$, so
  that one can take $\Osc^0 = \emptyset$.

  Fix now a value $m \ge 1$ and assume that, for some $\eps$, both
  inequalities \minilab{e:assumptions}
  \begin{equs}
    \|Z_A\|_{L^\infty} &\le \eps\;, \label{e:boundZ}\\
    \sup_{|\alpha| \le m} \|A_\alpha\|_\Lip &\le
    \eps^{-3^{-(m+1)}} \label{e:boundLip}
  \end{equs}
  hold. Our aim is to find a (universal) family of negligible sets
  $\Osc^m$ such that, modulo $\Osc^m$, these two bounds imply the
  bound $\sup_{\alpha} \|A_\alpha\|_{L^\infty} \le
  \eps^{3^{-m}}$. Before we proceed, we localise our argument to
  Wiener processes that do not behave too ``wildly.''  Using the fact
  that the H\"older norm of a Wiener process has Gaussian tails for
  every H\"older exponent smaller than $1/2$, we see that the bounds
  \begin{equ}[e:boundW]
    \sup_{t \in [0,1]} \sup_{|\alpha| \le m} |W_\alpha(t)| \le
    \eps^{-1/10}\;,\quad \sup_{s \neq t} \sup_{|\alpha| \le m}
    {|W_\alpha(t) - W_\alpha(s)| \over |t-s|^{2/5}} \le
    \eps^{-1/30}\;,
  \end{equ}
  both hold modulo some universal family $\Wien$ of negligible
  events. The reason for these particular choices of exponents will
  become clearer later on, but any two negative exponents would have been
  admissible.

  Choose an exponent $\kappa$ to be determined later and define a sequence of times
  $t_\ell = \ell \eps^\kappa$ for $\ell = 0, \ldots, \eps^{-\kappa}$,
  so that the interval $[0,1]$ gets divided into $\eps^{-\kappa}$
  subintervals of the form $[t_\ell, t_{\ell+1}]$.  We define
  $A_\alpha^\ell = A_\alpha(t_\ell)$ and similarly for $W_\alpha^\ell$. We also
  define the Wiener increments $\bar W_i^\ell(t) = W_i(t) -
  W_i(t_\ell)$ and their products $\bar W_\alpha^\ell = \Pi_{j \in
    \alpha} \bar W_{j}^\ell$. With these notations, one has for
  $t \in [t_\ell, t_{\ell+1}]$ the equality
  \begin{equs}
    Z_A(t) &= Z_A(t_\ell) + \sum_{\alpha \neq \emptyset} A_{\alpha}^\ell
    \bigl(W_\alpha(t) - W_\alpha^\ell\bigr)
    + \sum_\alpha \bigl(A_\alpha(t) - A_\alpha^\ell\bigr) W_\alpha(t) \label{e:Taylor}\\
    & = Z_A(t_\ell)
    +  \sum_{\alpha\neq \emptyset}\sum_{\sigma \subset \alpha \atop \sigma \neq \emptyset} 
    A_{\alpha}^\ell W_{\alpha \setminus \sigma}^\ell \bar W_\sigma^\ell(t) + \sum_\alpha 
    \bigl(A_\alpha(t) - A_\alpha^\ell\bigr) W_\alpha(t)\\
    & = Z_A(t_\ell)
    +  \sum_{\nu}\sum_{\sigma \neq \emptyset} C_{\nu,\sigma} A_{\nu \cup \sigma}^\ell 
    W_ \nu ^\ell \bar W_\sigma^\ell(t) + \sum_\alpha \bigl(A_\alpha(t) - A_\alpha^\ell\bigr) W_\alpha(t)\\
    & \equiv Z_A(t_\ell) + \sum_{\nu}\sum_{j=1}^d C_{\nu,(j)} A_{\nu
      \cup (j)}^\ell W_\nu ^\ell \bar W_j^\ell(t) + E_\ell(t)\;,
  \end{equs}
  for some ``error term'' $E_\ell$ that will be analysed later.  Here, the combinatorial factor
  $C_{\alpha,\sigma}$ counts the number of ways in which the
  multi-index $\sigma$ can appear in the multi-index $\alpha \cup
  \sigma$ (for example $C_{(i,j),(j)}$ is equal to $2$ if $i \neq j$
  and $3$ if $i = j$).  Using the Brownian scaling and the fact that
  the supremum of a Wiener process has Gaussian tails, we see that for
  every $\kappa' < \kappa$, the bound
  \begin{equ}[e:boundW2]
    \sup_{\ell \le \eps^{-\kappa}} \sup_{t \in [0,\eps^\kappa]}
    \sup_{j\in\{1,\ldots,d\}} |\bar W_j^\ell(t)| \le
    \eps^{\kappa'/2}\;,
  \end{equ}
  holds modulo some universal family $\Wien_{\kappa',m}$ of negligible
  events.

  Note now that all the terms appearing in $E_\ell$ are (up to
  combinatorial factors) either of the form $A_{\alpha \cup
    \sigma}^\ell W_\alpha^\ell \bar W_\sigma(t)$ with $|\sigma| \ge
  2$, or of the form $\bigl(A_{\alpha}(t) - A_{\alpha}^\ell\bigr)
  W_\alpha(t)$.  Together with \eref{e:boundW2} and the first bound in \eref{e:boundW},
  this shows that there exists a constant $C$ depending only on $m$
  such that \eref{e:boundLip} implies
  \begin{equ}[e:boundE]
    \sup_{\ell \le \eps^{-\kappa}}  \sup_{t \in [t_\ell, t_{\ell+1}]} |E_\ell(t)| \le
    C\bigl(\eps^{\kappa' - 1/27 - 1/10} + \eps^{\kappa - 1/9 -
      1/10}\bigr)\;,
  \end{equ}
  modulo $\Wien_{\kappa',m}$. Here we used the fact that
  \eref{e:boundLip} implies in particular that the bound
  $\|A_\alpha\|_{L^\infty} \le \eps^{-1/27}$ holds for every $\alpha$
  with $|\alpha| \ge 2$ (note that these terms are non-zero only if $m
  \ge 2$) and that $\|A_\alpha\|_{\Lip} \le \eps^{-1/9}$, since we
  assumed $m \ge 1$. At this point, we fix $\kappa = {5\over 4}$ and
  $\kappa' = {6\over 5}$, so that in particular both exponents appearing in
  \eref{e:boundE} are greater than $1$. We then define $\Wien' = \Wien \cup
  \Wien_{\kappa',m}$ so that, modulo $\Wien'$, \eref{e:boundZ} and
  \eref{e:Taylor} imply
  \begin{equ}[e:goodbound]
    \sup_{t \in [t_\ell, t_{\ell+1}]} \Bigl|\sum_{\alpha}\sum_{j=1}^d
    C_{\alpha,(j)} A_{\alpha \cup (j)}^\ell W_\alpha^\ell \bar
    W_j(t)\Bigr| \le 2\eps + \sup_{\ell \le \eps^{-\kappa}}  \sup_{t \in [t_\ell,
      t_{\ell+1}]}|E_\ell(t)| \le C\eps \;.
  \end{equ}
  The left hand side of this expression motivates the introduction of operators $M_j$ acting
  on the set of families of stochastic processes by
  \begin{equ}
    \bigl(M_j A\bigr)_\alpha = C_{\alpha,(j)} A_{\alpha \cup (j)}\;.
  \end{equ}
  Note that $M_j$ lowers the ``degree'' of $A$ by one in the sense that if $A_\alpha=0$
  for every $|\alpha| \ge m$, then $\bigl(M_j A\bigr)_\alpha =0$   for every $|\alpha| \ge m-1$.

  With this notation, we can rewrite \eref{e:goodbound} as
  \begin{equ}[e:boundMjA]
    \sup_{t \in [t_\ell, t_{\ell+1}]} \Bigl|\sum_{j=1}^d
    Z_{M_jA}(t_\ell) \bar W_j(t)\Bigr| \le C\eps \;.
  \end{equ}
  Using the Brownian scaling and applying Lemma~\ref{lem:constant},
  combined with Lemma~\ref{lem:sumup}, shows the existence of a family
  $\Osc$ of negligible events such that \eref{e:boundMjA}
  implies
  \begin{equ}
    |Z_{M_jA}(t_\ell)| \le \eps^{7/20} \;,\quad \forall \ell \le
    \eps^{-5/4}\;.
  \end{equ}
  Here, we used the fact that our choice of $\kappa$ implies that $1 -
  \kappa/2 > 7/20$.  This shows that the statements
  \eref{e:assumptions} imply
  \begin{equs}
    \|Z_{M_jA}\|_{L^\infty} &\le \eps^{7/20} + C_m\sup_\alpha
    \bigl(\eps^\kappa \|A_\alpha\|_\Lip \|W_\alpha\|_{L^\infty}
    + \eps^{2 \kappa /5} \|A_\alpha\|_{L^\infty} \|W_\alpha\|_{C^{2/5}} \bigr) \\
    &\le \eps^{7/20} + C_m \bigl(\eps^{\kappa - 1/9 - 1/0} + \eps^{1/2
      - 1/9 - 1/30}\bigr) \le C_m \eps^{7/20}\;,
    \label{e:boundMjAZ}
  \end{equs}
  modulo $\Wien' \cup \Osc$. Here, the constant $C_m > 1$
  depends only on $m$.

  We now finally arrived at the stage where we are able to apply our
  induction hypothesis to each of the processes $Z_{M_j A}$.  Note
  that since $7/20 > 1/3$, \eref{e:boundLip} implies that
  \begin{equ}
    \sup_{\alpha, j}\|(M_jA)_\alpha\|_\Lip \le
    (C_m\eps^{7/20})^{-3^{-m}}\;,
  \end{equ}
  for all sufficiently small $\eps$.  Therefore, outside of the event
  $(\Osc^{m-1})_{C_m\eps^{7/20}}$, one has the implication
  \begin{equs}[e:implic]
    \Bigl\{\sup_j \|Z_{M_jA}\|_{L^\infty} &\le C_m\eps^{7/20}\Bigr\}
    \,\&\,
    \Bigl\{\sup_\alpha \|A_\alpha\|_\Lip \le \eps^{-3^{-(m+1)}}\Bigr\} \\
    &\Longrightarrow\, \Bigl\{\sup_{\alpha,j}\|(M_j
    A)_\alpha\|_{L^\infty} \le C'_m \eps^{{7 \over
        20}3^{-(m-1)}}\Bigr\}\;,
  \end{equs}
  for some different constant $C'_m$ depending also only on $m$. Since
  $7/20 > 1/3$ and since $\|(M_j A)_\alpha\|_{L^\infty} \ge
  \|A_{\alpha \cup (j)}\|_{L^\infty}$, this implies in particular that
  $\|A_\alpha\|_{L^\infty} \le \eps^{3^{-m}}$ for every $\alpha \neq
  \emptyset$.

  In order to conclude the proof of the theorem, it therefore only
  remains to obtain a similar bound on $\|A_\emptyset\|_{L^\infty}$.
  We define a family of negligible events $\Wien_m''$ so that $\Wien'
  \subset \Wien_m''$ and such that the bound
  \begin{equ}[e:smallbW]
    \sup_{t \in [0,1]} \sup_{|\alpha| \le m} |W_\alpha(t)| \le
    \eps^{-{1\over 70}3^{-(m-1)}}\;,
  \end{equ}
  holds modulo $\Wien_m''$. We claim that if we define recursively
  \begin{equ}
    (\Osc^m)_\eps = (\Osc^{m-1})_{C\eps^{7/20}} \cup
    (\Wien''_m)_\eps\;,
  \end{equ}
  the family $\Osc^m$ has the requested properties.  It follows indeed
  from \eref{e:boundZ}, \eref{e:smallbW} and the definition of $Z_A$
  that, modulo $\Osc^m$, \eref{e:assumptions} imply the bound
  \begin{equ}[e:boundphi]
    \|A_\emptyset\|_{L^\infty} \le \eps + \sum_{\alpha \neq \emptyset}
    \|A_\alpha\|_{L^\infty} \|W_\alpha\|_{L^\infty} \le \eps + C_m'
    \eps^{(7/20 - 1/70)3^{-(m-1)}}\;.
  \end{equ}
  Since we choose the bound \eref{e:smallbW} in such a way that $7/20
  - 1/70 > 1/3$, we obtain $\|A_\emptyset\|_{L^\infty} \le \eps^{1/3}$
  for sufficiently small $\eps$.  Together with the remark following
  \eref{e:implic}, this concludes the proof of
  Theorem~\ref{theo:Norris}.
\end{proof}

\section{Examples}
\label{sec:examples}

In this section, we apply the abstract framework developed in this article to two concrete
examples: the stochastic Navier-Stokes equations on a sphere and a class of stochastic reaction-diffusion equations.
The examples are chosen in order to highlight the techniques that can be used to verify
the assumptions of our results and to get some idea of their scope of applicability. 
In particular, the Navier-Stokes equations provide an example where bounds on the Jacobian are not
very uniform, so that an initial condition dependent  control is required in Assumption~\ref{ass:global}.
The stochastic reaction-diffusion system on the other hand satisfies very strong a priori bounds, but Assumption~\ref{ass:basic}
is not verified with the usual choice $\CH = L^2$, so that one has to work a bit more to fit the equations into the framework
presented here. Our strategy is as follows: in a first section, we provide a simplified version of our results. 
We tried to find a formulation that strikes a balance between powerful results and easily verifiable assumptions.
This general formulation will then be used by both of the examples mentioned above. 

\subsection{A general formulation}

The `general purpose' theorem formulated in this section allows to obtain the asymptotic
strong Feller property for  a large class of semilinear SPDEs under a H\"or\-man\-der-type bracket condition.
Our first assumption ensures that all the stability conditions of the previous sections can be verified.

\SetAssumptionCounter{D}
\begin{assumption}\label{ass:bounds}
The operator $L$ has compact resolvent.
Furthermore, there exists a measurable function $V \colon \CH \to \R_+$ such that there exist
constants $c>0$ and $\alpha > 0$ such that the bound
\begin{equ}
V(u) \ge c\|u\|^\alpha\;,
\end{equ}
holds for all $u \in \CH$ and such that the following bounds hold:

There exists a constant $C>0$ and $\eta' \in [0,1)$ such that
\begin{equ}[e:VLyap]
\E \exp(V(u_1)) \le C \exp\bigl(\eta' V(u_0)\bigr)\;.
\end{equ} 
We also require the following bounds on the Jacobian, as well as the second variation on the dynamic.
For every $p>0$ and every $\delta > 0$, there exists a constant $C$ such that the bounds
\minilab{e:boundsGeneral}
 \begin{equs}
    \sup_{t \in [0,1]} \E\|u_t\|^{p} &\le C\exp\bigl(\delta V(u_0)\bigr)\;,\label{e:boundSol}\\
   \E \sup_{s,t \in [0,1]} \|J_{s,t}\|^{p} &\le C\exp\bigl(\delta V(u_0)\bigr)\;, \label{e:boundJac}\\
   \sup_{s,t \in [0,1]} \E \|J^{(2)}_{s,t}\|^{p} &\le C\exp\bigl(\delta V(u_0)\bigr)\;, \label{e:boundJac2}
 \end{equs}
hold for every $u_0 \in \CH$.
\end{assumption}

Our next assumption is simply a restatement of the H\"ormander bracket condition 
(considering only constant `vector fields'), with the additional condition that the $g_i$ belong to $\CH_\infty$.
This ensures that all the relevant brackets are in $\CH_\infty$ and hence admissible in the sense of Section~\ref{sec:Hormander}.

\begin{assumption}\label{ass:brackets}
The forcing directions $g_i$ belong to $\CH_\infty$. Furthermore, define a sequence of subsets of $\CH$ recursively by
 $\AA_0 = \{g_j\,:\, j=1,\ldots,d\}$ and
\begin{equ}
 \AA_{k+1} \eqdef \AA_k \cup \{N_m(h_1,\ldots,h_m)\,:\, h_j \in \AA_k\}\;.
\end{equ}
Then, the linear span of $\AA_\infty \eqdef \bigcup_{n > 0} \AA_n$ is dense
in $\CH$.
\end{assumption}

With these assumptions in hand, a simplified, yet sufficiently powerful for many uses, formulation of our main results is as follows:

\begin{theorem}\label{theo:general}
Consider the setting of equation \eref{e:SPDEintro} and assume that Assumptions~\ref{ass:basic},
\ref{ass:bounds}, and \ref{ass:brackets} hold. Then, there exist constants $C,\kappa >0$ and $\gamma \in (0,1)$ such that the Markov semigroup 
$\CP_t$ generated by \eref{e:SPDEintro} satisfies the bound
\begin{equ}[e:boundGeneral]
   \|\DF(\CP_{2n}\phi)(u)\| \leq  C e^{\kappa V(u_0)}\Big(
   \sqrt{(\CP_{2n} \phi^2)(u)} + \gamma^{n}\sqrt{(\CP_{2n}
     \|\DF\phi\|^2)(u)}\Big)\;,
\end{equ}
for every integer $n>0$.
In particular, it satisfies the asymptotic strong Feller property. 

Furthermore, if $\beta_\star > a-1$, then for every $m>0$, every $u \in \CH$, and every linear map
$\CT\colon \CH \to \R^m$, the projections of the time-$2$ transition probabilities $\CT^*\CP_2(u,\cdot)$ have $\CC^\infty$
densities with respect to Lebesgue measure on $\R^m$.
\end{theorem}

\begin{remark}
The final times $1$ and $2$ appearing in the statement are somewhat arbitrary since it suffices to
rescale the equation in time, which does not change any of our assumptions. We chose to keep them in this
way in order to avoid awkward notations in the proof.
\end{remark}

In this result, the H\"ormander-type assumption, Assumption~\ref{ass:Hormander}
is verified by using constant vector fields only. Before we turn to the proof of Theorem~\ref{theo:general}, we therefore
present the following useful little lemma:

\begin{lemma}\label{lem:dense}
Let $\CH$ be a separable Hilbert space and $\{g_i\}_{i=1}^\infty \subset \CH$ a collection of elements
such that its span is dense in $\CH$. Define a family of symmetric bilinear forms $\CQ_n$ on $\CH$ by
$\scal{h, \CQ_n h} = \sum_{i=1}^n \scal{g_i,h}^2$.
Let $\Pi\colon \CH \to \CH$ be any orthogonal projection on a finite-dimensional
subspace of $\CH$. Then, there exists $N>0$ and, for every $\alpha>0$ there exists  $c_\alpha > 0$ such that 
$\scal{h, \CQ_n h} \ge c_\alpha \|\Pi h\|^2$ for every $h \in \CH$ with $\|\Pi h\| \ge \alpha \|h\|$ and every $n\ge N$.
\end{lemma}

\begin{proof}
Assume by contradiction that the statement does not hold. Then, 
there exists $\alpha > 0$ and a sequence $h_n$ in $\CH$ such that $\|\Pi h_n\| = 1$, $\|h_n\| \le \alpha^{-1}$, and
such that $\lim_{n \to 0} \scal{h_n, \CQ_n h_n} \to 0$. Since $\|h_n\| \le \alpha^{-1}$ is bounded, we can assume
(modulo extracting a subsequence) that there exists $h \in \CH$ such that $h_n \to h$ in the weak
topology.
Since $\Pi$ has finite rank, one has $\|\Pi h\| = 1$. Furthermore, since the maps $h \mapsto \scal{h, Q_n h}$
are continuous in the weak topology and since $n \mapsto \scal{h, Q_n h}$ is increasing for every $n$, one has
\begin{equ}
\scal{h,Q_n h} = \lim_{m \to \infty} \scal{h_m, Q_n h_m} \le \lim_{m \to \infty} \scal{h_m, Q_m h_m} = 0\;,
\end{equ}
so that $\scal{h, g_i} = 0$ for every $i > 0$. This contradicts the fact that the span of the $g_i$ is dense in $\CH$.
\end{proof}

We are now in a position to turn to the proof of our general result.

\begin{proof}[of Theorem~\ref{theo:general}]
We show first that the supremum in \eref{e:boundSol} can easily be pulled under the expectation.
Indeed, it follows from the variation of constants formula that we have the bound
\begin{equ}
\|u_t\| \le \|S(t)u_0\| +C \int_0^s (t-s)^{-a}\|N(u_s)\|_{-a}\,ds + \|W_L(t)\|\;,
\end{equ}
where $W_L$ is the stochastic convolution of $G\,W$ with the semigroup $S$ generated by $L$. It follows immediately from H\"older's inequality that 
there exists a constant $C$ and an exponent $p>0$ such that
\begin{equ}
\sup_{t \le 1}\|u_t\| \le \|u_0\| + C\Bigl(\int_0^1\bigl(1+ \|u_s\|^{np}\bigr)\,ds\Bigr)^{1/p} + \sup_{t \le 1}\|W_L(t)\|\;.
\end{equ}
Combining this with \eref{e:boundSol}, we conclude immediately that for every $p>0$ and every $\delta > 0$ there exists
$C>0$ such that 
\begin{equ}[e:boundSolSup]
     \E\sup_{t \in [0,1]} \|u_t\|^{p} \le C\exp\bigl(\delta V(u_0)\bigr)\;.
\end{equ}

We now verify that Assumptions~\ref{ass:global} and \ref{ass:Hormander} are satisfied for our problem.
It follows from \eref{e:boundSolSup} and \eref{e:boundJac} that for every $\delta > 0$, 
Assumption~\ref{ass:global} holds with the choice $\Psi_0(u) = \exp\bigl(\delta V(u)\bigr)$.

Furthermore, Assumption~\ref{ass:Hormander} holds for every finite-rank orthogonal projection $\Pi\colon \CH \to \CH$ 
by Assumption~\ref{ass:brackets} and
Lemma~\ref{lem:dense}. Note that the function $\Lambda_\alpha$ is then constant, so that the condition on its moments is trivially satisfied.
We can therefore apply Theorem~\ref{theo:Malliavin} 
which states that for every $\alpha \in (0,1)$, every $\delta > 0$, every finite-rank projection $\Pi$, and every $p \ge 1$
  there exists a  constant $C$ such that the bound
  \begin{equ}[e:inverseMalliavin]
    \P \Bigl( \inf_{\phi \in \CS_\alpha} {\scal{\phi, \MM_1 \phi}\over \|\phi\|^2} \le
    \eps \Bigr) \le C \exp\bigl(\delta V(u_0)\bigr) \eps^{p}\;,
  \end{equ}
holds for every $u_0 \in \CH$ and every $\eps \le 1$. 

Combining this statement with \eref{e:VLyap}, we see that Assumption~\ref{ass:Malliavin} is satisfied with $\bar q = 8$ (for example) and
$U(u) = \exp\bigl(\delta V(u)\bigr)$ with every $\delta \le {1\over 8}$.

The bound \eref{e:VLyap} is nothing but a restatement of
Assumption~\ref{ass:Lyapunov}. 
Since we assume that \eref{e:boundJac} and \eref{e:boundJac2} hold for every $\delta>0$, 
we infer that Assumption~\ref{ass:Jacobian} holds with $\bar p = 20$ and $\eta$ sufficiently small.
It remains to verify that, for every $C_\Pi>0$ there exists a finite-rank projection $\Pi$ such that
\eref{e:smoothingJ} is satisfied. This ensures that the required relation $C_\Pi > C_J + 2\eta C_L/(1-\eta')$
can be satisfied by a suitable choice of $\Pi$.

Because $L$ has compact resolvent by assumption,  it has a complete system of eigenvectors with the corresponding eigenvalues
$\{\lambda_n\}$ satisfying $\lim_{n \to \infty} \lambda_n = \infty$. Therefore, if we denote by $\Pi_N$ the projection
onto the subspace of $\CH$ spanned by the first $N$ eigenfunctions, we have the identity
\begin{equ}
\|e^{-Lt} \Pi_N^\perp\| = e^{-\lambda_{N+1}t}\;.
\end{equ}
This allows us to get a bound on $J_{0,1} \Pi^\perp$ as follows. It follows from \eref{e:Jacobian} and the variation
of constants formula that
\begin{equs}
\|J_{0,t} \Pi^\perp\| &\le \|e^{-Lt} \Pi^\perp\| + \int_0^t C s^{-a}\|DN(u_s)\|_{-a}\,ds \\
&\le e^{-\lambda_{N+1}t} + Ct^{1-a} \sup_{s\le t} \|u_s\|^{k}\;,
\end{equs}
so that, for every $\delta > 0$, we have by \eref{e:boundSolSup} the bound
\begin{equ}
\E \|J_{0,t} \Pi^\perp\|^p \le C_{\delta,p} \bigl(e^{-\lambda_{N+1}t} + t^{1-a}\bigr) \exp \bigl(\delta V(u_0)\bigr)\;,
\end{equ}
for some family of constants $C_{\delta,p}$ independent of $t \in [0,1]$.
Since $a < 1$, it follows that for every $\eps, \delta > 0$ and $p > 0$, we can find $N$ sufficently large and $t$ sufficiently small such that
\begin{equ}
\E \|J_{0,t} \Pi^\perp\|^p \le \eps \exp \bigl(\delta V(u_0)\bigr)\bigr)\;.
\end{equ}
Combining this with \eref{e:boundJac} and the fact that $\|J_{0,1}\Pi^\perp\| \le \|J_{t,1}\| \|J_{0,t}\Pi^\perp\|$,
we obtain
\begin{equ}
\E \|J_{0,1} \Pi^\perp\|^{\bar p} \le \Bigl(\E \|J_{0,t}\|^{2\bar p} \E \|J_{t,1}\|^{2\bar p}\Bigr)^{1\over 2} \le  C \eps \exp \bigl(2\delta V(u_0)\bigr)\bigr)\;,
\end{equ}
provided that $N$ is sufficiently large. By choosing $\delta$ sufficiently small, it follows that 
Assumption~\ref{ass:smoothing} (with arbitrary values for $\bar p$ and $C_\Pi$)
can always be satisfied by choosing for $\Pi$ the projection onto the first $N$ eigenvectors of $L$ for
some large enough value of $N$. The bound \eref{e:boundGeneral} now follows from a simple application of Theorem~\ref{thm:asfEstimate}.

It remains to prove the statement about the smoothness of $\CT^*\CP_2(u,\cdot)$, which will be a consequence of
\eref{e:inverseMalliavin} by \cite[Cor.~2.1.2]{Nualart}. The reason why we consider the process at time $2$ is that,
in order to avoid the singularity at the origin, we consider the solution $u_2$ as an element of the probability space
with Gaussian structure given by the increments of $W$ over the interval $[1,2]$. The increments of $W$ over $[0,1]$
are then considered as some ``redundant'' randomness, which is irrelevant by \cite[Ch.~1]{Nualart}.
With this slightly tweaked Gaussian structure, the Malliavin matrix of $\Pi u_2$ is given almost surely
by $\Pi \CM_1(u_1) \Pi$, where $\CM_1$ is defined as before, but over the interval $[1,2]$.
The claim now follows from \eref{e:inverseMalliavin} and \eref{e:VLyap}, provided that the random variable $\Pi u_2$ 
belongs to the space $\CD^{\infty}$ of random variables whose Malliavin derivatives of all orders have moments of all orders.

Recall now (see for example \cite[Section~5.1]{Yuri}) that for any $n$-tuple of elements $h_1,\ldots,h_n \in L^2([1,2], \R^d)$, the $n$th Malliavin derivative of $u_2$ in the
directions $h_1,\ldots,h_n$ is given by
\begin{equ}[e:mallu2]
\CD^n u_2(h) = \int_{1 \le s_1 <\cdots < s_n \le 2} J_{s,1}^{(n)} Gh_s \,ds\;.
\end{equ}
Applying \eref{e:apsol} in Proposition~\ref{prop:apriori} we see that, for every $u_0 \in \CH$, every $\gamma < \gamma_\star+1$, 
and every $p>0$, one has the bound
\begin{equ}
\E \sup_{t \in [1,2]} \|u_t\|_\gamma^p < \infty\;.
\end{equ}
We conclude from Proposition~\ref{prop:boundvarn} that 
\begin{equ}
\E \sup_{1 \le s_1 <\cdots < s_n \le 2} \sup_{\|\phi_j\| \le 1} \|J_{s,t}^{(k)}(\phi_1,\ldots,\phi_k)\|^p \le \infty\;,
\end{equ}
so that, by \eref{e:mallu2}, $u_2$ does indeed have Malliavin derivatives of all orders with bounded
moments of all orders. This concludes the proof.
\end{proof}

\subsection{The 2D Navier-Stokes equations on a sphere}
\SetAssumptionCounter{NS}

Consider the stochastically forced two-dimensional Navier-Stokes equations on the two-dimensional 
sphere $S^2$:
\begin{equ}[e:SNSsphere]
du = \nu \Delta u \,dt + \nu \Ric u\, dt - \nabla_{\! u} u\,dt - \nabla p\,dt + Q\,dW(t)\;,\qquad \div u = 0\;.
\end{equ}
Here, the velocity field $u$ is an element of $H^1(S^2, TS^2)$, $\nabla_{\! u} u$ denotes the
covariant differentiation of $u$ along itself with respect to the Levi-Civita connection on $S^2$, 
$\Delta = -\nabla^* \nabla$ is the (negative of the) Bochner Laplacian on $S^2$, and $\Ric$ denotes the Ricci operator from $TS^2$ into
itself. In the case of the sphere, the latter is just the multiplication with the scalar $1$.
See also \cite{Taylor92,TemWang,Naga97} for more details on the Navier-Stokes equations on manifolds.

As in the flat case, it is possible to represent $u$ uniquely by a scalar ``vorticity'' field $w$ given
by 
\begin{equ}[e:defw]
w = \curl u \eqdef - \div \bigl(n \wedge u\bigr)\;,
\end{equ}
where $n$ denotes the unit vector in $\R^3$ normal to the
surface of the sphere (so that $n \wedge u$ defines again a vector field on the sphere). 
With this notation, one can rewrite \eref{e:SNSsphere} as
\begin{equ}[e:SNSvort]
dw = \nu \Delta w \,dt - \div (w\,Kw)\,dt + G\,dW(t)\;.
\end{equ}
Here, we denoted by $K$ the operator that reconstructs a velocity field from its vorticity field, that is
\begin{equ}
u = Kw = - \curl \Delta^{-1} w \eqdef n \wedge \nabla \Delta^{-1} w\;, 
\end{equ}
and $\Delta$ denotes the Laplace-Beltrami operator on the
sphere. See \cite{TemWang} for a more detailed derivation of these equations. 
In order to fit the framework developed in this article, we assume that the operator $G$ is of finite rank
and that its image consists of smooth functions, so that the noise term can be written as
\begin{equ}
G\,dW(t) = \sum_{i=1}^n g_i\,dW_i(t)\;,\qquad g_i \in H^\infty(S^2,\R)\;.
\end{equ}
We choose to work in the space $\CH = L^2(S^2,\R)$ for the equation \eref{e:SNSvort} in vorticity formulation, 
so that the interpolation spaces $\CH_\alpha$
coincide with the fractional Sobolev spaces $H^{2\alpha}(S^2,\R)$, see \cite{Triebel}.
In particular, elements $w \in \CH_\alpha$ are characterised by the fact that the functions
$x \mapsto \phi(x) w(\psi(x))$
belong to $H^{2\alpha}(\R^2)$ for any compactly supported smooth function $\phi$ and
any function $\psi \colon \R^2 \to S^2$ which is smooth on an open set containing the support of $\phi$.
Since the sphere is compact, this implies that the usual Sobolev embeddings for the torus also
hold true in this case.

Define now $\AA_0 = \{g_i\,:\, i=1,\ldots,n\}$ and set recursively
\begin{equ}
\AA_{n+1} = \AA_n \cup \{B(v,w)\,:\; v,w \in \AA_n\}\;,
\end{equ}
where we made use of the symmetrised nonlinearity
\begin{equ}
B(v,w) = \hf \bigl(\div (w\,Kv) + \div (v\,Kw)\bigr)\;.
\end{equ}

We then have the following result:

\begin{theorem}
If the closure of the linear span of $\AA_\infty = \bigcup_{n \ge 0} \AA_n$ is equal to all of $L^2(S^2,\R)$, then
the equations \eref{e:SNSsphere} have a unique invariant measure.
\end{theorem}

\begin{remark}
  Sufficient conditions for density of $\AA_\infty$ and for approximate
  controllability are given in
  \cite{agrachev_sarychev_2007}. In particular, the authors there give an example of $\AA_0$
  containing five spherical harmonics that satisfies our
  condition. Note however that controllability is not required for our
  result to hold, since we only use the fact that the origin belongs
  to the topological support of every invariant measure. On the other
  hand, as shown in \cite{MatPar06:1742}, controllability allows to obtain
  positivity of the projected densities of transition probabilities.
\end{remark}

\begin{proof}
  The main step in the proof is to check that we can apply
  Theorem~\ref{thm:asfEstimate} to conclude that the Markov semigroup
  generated by the solutions to \eref{e:SNSsphere} has the asymptotic
  strong Feller property.  Let us first check that the Navier-Stokes
  nonlinearity on the sphere does indeed satisfy
  Assumption~\ref{ass:basic} for some $a \in [0,1)$. It is clear that
  the nonlinearity $N$, defined by $N(w)=B(w,w)$, is continuous from
  $\CH_\infty$ to $\CH_\infty$ (which coincides with the space of
  infinitely differentiable functions on the sphere), so in order to
  show point 2, it remains to show that $N$ maps $\CH_{\gamma}$ into
  $\CH_{\gamma - a}$ for a range of values $\gamma \ge 0$ and some $a
  \in [0,1)$.

Setting $\hat B(w,w') = \div (w\,Kw')$ so that $N(w) = B(w,w) = \hat B(w,w)$, one can show exactly as in \cite{ConstFoi} that, 
for any triplet $(s_1,s_2,s_3)$ with $s_i \ge 0$, $\sum_{i} s_i >1$, one has bounds of the type
\begin{equs}
\int_{S^2} v(x)\,\hat B(w,w')(x)\,dx &\le C\|v\|_{H^{s_1}} \|w\|_{H^{1+s_2}} \|w'\|_{H^{s_3-1}}\;, \\
\int_{S^2} v(x)\,\hat B(w,w')(x)\,dx &\le C\|v\|_{H^{1+s_1}} \|w\|_{H^{s_2}} \|w'\|_{H^{s_3-1}}\;,
\end{equs}
for some constant $C$ depending on the choice of the $s_i$. In particular, $\hat B$ can be interpreted as a continuous linear map
from $\CH \otimes \CH$ into $\CH_{-{3 \over 4}}$ (for example) and from $\CH_{1\over 2} \otimes \CH_{1\over 2}$ into
$\CH$ (using the usual identification of bilinear maps with linear maps between tensor products). 
It  thus follows from the Calder\'on-Lions interpolation theorem as in Remark~\ref{rem:interpolation} that 
$\hat B$ is a continuous linear map from $\CH_\alpha \otimes \CH_\alpha$ into $\CH_\beta$ for
$\beta = {3\alpha \over 2} - {3\over 4}$ and $\alpha \in [0,\hf]$. For $\alpha > \hf$, we use the fact that
$\CH_\alpha$ is an algebra \cite{Trie92} to deduce that $\hat B$ is continuous from $\CH_\alpha \otimes \CH_\alpha$
into $\CH_{\alpha-{1\over 2}}$. This shows that point 2 of Assumption~\ref{ass:basic} is satisfied with $a = {3\over 4}$
(any exponent strictly larger than $\hf$ would do, actually) and $\gamma_\star = +\infty$.

Turning to point 3 of Assumption~\ref{ass:basic}, it suffices to show that, for $v$ sufficiently
smooth, the map $w \mapsto \hat B(v,w)$ is bounded from $\CH_{-\beta}$ into $\CH_{-\beta-a}$.
It is well-known on the other hand that if $v \in \CC^k$ then the multiplication operator $w \mapsto vw$
is continuous in $H^s$ for all $|s| \le k$. It follows immediately that 
$DN^*(v)$ is continuous from $\CH_{-\beta}$ into $\CH_{-\beta-{1\over 2}}$, provided that 
$v \in \CC^k$ for $k \ge 2\beta$. Point 3 then follows with $\beta_\star = \infty$.

For any fixed $\eta > 0$, it follows exactly as in \cite[Lemma~4.10]{HairerMattingly06AOM} that
Assumption~\ref{ass:bounds} is verified with $V(w) = \eta \|w\|^2$ for $\eta$ sufficiently small. 
This concludes the verification
of the assumptions of Theorem~\ref{theo:general} and the claim follows.
\end{proof}

\begin{remark}
Just as in \cite{HairerMattingly06AOM}, this result is optimal in the following sense. The closure
$\bar \AA_\infty$ of the linear span of $\AA_\infty$ in $L^2$ is always an invariant subspace for \eref{e:SNSvort} and
the invariant measure for the Markov process restricted to $\bar \AA_\infty$ is unique.
However, if $\bar \AA_\infty \neq L^2$, then one expects in general the presence of more than one invariant
probability measure in $L^2$ at low values of the viscosity $\nu$.
\end{remark}

\subsection{Stochastic reaction-diffusion equations}
\SetAssumptionCounter{RD}

In this section, we consider a general class of reaction-diffusion equations on a ``nice'' domain $D$. The 
dimension $m$ of the ambient space is chosen smaller or equal to $3$ for technical reasons. However,
 the number $\ell$ of components in the reaction is arbitrary. The domain
$D$ is assumed to be either of
\begin{claim}
\item A compact smooth $m$-dimensional Riemannian manifold.
\item A bounded open domain of $\R^m$ with smooth boundary.
\item A hypercube in $\R^m$.
\end{claim}
We furthermore denote by $\Delta$ the Laplace (resp. Laplace-Beltrami) operator on $D$, endowed with either
Neumann or Dirichlet boundary conditions. With these notations in place, the equations that we consider are
\begin{equ}[e:RD]
du = \Delta u\, dt + f\circ u\, dt + \sum_{i=1}^d g_i\, dW_i(t)\;,
\end{equ}
with $u(t)\colon D \to \R^\ell$ and $f \colon \R^\ell \to \R^\ell$ a polynomial of arbitrary degree $n$ with $n\ge 3$ an odd integer. 
(We exclude the case $n = 1$ since this gives rise to a linear equation and is trivial to analyse.)
The functions $g_i$ describing the
stochastic component of the equations are assumed to belong to $\CH_\infty$, the intersection of the domains of 
$\Delta^\alpha$ in $L^2(D)$ for all $\alpha > 0$. It is a straightforward exercise to check that \eref{e:RD}
has unique local solutions in $\CE = \CC(D,\R^\ell)$ for every initial condition in $\CC(D,\R^\ell)$ (replace $\CC$ by
$\CC_0$ in the case of Dirichlet boundary conditions). In order to obtain global solutions, we make the following 
assumption on the nonlinearity:

\begin{assumption}\label{ass:Lyap}
Writing $f = \sum_{k=0}^{n} f_k$ for $f$ with $f_k$ being $k$-linear maps from $\R^\ell$ to itself, we assume that
$n$ is odd and that
\begin{equ}
\scal{f_n(u,\ldots,u,v),v} < 0\;,
\end{equ}
for every $u,v \in \R^\ell \setminus \{0\}$.
\end{assumption}

\begin{remark}
Provided that Assumption~\ref{ass:Lyap} holds, one 
can check that there exist positive constants $c$ and $C$ such that the inequality
\begin{equ}[e:boundf]
\scal{f(u+v),u} \le C(1 + \|v\|^{n+1}) - c \|u\|^{n+1}\;,
\end{equ}
holds for every $u,v \in \R^\ell$.
\end{remark}

Essentially, Assumption~\ref{ass:Lyap} makes sure that the function $u
\mapsto |u|^2$ is a Lyapunov function for the ``reaction'' part $\dot
u = f(u)$ of \eref{e:RD}. In the interest of brevity, we define
$\Sup_{t,\infty}(v)=1+\sup_{s \le t} \|v(s)\|_{\CE}$ for any
function $v \in L^\infty ([0,t],\CE)$ and
$\Sup_{t,r}(v)=1+\sup_{s \le t} \|v(s)\|_{H^r}$ for $v \in L^\infty
([0,t],H^r(D))$, As a consequence of Assumption~\ref{ass:Lyap}, we
obtain the following \textit{a priori} bound on the solutions to
\eref{e:RD}:

\begin{proposition}\label{prop:boundRD}
Under Assumption~\ref{ass:Lyap}, there exist constants $c$ and $C$ such that
the bound
\begin{equ}
  \|u(t)\|_{L^\infty} \le C \Bigl(\frac{\|u_0\|_{L^\infty}}{(1 +
    t\|u_0\|_{L^\infty}^{n})^{1/n}} +
  \Sup_{t,\infty}(W_\Delta)\Bigr)\;,
\end{equ}
holds almost surely for every $u_0 \in \CE$, where $\CE$ is either
$\CC(D,\R^\ell)$ or $\CC_0(D,\R^\ell)$, depending on the boundary
conditions of $\Delta$. In particular, for every $t_0 > 0$ there
exists a constant $C$ such that one has the almost sure bound
\begin{equ}[e:bounddV]
\|u(t)\|_{L^\infty} \le C \Sup_{t,\infty}(W_\Delta)\;,
\end{equ}
independently of the initial condition, provided that $t \ge t_0$.
\end{proposition}

\begin{proof}
  The proof is straightforward and detailed calculations for a variant
  of it can be found for example in \cite{SPDEnotes}. Setting $v = u -
  W_\Delta(t)$ where $W_\Delta$ is the ``stochastic convolution''
  solving the linearised equation \eref{e:RD} with $f\equiv 0$, and
  defining $V(v) = \|v\|_{L^\infty}^2$, we obtain from \eref{e:boundf}
  the almost sure bound
\begin{equ}
{d \over dt} V(v(t)) \le C\Sup_{t,\infty}^{n+1}(W_\Delta) - c V^{(n+1)/2}(v(t))\;.
\end{equ}
In particular, there exist possibly different constants such that
\begin{equ}
{d \over dt} V(v(t)) \le  - c V^{(n+1)/2}(v(t))
\end{equ}
for all $v$ such that $V(v(t)) \ge
C\Sup_{t,\infty}^2(W_\Delta))$. Since we assumed that $n \ge 3$, a
simple comparison theorem for ODEs then implies that
\begin{equ}
  V(v(t)) \le C {\|u_0\|^2 \over (1 + t\|u_0\|^{2/\alpha})^\alpha}
  \wedge\Sup_{t,\infty}^2(W_\Delta)\;,
\end{equ}
where we set $\alpha = 2/n$. The requested bound then follows at once. The second bound is an immediate consequence 
of the first one.
\end{proof}

\begin{remark}
  The function $t \mapsto V(v(t))$ is of course not differentiable in
  time in general. The left hand side in \eref{e:bounddV} should
  therefore be interpreted as the right upper Dini derivative
  $\limsup_{h\to 0^+} h^{-1} \bigl(V(v(t+h)) - V(v(t))\bigr)$.
\end{remark}

In order to fit the framework developed in this article, we cannot
take $L^2$ as our base space, since the nonlinearity will not in
general map $L^2$ into any Sobolev space of negative order. However,
provided that $k > m/2$, the Sobolev spaces $H^k$ form an algebra, so
that the nonlinearity $u \mapsto N(u) \eqdef f\circ u$ is continuous
from $H^k$ to $H^k$ in this case. It is therefore natural to choose
$\CH = H^k$ for some $k > m/2$.  In this case, for $\alpha > 0$, the
interpolation spaces $\CH_\alpha$ coincide with the Sobolev spaces
$H^{k + 2\alpha}$, so that one has $N \in
\Poly(\CH_\alpha,\CH_\alpha)$ for every $\alpha > 0$. This shows that
Assumption~\ref{ass:basic} is satisfied with $a = 0$, $\gamma_\star =
\infty$ and $\beta_\star = \infty$. It turns out that it is relatively
easy to obtain bounds in the Sobolev space $H^2$. From now on, we do
therefore assume that the following holds:

\begin{assumption}\label{ass:dim3}
The space dimension $m$ is smaller or equal to $3$.
\end{assumption}

This will allow us to work in $\CH = H^2$. Before we state the main theorem of this section,
we obtain a number of \textit{a priori} bounds that will allow us to
verify that the assumptions from the previous parts of this article do
indeed apply to the problem at hand.

By using a bootstrapping argument similar to
Proposition~\ref{prop:bootstrapping}, we can obtain the following
\textit{a priori} estimate:
\begin{proposition}\label{prop:aprioriRD} Assume that Assumptions~\ref{ass:Lyap} and \ref{ass:dim3} hold. If $u$ is the solution to
  \eref{e:RD} with initial condition $u_0 \in H^2$ then there exists a
  constant $C$ such that the bounds
\begin{equs}
  \|u(t)\|_{H^2} &\le C \Sup_{t,\infty}^{2n}(u) 
  \bigl(\|u_0\|_{H^2} +\Sup_{t,2}(W_\Delta)\bigr)\;,\\
  \|u(t)\|_{H^2} &\leq C \Sup_{t,\infty}^{2n}(u) 
  \bigl( \frac1t\|u_0\|_{L^2} +\Sup_{t,2}(W_\Delta)\bigr)\;,
\end{equs}
hold for all $t \leq 1$ almost surely.
\end{proposition}

\begin{proof}
From Duhamel's formula, we obtain the bound
\begin{equs}
\|u(t)\|_{H^1} &\le \|u_0\|_{H^1} + \int_0^t {C\over \sqrt{t-s}} \|f \circ u(s)\|_{L^2}\,ds + \Sup_{t,1}(W_\Delta) \\
&\le \|u_0\|_{H^2} +  C \sqrt t \Sup_{t,\infty}^n(u)+ \Sup_{t,1}(W_\Delta)\;.
\end{equs}
At this stage, we use that since $f$ is a polynomial of degree $n$, there exists a constant $C$ such that
\begin{equ}[e:algebra]
\|f\circ u\|_{H^1} \le C \bigl(1+\|u\|_{L^\infty}^n + \|u\|_{L^\infty}^{n-1} \|u\|_{H^1}\bigr)\;.
\end{equ}
Using Duhamel's formula again, this yields
\begin{equs}
  \|u(t)\|_{H^2} &\le \|u_0\|_{H^2} + \int_0^t {C\over \sqrt{t-s}} \|f \circ u(s)\|_{H^1}\,ds+\Sup_{t,2}(W_\Delta) \\
  &\le \|u_0\|_{H^2} +   \int_0^t {C\over \sqrt{t-s}} \bigl(1+\|u(s)\|_{L^\infty}^n + \|u(s)\|_{L^\infty}^{n-1} \|u(s)\|_{H^1}\bigr)\,ds \\
  &\qquad +\Sup_{t,2}(W_\Delta)\\
  &\le \|u_0\|_{H^2} + \int_0^t {C\over \sqrt{t-s}} \Bigl(\Sup_{t,\infty}^n(u) +
  \Sup_{t,\infty}^{n-1}(u) \bigl(\|u_0\|_{H^2} \\
  &\qquad \qquad \qquad \qquad \quad + \sqrt s \Sup_{t,\infty}^n(u)+\Sup_{t,1}(W_\Delta)\bigr)\Bigr)\,ds +\Sup_{t,2}(W_\Delta) \;.
\end{equs}
Integrating the last term yields the first bound. The second bound can be obtained in exactly the same way, using the
smoothing properties of the semigroup generated by the Laplacian.
\end{proof}

As a consequence, we obtain the following bound on the exponential moments in $H^2$
of the solution starting from any initial condition:

\begin{proposition}\label{prop:RDUniformV}
For every $T>0$, there exists a  constant $C > 0$ such that
\begin{equ}
\E \exp\bigl(\|u(T)\|_{H^2}^{1/n}\bigr) \le C\;,
\end{equ}
for every initial condition $u_0 \in H^2$.
\end{proposition}

\begin{proof}
Without loss of generality, we set $T=1$.
Combining Proposition~\ref{prop:aprioriRD} and the Markov property, we see that there exists a constant $C>0$ such that
\begin{equs}
\|u(1)\|_{H^2}&\le C\Big(\sup_{\hf \le s \le 1} \|u(s)\|_{L^\infty}^n\Big) \Big(\|u(\hf)\|_{L^2} + \Sup_{1,2}(W_\Delta)\Big)\\
&\le C\Big(\sup_{\hf \le s \le 1} \|u(s)\|_{L^\infty}^n\Big)\Sup_{1,2}(W_\Delta)\;.
\end{equs}
The requested bound then follows from \eqref{e:bounddV} and the fact
that $\Sup_{1,2}(W_\Delta)$ has Gaussian tails by Fernique's theorem.
\end{proof}

We now turn to bounds on the Jacobian $J$ for \eref{e:RD}. Recall from
\eref{e:Jacobian} that, given any ``tangent vector'' $\xi$, the
Jacobian $J_{s,t}\xi$ satisfies the random PDE
\begin{equ} {d\over dt} J_{s,t}\xi = \Delta J_{s,t}\xi + (Df \circ
  u)(t) J_{s,t}\xi\;,
\end{equ}
where $Df$ denotes the derivative of the map $f$.  Our main tool is
the fact that, from Assumption~\ref{ass:Lyap}, we obtain the existence
of a constant $C>0$ such that
\begin{equ}
\scal{Df(u)v,v} \le C|v|^2\;,
\end{equ}
for every $u,v \in \R^\ell$.
In particular, we obtain the \textit{a priori} $L^2$ estimate:
\begin{equ}[e:jacdetRD] {d\over dt} \|J_{s,t}\xi\|_{L^2}^2 = -
  2\|\nabla J_{s,t}\xi\|_{L^2}^2 + 2\scal{J_{s,t}\xi, (Df \circ u)(t)
    J_{s,t}\xi} \le 2C\|J_{s,t}\xi\|_{L^2}^2\;,
\end{equ}
so that $\|J_{s,t}\xi\|_{L^2} \le e^{C(t-s)}\|\xi\|_{L^2}$ almost
surely. We now us similar reasoning to obtain a sequence of similar
estimates in smoother spaces.
\begin{proposition}\label{prop:RDJ}
For any $u_0 \in H^2$, the Jacobian satisfies the operator bounds 
  \begin{align*}
    \|J_{s,t}\|_{L^2 \rightarrow L^2} &\le C\\
    \|J_{s,t}\|_{L^2\rightarrow H^1}  &\leq  C \Bigl({1\over \sqrt{t-s}} + \Sup_{t,\infty}^{n}(u)\Bigr)\,,\\
    \|J_{s,t}\|_{H^1\rightarrow H^1}  &\leq  C\Sup_{t,\infty}^{n}(u)\,,\\
    \|J_{s,t}\|_{H^2 \rightarrow H^2} & \leq  C \Sup_{t,\infty}^{4n}(u) 
    \bigl(\|u_0\|_{H^2} +\Sup_{t,2}(W_\Delta)\bigr)\\
    \|J_{s,t}\|_{H^1 \rightarrow H^2} & \leq  C \Sup_{t,\infty}^{4n}(u) 
    \bigl(\|u_0\|_{H^2} +\Sup_{t,2}(W_\Delta)\bigr) +\frac{C}{\sqrt{t-s}}
  \end{align*}
  for $0\leq s < t \leq 1$ with $\Sup_{t,\infty}$ defined just before
  Proposition \ref{prop:aprioriRD}.
\end{proposition}

\begin{proof} The first estimate is just a rewriting of the
  calculation before the Proposition.  As in the proof of the \textit{
    a priori} bounds for the solution, we are going to use a
  bootstrapping argument, starting from the bound
  \eref{e:jacdetRD}. Applying Duhamel's formula and using the notation
  $\Sup_{t,\infty}$ as before, we obtain
\begin{equs}
  \|J_{s,t}\xi\|_{H^1} &\le \|\xi\|_{H^1} + \int_s^t {C\over
    \sqrt{t-r}} \|(Df
  \circ u)(r) J_{s,r}\xi\|_{L^2}\,dr \\
  &\le \|\xi\|_{H^1} \Bigl(1 + \Sup_{t,\infty}^{n-1}(u) \int_s^t {C\over \sqrt{t-r}} e^{C(r-s)}\,dr \Bigr) \\
  &\le \|\xi\|_{H^1} \Bigl(1 + C\Sup_{t,\infty}^{n-1}(u) e^{C|t-s|}\sqrt{t-s}\Bigr)\\
  &\le \|\xi\|_{H^1} C\Sup_{t,\infty}^{n-1}(u) e^{C|t-s|}\;.
\end{equs}
(And similarly for the second bound.)
Regarding the $H^2$ norm of the Jacobian, we use the fact that there
is a constant $C$ such that the bound
\begin{equs}
  \|Df(u)v\|_{H^1} &\le C \bigl(\|u\|_{L^\infty}^{n-1} \|v\|_{H^1} +
  \|u\|_{L^\infty}^{n-2} \|\nabla u\|_{L^4} \|v\|_{L^4}\bigr) \\
  &\le C
  \|u\|_{L^\infty}^{n-2} \|u\|_{H^2} \|v\|_{H^1}
\end{equs}
holds. Hence we get similarly to before
\begin{equs}
  \|J_{s,t}\xi\|_{H^2} &\le \|\xi\|_{H^2} + \int_s^t {C\over
    \sqrt{t-r}} \|(Df \circ u)(r) J_{s,r}\xi\|_{H^1}\,dr \\
  &\le \|\xi\|_{H^2} + \int_s^t {C\over \sqrt{t-r}}
  \|u(r)\|_{L^\infty}^{n-1}
  \|u(r)\|_{H^2} \|J_{s,r}\xi\|_{H^1}\,dr \\
  &\le C \Sup_{t,\infty}^{4n}(u) 
  \bigl(\|u_0\|_{H^2} +\Sup_{t,2}(W_\Delta)\bigr) \|\xi\|_{H^2} \;,
\end{equs}
which is the requested bound. To obtain the last bound, one proceeds
identically except that one used $\|e^{L(t-s)} \xi\|_{H^2} \leq
C\|\xi\|_{H^1}/\sqrt{t-s}$.
\end{proof}

We now turn to the second variation. 
\begin{proposition}\label{prop:RDJ2}
  For any $u_0 \in H^2$, the second variation $J^{(2)}$ of the
  solution to \eqref{e:RD} satisfies
  \begin{align*}
    \|J_{s,t}^{(2)}\|_{H^2\otimes H^2 \rightarrow H^2}  \leq  C \Sup_{t,\infty}^{13n}(u) 
    \bigl(\|u_0\|_{H^2} +\Sup_{t,2}(W_\Delta)\bigr)^4
  \end{align*}
  for $0\leq s < t \leq 1$ with $\Sup_{t,\infty}$ defined just before
  Proposition \ref{prop:aprioriRD}.
\end{proposition}
\begin{proof}
Again using  Duhamel's formula, we have
\begin{equation}
  \label{eq:J2RD}
  J^{(2)}_{s,t}(\phi,\psi) = \int_s^t J_{r,t} D^2F(u_r)(J_{s,r}\phi,J_{s,r}\psi)dr\;.
\end{equation}
To control the $H^2$ norm we will need the following estimate:
\begin{align*}
  \|\nabla^2 D^2F(u)&(\phi,\psi)\|_{L^2} \leq C(1+
  \|u\|^{n-2}_\infty)(\|(\nabla^2 u)\phi \psi\|_{L^2}+ \|(\nabla
  u)^2\phi \psi\|_{L^2}  \\
  &\quad + \|(\nabla^2\phi)
  \psi\|_{L^2} +\|\phi (\nabla^2 \psi)\|_{L^2}+\|(\nabla \phi) (\nabla\psi)\|_{L^2}\\
  &\quad + \|(\nabla u)(\nabla \psi) \phi\|_{L^2} + \|(\nabla u)(\nabla \phi) \psi\|_{L^2}) \\
  &\leq C(1+
  \|u\|^{n-2}_\infty)(\|u\|_{H^2}\|\phi\|_{L^\infty} \|\psi\|_{L^\infty}+ \|u\|_{H^2}^2\|\phi\|_{L^\infty} \|\psi\|_{L^\infty} \\
  &\quad + \|\phi\|_{H^2}
  \|\psi\|_{L^\infty}+\|\phi\|_{L^\infty} \|\psi\|_{H^2}+\|\phi\|_{H^2} \|\psi\|_{H^2} \\
  &\quad +\|u\|_{H^2}\|\phi\|_{L^\infty} \|\psi\|_{H^2}+ \|u\|_{H^2}\|\phi\|_{H^2} \|\psi\|_{L^\infty})\\
 &\leq C(1+
  \|u\|^{n-2}_\infty)(1+\|u\|_{H^2}^2)\|\phi\|_{H^2} \|\psi\|_{H^2}\;.
  \end{align*}
  In this estimate, we have used repeatedly the fact that
  $\|v\|_{L^4} \leq C \|v\|_{H^1}$ and $\|v\|_{L^\infty} \leq C
  \|v\|_{H^2}$.  Using this estimate in \eqref{eq:J2RD}, we
  obtain
  \begin{equs}
    \|J^{(2)}_{s,t}(\phi,\psi)\|_{H^2} &\leq C \Sup_{t,\infty}^{n-2}(u) \int_s^t
    (1+\|u_r\|_{H^2}^2) \|J_{r,t}\|_{H^2\to H^2} \|J_{s,r}\phi\|_{H^2}
    \|J_{s,r}\psi\|_{H^2} dr \\
  &\leq C  \Sup_{t,\infty}^{17n}(u)(\|u_0\|_{H^2} +\Sup_{t,2}(W_\Delta))^5\|\psi\|_{H^2}\|\phi\|_{H^2}\;,
  \end{equs}
which completes the proof.
\end{proof}

  We now set the stage to prove the analogue of Theorem
  \ref{thm:asfEstimate} for equation \eqref{e:RD}. We begin by
  collecting a number for relevant results implied by the preceding
  calculations. We will work in $H^2$ since this will be the base
  space for what follows.
  \begin{proposition}\label{prop:RDass} Define $V(u) = \|u\|_{H^2}^{1/n}$.
  Then, for every $p>0$ there exists a constant $C_p$ and, for every $\eta > 0$ and $p>0$
   there exists a constant $C_{\eta,p}$ so that the bounds
    \begin{align*}
      \E \sup_{1/2\leq t \leq 1}\|u(t)\|_{H^2}^p &\leq C_p \\
      \E \sup_{1/2 \leq s< t \leq 1} \|J_{s,t}\|_{H^2 \rightarrow H^2}^p &\leq C_p \\
      \sup_{0\leq s < t \leq 1}\E \|J_{s,t}\|_{H^2 \rightarrow H^2}^p&\leq \exp(\eta pV(u) + pC_{\eta,p})\\
      \sup_{0\leq s < t \leq 1}\E \|J_{s,t}^{(2)}\|_{H^2\otimes H^2 \rightarrow H^2}^p&\leq \exp(\eta pV(u) + pC_{\eta,p})
    \end{align*}
hold for all $u_0 \in H^2$.
  \end{proposition}
  \begin{proof}
  The first two bounds are a consequence of Propositions~\ref{prop:RDUniformV},
  \ref{prop:RDJ} and \ref{prop:boundRD}.
  
In order to get the second two bounds, note that 
\begin{equs}
 \sup_{0\leq s < t \leq 1}\E \|J_{s,t}^{(2)}\|_{H^2\otimes H^2 \rightarrow H^2}^p &\le C_p (1+\|u_0\|_{H^2})^{(17n+5)p} \\
& = \exp \Bigl((17n+5)p \log(1+\|u_0\|_{H^2}) + \log C_p\Bigr)\;,
\end{equs}
as a consequence of Propositions~\ref{prop:RDJ2}, \ref{prop:RDUniformV} and \ref{prop:boundRD}.
A similar bound holds for $J$.
Since, for any positive $q,r,\eta, K$ there exists a
    $C_{q,r,\eta,K}$ so that $q\log(1+x) + \log(K) \leq \eta x^r +
    C_{q,r,\eta,K}$ for all $x \geq0$, the quoted bound holds.
  \end{proof}

  We assume from now on that the $g_k$ used in the definition of the forcing all
  belong to $H^4$. We now construct a particular subset of the $\AA_n$
  defined in Section~\ref{sec:Hormander} using on the highest degree
  nonlinear term. By doing so we obtain only constant vector fields,
  thus trivializing Assumption~\ref{ass:Hormander} in light of Lemma~\ref{lem:dense}. Setting
  $\AAt_1=\{g_1,\cdots,g_d\}$, we define recursively
  $\AAt_{k+1}=\AAt_k \cup \{ F_n(h_1,\cdots,h_n):h_j \in
  \AAt_k\}$ and $\AAt_\infty = \bigcup \AAt_k$. Notice
  that since $g_k \in H^4$, we know that all of the $\AAt_n \subset
  H^4$ since $H^4$  is a multiplicative algebra in our setting.

  \begin{proposition} If $\SPAN(\AAt_\infty)$ is dense in $H^2$ then
    given any $H^2$-orthogonal projection $\Pi$ onto a finite
    dimensional subspace, there exists $\theta>0$ such that Assumption~\ref{ass:Malliavin} holds with
    $U(u)=\Psi_0^\theta$.
  \end{proposition}
  \begin{proof}
    Proposition~\ref{prop:RDass} guarantees that all of the
    assumptions of Theorem \ref{theo:Malliavin} hold except
    Assumption~\ref{ass:Hormander}. However since by construction all
    of the vector fields in $\AAt_n$ are constant
    Assumption~\ref{ass:Hormander} clearly holds with $\Lambda$ a
    constant if $\Pi$ is an orthogonal projection onto a subspace of
    $\SPAN(\AAt_n)$. Lemma~\ref{lem:dense} furthermore shows that it actually holds for
    any finite rank orthogonal projection.
  \end{proof}

  Let $\Pi_M$ be the projection on the eigenfunctions of the Laplacian
  with eigenvalues smaller than $M^2$. We will now restrict ourselves to
  such a projection since it allows for easy verification of the pathwise
  smoothing/contracting properties needed for
  Assumption~\ref{ass:smoothing}. We have indeed the following bound:
  
  \begin{proposition}\label{prop:smoothingRD}
  Given any positive $\eta$, $r$ and $p$, there exists a $C_{\eta,r,p}$ so that the bound
    \begin{align*}
      \E \|J_{0,1}\Pi_M^\perp\|_{H^2 \rightarrow H^2}^p \leq   \exp(p\eta \|u_0\|_{H^2}^r - p\log(M)+pC_{\eta,r,p})
    \end{align*}
   holds  for all  $u_0 \in H^2$,  and all $M \in\N$.
  \end{proposition}
  \begin{proof}
    First observe that
    \begin{align*}
      \|J_{0,1}\Pi_M^\perp\|_{H^2 \rightarrow H^2}&\leq  \|J_{0,1}\|_{H^1 \rightarrow H^2}
      \|\Pi_M^\perp\|_{H^2 \rightarrow H^1} \leq M^{-1}\|J_{0,1}\|_{H^1 \rightarrow H^2}  \\
      &\leq C M^{-1}\Sup_{1,\infty}^{4n}(\|u_0\|_{H^2}  + \Sup_{1,2}(W_\Delta))\\
      &\leq C M^{-1} \Sup_{1,2}(W_\Delta)^{4n+1}(\|u_0\|_{H^2}  + 1)^{4n+1}
    \end{align*}
    Raising both sides to the power $p$, taking expectations, and using the fact that the law of $\Sup_{1,2}(W_\Delta)$
    has Gaussian tails,  we obtain
\begin{equ}
      \E \|J_{0,1}\Pi_M^\perp\|_{H^2 \rightarrow H^2}^p  \leq \exp(p(4n+1)\log(1+\|u_0\|) - p\log(M)+pC_p)\;.
\end{equ}
The claim now follows from the fact that, for any $\eta >0$ and $r>0$, there exist a $C_{\eta,r}$ with
    $(4n+1)\log(1+x) \leq \eta x^r + C_{\eta,r}$ for all $x \geq 0$.
 \end{proof}

\begin{theorem}\label{theo:MalliavinRD}
Let $\CP_t$ be the Markov semigroup on $H^2$
    generated by \eqref{e:RD}. If the linear span of $\AAt_\infty$ is dense
    in $H^2$ then, for every orthogonal finite rank projection $\Pi \colon H^2 \to H^2$,
    for every $p>0$, and for every $\alpha > 0$, there exists a constant $C(\alpha,p,\Pi)$ 
    such that the bound \eref{e:Malliavin} on the Malliavin matrix 
    holds with $U=1$.
\end{theorem}

  \begin{proof}
The result follows from Theorem~\ref{theo:Malliavin}. One can check that Assumption~\ref{ass:basic}
    holds with $\CH=H^2$, $a=0$, $\gamma_\star = \beta_\star = \infty$ since $H^\ell$ is a multiplicative algebra for
    every $\ell \ge 2$ (this is true because we restricted ourselves to dimension $m \le 3$). Since the most involved
 part is the assumption on the adjoint, part \ref{ass:DN}, we give the details for that one. One can verify that
 the adjoint of $DN(u)$ in $\CH$ acts on elements $v$ in $\CH^\infty$ as
\begin{equ}
DN^*(u) v = \Delta^{-2} f'(u) \Delta^2 v\;.
\end{equ}
(This is because $\CH$ is the Sobolev space $H^2$ and not the space $L^2$.)
The claim then follows from the fact that the multiplication by a smooth enough function is a bounded
operator in every Sobolev space $H^{\ell}$ with $\ell \in \R$.
    
Since Assumption~\ref{ass:Hormander} (with $\Lambda_\alpha$ a constant depending on $\Pi$) can be verified by using Lemma~\ref{lem:dense}, it remains to verify Assumption~\ref{ass:global} with $\Psi_0 = 1$.
This in turn is an immediate consequence of Proposition~\ref{prop:RDass}.
\end{proof}

Combining all of these results, we finally obtain the following result on the asymptotic strong Feller
property of a general reaction-diffusion equation:

  \begin{theorem}\label{theo:ASFRD} Let $\CP_t$ be the Markov semigroup on $H^2$
    generated by \eqref{e:RD} and let Assumptions~\ref{ass:Lyap} and
    \ref{ass:dim3} hold. If the linear span of $\AAt_\infty$ is dense
    in $H^2$ then, for any $\zeta > 0$, there exists a positive
    constant $C$ so that for every $u \in H^2$, and $\phi:H^2
    \rightarrow \R$ on has
    \begin{equ}[e:DPtRD]
      \|\DF \CP_{t}\phi(u)\|_{L^2 \rightarrow \R} \leq C \bigl( \|
      \phi\|_{L^\infty} + e^{-\zeta t} \sup_{v\in H^2} \|\DF\phi(v)\|_{H^2\to \R}\bigr)\;.
    \end{equ}
In particular, $\CP_t$ has the asymptotic strong Feller property in
    $H^2$.
  \end{theorem}
  
\begin{remark}
It is easy to infer from the \textit{a priori} bounds given in Propositions~\ref{prop:aprioriRD}, \ref{prop:RDUniformV}, 
\ref{prop:RDJ} and \ref{prop:RDJ2} that the assumptions  of our `all purpose' Theorem~\ref{theo:general} hold with 
$V(u) = \|u\|^\alpha$ for a sufficiently small exponent $\alpha$. However, the bound \eref{e:boundGeneral} is slightly weaker
than the bound \eref{e:DPtRD}. This shows that it may be worth under some circumstances to make the effort 
to apply  the more general Theorem~\ref{thm:asfEstimate}.
\end{remark}
  
\begin{remark}
As a corollary, we see that for the semigroup on $\CE$, one has
\begin{equ}
      \|\DF \CP_{t}\phi(u)\| \leq C \bigl( \|
      \phi\|_{L^\infty} + e^{-\zeta t} \|\DF\phi\|_{L^\infty}\bigr)\;,
\end{equ}
where all the derivatives a Fr\'echet derivatives of functions from $\CE$ to $\R$.

In particular, in space dimension $m=1$, the same bound is obtained in the space $H^1$ since
one then has $H^1\subset \CE$.
\end{remark}

\begin{proof}
  The result follows from Theorem~\ref{thm:asfEstimate}. Fix $\Pi =
  \Pi_M$, the projection onto the eigenfunctions of $\Delta$ with
  eigenvalues of modulus less than $M^2$. The constant $M$ is going to
  be determined later on.  Assumption~\ref{ass:Lyapunov} with $V(u) =
  \|u\|_{\CH}^{1/n}$ and $\eta' = 0$ follows immediately from
  Proposition~\ref{prop:RDUniformV}. Fix any $\bar p > 10$ and any
  positive $\eta < 1/\bar p$.  Assumption~\ref{ass:Jacobian} then
  follows from Proposition~\ref{prop:RDass}. It then follows from
  Proposition~\ref{prop:smoothingRD} that we can choose the value of
  $M$ in the definition of $\Pi$ sufficiently large so that
  Assumption~\ref{ass:smoothing} holds and such that $(C_\Pi - C_J)/2
  - \eta C_L > \zeta $.  Since, in view of
  Theorem~\ref{theo:MalliavinRD}, Assumption~\ref{ass:Malliavin} holds
  with $U=1$, we thus obtain from Theorem~\ref{thm:asfEstimate} the bound
    \begin{equ}[e:interboundRD]
      \|\DF \CP_{t}\phi(u)\|_{H^2 \rightarrow \R} \leq C e^{\eta \|u\|^{1/n}}\bigl( \| \phi\|_{L^\infty} + e^{-\zeta t} \sup_{v\in H^2}\|\DF\phi(v)\|_{H^2\to \R}\bigr)\;.
    \end{equ}
In order to obtain \eref{e:DPtRD}, we note that one has
\begin{equ}[e:uniformJac]
  \E \|J_{0,2}\|_{L^2\to H^2}^2 \le \E \|J_{0,1}\|_{L^2\to L^2}^2
  \|J_{1,2}\|_{L^2\to H^2}^2 \le C \E \|J_{1,2}\|_{L^2\to H^2}^2 \le
  C\;,
\end{equ}
where $C$ is a universal constant independent of the initial condition.
Here, we combined the bounds of Proposition~\ref{prop:RDJ} with Proposition~\ref{prop:boundRD} 
in order to obtain the last bound. We thus have
\begin{equs}
  \|\DF \CP_{t}\phi(u)&\|_{L^2 \rightarrow \R}  = \|\DF \CP_2 \CP_{t-2}\phi(u)\|_{L^2 \rightarrow H^2} \\
  &\le  \E \|\DF \CP_{t-2}\phi(u_2)\|_{H^2 \rightarrow \R} \|J_{0,2}\|_{L^2 \to H^2} \\
  &\le C \bigl( \|\phi\|_{L^\infty} + e^{-\zeta t}
\sup_v  \|\DF\phi(v)\|_{H^2\to \R}\bigr) \E \bigl(e^{2\eta \|u_2\|^{1/n}}
  \|J_{0,2}\|_{L^2 \to H^2} \bigr)\;,
\end{equs}
where we made use of \eref{e:interboundRD} to obtain the last inequality. The requested bound now follows
from \eref{e:uniformJac} and Proposition~\ref{prop:RDUniformV}.
\end{proof}

\subsection{Unique ergodicity of the stochastic Ginzburg-Landau equation}
\label{sec:accessibleGL}

In this section, we show under very weak conditions on the driving
noise that the stochastic real Ginzburg-Landau equation has a unique
invariant measure. Recall that this equation is given by
\begin{equ}[e:SGL]
  du(x,t) = \nu\d_x^2u(x,t)\, dt + \eta u(x,t)\, dt - u^3(x,t)\, dt +
  \sum_{j=1}^d g_j(x)\, dW_j(t)\;,
\end{equ}
where the spatial variable $x$ takes values on the circle $x \in S^1$
and the driving functions $g_j$ belong to $\CC^\infty(S^1,\R)$. The
two positive parameters $\nu$ and $\eta$ are assumed to be fixed
throughout this section.  This is a particularly simple case of the
type of equation considered above, so that Theorem~\ref{theo:ASFRD}
applies. The aim of this section is to show one possible technique for
obtaining the uniqueness of the invariant measure for such a parabolic
SPDE. It relies on Corollary~\ref{cor:uniqueIM} and yields:
\begin{theorem}\label{theo:mainGL}
Consider \eref{e:SGL} and suppose that 
\begin{claim}
\item[1.] there exists a linear combination $g$ of the $g_j$ that has
  only finitely many simple zeroes,
\item[2.] the smallest vector space containing all the $g_j$ and
  closed under the operation $(f,g,h) \mapsto fgh$ is dense in
  $H^1(S^1)$.
\end{claim}
Then \eref{e:SGL} has exactly one invariant probability measure.
\end{theorem}

\begin{remark}
The second assumption is satisfied for example if $d \ge 3$ and $g_1(x) = 1$, $g_2(x) = \sin x$ and
$g_3(x) = \cos x$.
\end{remark}

\begin{remark}
  We believe that the first condition in Theorem~\ref{theo:mainGL} is
  not needed, since in finite dimensions  such a Lie bracket
  condition implies global controllability for polynomial systems of
  odd degree. See for example  \cite{Jurd97}.
\end{remark}

\begin{remark}
Actually, we could have relaxed the regularity assumption on the $g_j$'s.
If we choose $\CH = H^1$, $\gamma_\star = 2\eps - 1$, $a = 1-\eps$,
and $\beta_\star = 1$, we can check that Assumption~\ref{ass:basic}
is satisfied as soon as $g_j \in H^{1+4\eps}$. Furthermore, in this case,
all the relevant Lie brackets for assumption~2 in
Theorem~\ref{theo:mainGL} are admissible, so that its conclusion still holds.
\end{remark}

Looking at Corollary~\ref{cor:uniqueIM}, the two main ingredients
needed to prove Theorem~\ref{theo:mainGL} are the establishment of the
estimate in \eref{e:SF} and the needed form of irreducibility. The
first will follows almost instantly from the second assumption of
Theorem~\ref{theo:mainGL} which ensures that $\SPAN(\AAt_\infty)$ is
dense in $H^1$. The irreducibility is given by the following proposition
whose proof is postponed to the end of this section.

\begin{proposition}\label{prop:RDIred} Consider  \eqref{e:SGL}
  under the second condition in Theorem~\ref{theo:mainGL}.  Then there
  exists a positive $K$ so that for any $\epsilon >0$ there is a $v$
  with $\|v\|_{H^1} \leq K$ and a $T>0$ so that
  $\CP_T(u_0,\CB_\epsilon(v)) >0$ for all $u_0 \in H^1$. Here
  $\CB_\epsilon(v)$ is the $\epsilon$ ball in the $H^1$--norm.
\end{proposition}

\begin{proof}[of Theorem~\ref{theo:mainGL}]
   The existence of an invariant probability measure for \eref{e:SGL}
  is standard, see for example \cite{CerrRD}. Furthermore, since we
  are working in space dimension $1$, $H^1$ is already a
  multiplicative algebra and one can retrace the proof of
  Theorem~\ref{theo:ASFRD} for $\CH = H^1$.  This shows that
  assumption~2. implies that the semigroup generated by \eref{e:SGL}
  satisfies \eref{e:SF} on the Hilbert space $\CH = H^1(S^1)$. It
  therefore remains to show that assumption~1. implies the assumption
  of Corollary~\ref{cor:uniqueIM}.
\end{proof}

In fact we have established much more than just uniqueness of the
invariant measure. We now use the results from \cite{HaiMat08:??} to
establish a spectral gap. For any Fr\'echet differentiable functions
from $\phi:H^1 \rightarrow \R$ define the norm $\|\phi\|_{\text{Lip}}=
\sup_u \bigl(|\phi(u)| + \|\DF\phi(u)\|_{H^1 \to \R}\bigr)$. In turn we
define a metric on probability measures $\mu, \nu$ on $H^1$ by
$d(\mu,\nu) = \sup\{ \int \phi d\mu - \int \phi d\nu :
\|\phi\|_{\text{Lip}} \leq 1 \}$. Combining \eqref{e:DPtRD},
Proposition~\ref{prop:RDIred} and \cite[Theorem~2.5]{HaiMat08:??}
yields the following corollary to Theorem~\ref{theo:mainGL}.
  \begin{corollary}
     Under the assumption of Theorem~\ref{theo:mainGL}, there exist
     positive constants $C$ and $\gamma$ so that  $d(\CP_t^* \mu
     ,\CP_t^* \nu) \leq C e^{-\gamma t} d(\mu,\nu)$ for any two
     probability measures $\mu$ and $\nu$ on $H^1$ and $t \geq 1$.
  \end{corollary}

\begin{proof}[of Proposition \ref{prop:RDIred}]
  Fix an arbitrary initial condition $u_0$ and some $\eps > 0$.  Our
  aim is to find a target $v$, bounded controls $V_j(t)$, and a
  terminal time $T>0$, so that the solution to the controlled problem
\begin{equ}[e:controlGL]
  \d_t u(x,t) = \nu\d_x^2u(x,t) + \eta u(x,t) - u^3(x,t) +
  f(x,t)\;,\quad f(x,t) \eqdef \sum_{j=1}^d g_j(x)\, V_j(t)\;,
\end{equ}
satisfies $\|u(T) - v\|_{H^1} \le \eps$.  Furthermore, we want to be
able to choose $v$ such that $\|v\|_{H^1} \le K$ for some constant $K$
independent of $\eps$.  The claim on the topological supports of
transition probabilities then follows immediately from the fact that
the It\^o map $(u_0,W) \mapsto u_t$ is continuous in the second
argument in our case.
 
 The idea is to choose $f$ of the form
\begin{equ}
f(x,t) = \left\{\begin{array}{cl} \eps^{-\gamma} g(x) & \text{for $1 \le t \le 2$,} \\ 0 & \text{otherwise,} \end{array}\right.
\end{equ}
and to set $T = 3$. We furthermore set $v$ to be the solution at time $1$ for the uncontrolled
equation (that is \eref{e:controlGL}  with $f = 0$) with an initial condition $v_0$ satisfying 
\begin{equ}[e:stationaryv0]
 \nu\d_x^2 v_0(x) + \eta v_0(x) - v_0^3(x) + \eps^{-\gamma}g(x) = 0\;,
\end{equ}
for some exponent $\gamma > 0$ to be determined. Such a $v_0$ always exists since the coercive ``energy functional''
\begin{equ}
E(v) = \int_{S^1} \Bigl({\nu\over 2}|\d_xv(x)|^2 - {\eta \over 2}|v(x)|^2 + {1\over 4}|v(x)|^4 - \eps^{-\gamma}g(x) v(x)\Bigr)\,dx
\end{equ}
has at least one critical point.
Even though $v_0$ is in general very large (see however Lemma~\ref{lem:boundv0} below), 
it follows from \eref{e:bounddV} that the target 
$v$ constructed in this way is bounded independently of $\eps$. 

The remaining ingredient of the proof are Lemmas~\ref{lem:bounddissip} and \ref{lem:boundv0} below. 
To show that this is sufficient, note first that \eref{e:bounddV} implies the existence of a constant $C$
such that $\|u(1)\|_{L^2} \le C$ independently of $u_0$. It then follows from
Lemmas~\ref{lem:bounddissip} and \ref{lem:boundv0} that (choosing for example $\beta = {\gamma/14}$) there exists
a constant $C$ such that one has the bound
\begin{equ}
\|u(2) - v_0\|_{L^2} \le C \eps^{\gamma \over 6}\;.
\end{equ}
Since the uncontrolled equation expands at rate at most $\eta$, this immediately yields $\|u(T) - v\|_{L^2} \le C \eps^{\gamma \over 6}$. On the other hand, we know from Proposition~\ref{prop:aprioriRD}
that there exists a constant $C$ such that $\|u(T) - v\|_{H^2} \le C$, so that
\begin{equ}
\|u(T) - v\|_{H^1} \le \bigl(\|u(T) - v\|_{L^2}\|u(T) - v\|_{H^2}\bigr)^{1/2} \le C \eps^{\gamma /12}\;,
\end{equ}
and the claim follows by choosing $\gamma > 12$.
\end{proof}

\begin{lemma}\label{lem:boundv0}
There exists a constant $C_v$ independent of $\eps < 1$ such that the bound
$\|v_0(x)\|_{L^\infty} \le C_v \eps^{-\gamma/3}$ holds.
\end{lemma}

\begin{proof}
It follows immediately from \eref{e:stationaryv0}, using the fact that $\d_x^2 v_0 \le 0$ at the maximum 
and $\d_x^2 v_0 \ge 0$ at the minimum.
\end{proof}

\begin{lemma}\label{lem:bounddissip}
For every exponent $\beta \in [0,\gamma/4]$ there exists a constant $C$ such that the bound
\begin{equ}
\int_{S^1} (u-v_0)(u^3 - v_0^3)\, dx \ge C\eps^{-2\beta}\int_{S^1} (u-v_0)^2\,dx - C\eps^{{\gamma - 13\beta\over 3}}
\end{equ}
holds for every $\eps \le 1$ and every $u \in L^2(S^1)$.
\end{lemma}

\begin{proof}
The proof is based on the fact that since $g$ has only isolated zeroes,
the function $v_0$ necessarily has the property that
it is large at most points. More precisely, consider some exponent $\beta \in [0,\gamma/3]$ and
define the set $A = \{x \in S^1\,:\, |v_0(x)| > \eps^{-\beta}\}$. We claim that there  then exists a constant
$C$ such that the Lebesgue measure of $A$ is bounded by $|A| \le C \eps^\alpha$
for $\alpha = \min\{\gamma - 3\beta, {\gamma - \beta \over 3}\}$.
Indeed, consider the set $\tilde A$ of points such that $|g(x)| \le 2\eps^\alpha$. Since $g$ is assumed
to be smooth and have simple zeroes, $|\tilde A| \le C \eps^\alpha$ and the complement of $\tilde A$ consists
of finitely many intervals on which $g$ has a definite sign. 

Consider one such interval $I$ on which $g(x) > 9\eps^\alpha$,
so that the definition of $v_0$ yields the estimate $v_0'' < -9\eps^{\alpha - \gamma} - v_0 + v_0^3$.
It follows that, for every $x \in I$, one either has $v_0''(x) < - \eps^{\alpha - \gamma}$, or one has
$v_0(x) > 2\eps^{\alpha - \gamma \over 3}\ge 2\eps^{-\beta}$ (since we set $\alpha \le \gamma - 3\beta$).
We conclude that $I\cap A$ consists of at most two intervals and that $v_0(x) > \eps^{\alpha - \gamma \over 3}$ 
for every $x \in I\cap A$, so that $|I \cap A| \le C\eps^{\gamma - \alpha - \beta \over 2}$ and the bound follows.
(The same reasoning but with opposite signs applies to those intervals on which $g(x) < -9\eps^\alpha$.)

This yields the sequence of bounds
\begin{equs}
2\int_{S^1} (u-v_0)&(u^3 - v_0^3)\, dx \ge  \int_{S^1} (u-v_0)^2 (u^2 + v_0^2)\,dx \\
&\ge {\eps^{-2\beta}} \int_{A} (u-v_0)^2\, dx
+ \int_{A^c} (u-v_0)^2 (u^2 + v_0^2)\,dx \\
& \ge {\eps^{-2\beta}} \int_{A} (u-v_0)^2\, dx
+ {1\over 4} \int_{A^c} (u-v_0)^4\,dx \\
& \ge \eps^{-2\beta} \int_{A} (u-v_0)^2\, dx
+ {1\over 4|A^c|} \Bigl(\int_{A^c} (u-v_0)^2\,dx\Bigr)^2 \\
&\ge \eps^{-2\beta} \int_{A} (u-v_0)^2\, dx
+ {C\over \eps^\alpha} \Bigl(\eps^{\alpha-2\beta} \int_{A^c} (u-v_0)^2\,dx - \eps^{2\alpha-4\beta}\Bigr) \\
&\ge C\eps^{-2\beta}\int_{S^1} (u-v_0)^2\,dx - C\eps^{\alpha-4\beta}\;,
\end{equs}
which is the required estimate.
\end{proof}

  \bibliographystyle{Martin}

\bibliography{./refs}

\end{document}